\documentclass[a4paper,10pt]{article}

\usepackage{amsmath,amssymb,amsthm,amscd}
\usepackage[mathscr]{eucal}
\usepackage{multirow}
\usepackage{hhline}
\usepackage[all]{xy}

\makeatletter
 
  \@addtoreset{equation}{subsection}
\makeatother

\theoremstyle{plain}

\newcommand{\plim}{\varprojlim}
\newcommand{\mcal}{\mathcal}
\newcommand{\mbf}{\mathbf}

\newcommand{\mfrak}{\mathfrak}
\newcommand{\mbb}{\mathbb}
\newcommand{\mrm}{\mathrm}
\newcommand{\mf}{\mathfrak}
\newcommand{\mfS}{\mathfrak{S}}
\newcommand{\vphi}{\varphi}
\newcommand{\mfM}{\mathfrak{M}}
\newcommand{\mfN}{\mathfrak{N}}
\newcommand{\mfK}{\mathfrak{K}}

\newcommand{\wh}{\widehat}
\newcommand{\whR}{\widehat{\mathcal{R}}}
\newcommand{\whRi}{\widehat{\mathcal{R}}_{\infty}}
\newcommand{\Max}{\mathrm{Max}}
\newcommand{\mMax}{\mathrm{max}}
\newcommand{\mfSur}{\mfS^{\mrm{ur}}}
\newcommand{\iotamax}{\iota_{\mrm{max}}}
\newcommand{\Min}{\mathrm{Min}}
\newcommand{\mMin}{\mathrm{min}}
\newcommand{\iotamin}{\iota_{\mrm{min}}}

\newtheorem{theorem}{Theorem}[section]
\newtheorem{corollary}[theorem]{Corollary}
\newtheorem{lemma}[theorem]{Lemma}
\newtheorem{question}[theorem]{Question}
\newtheorem{proposition}[theorem]{Proposition}

\theoremstyle{definition}
\newtheorem{definition}[theorem]{Definition}
\newtheorem{remark}[theorem]{Remark}

\newtheorem{example}[theorem]{Example}

\newcommand{\e}{\varepsilon}

\newcommand{\cO}{\mathcal{O}}
\newcommand{\E}{\mathcal{E}}
\newcommand{\Eur}{\mathcal{E}^{\mathrm{ur}}}
\newcommand{\cOdual}{\mathcal{O}^{\vee}}
\newcommand{\cOur}{\mathcal{O}^{\mathrm{ur}}}
\newcommand{\cOurdual}{\mathcal{O}^{\mathrm{ur},\vee}}
\newcommand{\cOG}{\mathcal{O}_{\hat{G}}}
\newcommand{\EG}{\mathcal{E}_{\hat{G}}}

\newtheorem*{acknowledgements}{Acknowledgements}

\title{Torsion representations arising from $(\vphi,\hat{G})$-modules}

\author{Yoshiyasu Ozeki\footnote{
Graduate School of Mathematics, 
Kyushu University, Fukuoka 819-0395, Japan.
\endgraf
e-mail: {\tt y-ozeki@math.kyushu-u.ac.jp}
\endgraf
Partly supported by the Grant-in-Aid for Research Activity Start-up, 
The Ministry of Education, Culture, Sports, Science and Technology, Japan.}}

\date{}
\begin{document}
\maketitle

\begin{abstract}

The notion of a $(\vphi,\hat{G})$-module is defined by Tong Liu in 2010
to classify lattices in semi-stable representations.
In this paper, 
we study torsion $(\vphi,\hat{G})$-modules, 
and torsion $p$-adic representations associated with them, 
including the case where $p=2$.
First we prove that the category of torsion $p$-adic representations 
arising from torsion $(\vphi,\hat{G})$-modules is an abelian category.
Secondly, we construct a maximal (minimal) theory for $(\vphi,\hat{G})$-modules
by using the theory of \'etale $(\vphi, \hat{G})$-modules,
essentially proved by Xavier Caruso, which is an analogue of 
Fontaine's theory of \'etale $(\vphi,\Gamma)$-modules.
Non-isomorphic two maximal (minimal) objects give non-isomorphic two torsion $p$-adic representations.
\end{abstract}

\tableofcontents


\section{Introduction}
The notion of 
a $(\vphi,\hat{G})$-module was introduced by T. Liu in \cite{Li3}
to classify 
lattices in semi-stable representations.
In this paper,
we give various properties of torsion $(\vphi,\hat{G})$-modules 
such as the Cartier duality theorem.
Furthermore,
we study the category of torsion representations arising from 
torsion $(\vphi,\hat{G})$-modules.
Let $G$ be the absolute Galois group of a complete discrete valuation field $K$
of mixed characteristic $(0,p)$ with perfect residue field. 
Fix $r\in \{0, 1,2,\dots ,\infty \}$.
Our study is motivated by the following question:
\begin{quotation}
Is any torsion $\mbb{Z}_p$-representation of $G$
a torsion semi-stable representation with Hodge-Tate weights in $[0,r]$?
\end{quotation}
Here, a  torsion $\mbb{Z}_p$-representation of $G$ is said to be {\it torsion semi-stable 
with Hodge-Tate weights in $[0,r]$} 
if it can be written as a quotient of two lattices 
in a semi-stable $p$-adic representation  of $G$ with Hodge-Tate weights in $[0,r]$.
It is known that the above question does not have an affirmative answer 
if $r<\infty$ and thus  it makes sense only if $r=\infty$.
We propose 
an approach to this question by using $(\vphi,\hat{G})$-modules
which gives a description of 
a (torsion) semi-stable $p$-adic representation with Hodge-Tate weights in $[0,r]$.
The theory of Breuil modules also gives a description of 
these representations in terms of linear algebra
(cf.\ \cite{Li2}),
however, for technical reasons, we have to assume $r<p-1$ 
when we use this theory for integral or torsion representations. 
On the other hand, 
there is no restriction on $r$ in 
the theory of $(\vphi, \hat{G})$-modules.
This is the main reason why we focus on $(\vphi, \hat{G})$-modules.

Let $\mrm{Rep}_{\mrm{tor}}(G)$ be the category of 
finite torsion $\mathbb{Z}_p$-representations.
Let  $\mrm{Rep}^{\mrm{st}}_{\mrm{tor}}(G)$ 
be the category of 
torsion semi-stable representations.
We denote by $\mrm{Mod}^{r,\hat{G}}_{/\mfS_{\infty}}$ 
the category of torsion $(\vphi,\hat{G})$-modules 
of height $r$ and 
$\hat{T}\colon \mrm{Mod}^{r,\hat{G}}_{/\mfS_{\infty}} \to \mrm{Rep}_{\mrm{tor}}(G)$
the associated functor (see Section 2).
Let $\mrm{Rep}^{\hat{G}}_{\mrm{tor}}(G)$
be the category of 
torsion representations arising from
torsion $(\vphi,\hat{G})$-modules, that is, the essential image of 
$\hat{T}$ on $\mrm{Mod}^{\infty,\hat{G}}_{/\mfS_{\infty}}$.
Then the inclusions
\[
\mrm{Rep}^{\mrm{st}}_{\mrm{tor}}(G)\subset \mrm{Rep}^{\hat{G}}_{\mrm{tor}}(G) 
\subset \mrm{Rep}_{\mrm{tor}}(G)
\]
are known (cf.\ \cite{CL2}, Theorem 3.1.3). 
Since our interest is related with the equality of categories
$\mrm{Rep}^{\mrm{st}}_{\mrm{tor}}(G)$ and $\mrm{Rep}_{\mrm{tor}}(G)$,
we want to know a difference (if it exists) between above three categories.
The following is the first main result of this paper:
\begin{theorem}
\label{MThm1}
The category $\mrm{Rep}^{\hat{G}}_{\mrm{tor}}(G)$ 
is an abelian full subcategory of 
$\mrm{Rep}_{\mrm{tor}}(G)$ which is stable under taking a subquotient, $\oplus,\ \otimes$ and a dual.
\end{theorem}

\noindent
To show the category $\mrm{Rep}^{\hat{G}}_{\mrm{tor}}(G)$ 
is abelian,
we give two different proofs.
The first one uses a deep relation, proved by T. Liu, between $(\vphi,\hat{G})$-modules
and representations associated with them (cf.\ Lemma \ref{rel2}). 
The second proof is based on a result on {\it maximal $($minimal$)$ objects} of 
$(\vphi, \hat{G})$-modules.
In general, 
the category $\mrm{Mod}^{r,\hat{G}}_{/\mfS_{\infty}}$ 
is not abelian and 
$\hat{T}\colon \mrm{Mod}^{r,\hat{G}}_{/\mfS_{\infty}} \to \mrm{Rep}_{\mrm{tor}}(G)$
is not fully faithful.
The theory of maximal (minimal) objects allows us to avoid such a situation:
Denote by $\Max^{r,\hat{G}}_{/\mfS_{\infty}}$ 
the full subcategory of $\mrm{Mod}^{r,\hat{G}}_{/\mfS_{\infty}}$
whose objects are maximal.
Then we obtain a functor
$\Max^r\colon \mrm{Mod}^{r,\hat{G}}_{/\mfS_{\infty}}\to \Max^{r,\hat{G}}_{/\mfS_{\infty}}$
which is a retraction of a natural inclusion 
$\Max^{r,\hat{G}}_{/\mfS_{\infty}}\hookrightarrow \mrm{Mod}^{r,\hat{G}}_{/\mfS_{\infty}}$
and commutes with $\hat{T}$.
We prove 
\begin{theorem}
\label{MThm3}
The category $\Max^{r,\hat{G}}_{/\mfS_{\infty}}$
is abelian and artinian.
Furthermore,
the restriction of $\hat{T}$  on $\Max^{r,\hat{G}}_{/\mfS_{\infty}}$
is exact and fully faithful, and 
its essential image is stable under taking a subquotient.
\end{theorem}
\noindent
In particular,  we immediately  understand that 
the category $\mrm{Rep}^{\hat{G}}_{\mrm{tor}}(G)$ 
is abelian.
If $r<\infty$,
we can define 
the full subcategory $\Min^{r,\hat{G}}_{/\mfS_{\infty}}$ of $\mrm{Mod}^{r,\hat{G}}_{/\mfS_{\infty}}$
whose objects are minimal
and the functor 
$\Min^r\colon \mrm{Mod}^{r,\hat{G}}_{/\mfS_{\infty}}\to \Min^{r,\hat{G}}_{/\mfS_{\infty}}$;
they satisfy analogous properties as those stated
in Theorem \ref{MThm3}.
Furthermore, the Cartier duality theorem gives a connection between maximal objects and minimal objects
(cf.\ Proposition \ref{maxmin}).
Maximal (minimal) objects are first defined for finite flat group schemes by M. Raynaud \cite{Ra}.
X. Caruso and T. Liu generalized Raynaud's theory, with respect to finite flat group schemes
killed by a power of $p$, to torsion Kisin  modules \cite{CL1},  
whose representations are defined on $G_{\infty}$.
Here $G_{\infty}=\mrm{Gal}(\bar{K}/K_{\infty})$ and 
$K_{\infty}=\cup_{n\ge 0}K({\pi_n})$, $\pi_0=\pi$ 
a uniformizer of $K$, $\pi_{n+1}^p=\pi_n$. 
Furthermore,
a categorical interpretation of maximal (minimal) objects is given in \cite{Ca1}. 
Our theorem described as above is an extended result of \cite{CL1} in a certain sense.
In the case where $r=\infty$, 
we obtain the following:

\begin{corollary}
The functor $\hat{T}\colon \mrm{Mod}^{\infty,\hat{G}}_{/\mfS_{\infty}} \to \mrm{Rep}_{\mrm{tor}}(G)$
induces the equivalence of abelian categories
between  the category $\Max^{\infty,\hat{G}}_{/\mfS_{\infty}}$ of 
maximal torsion $(\vphi,\hat{G})$-modules of finite height
and the category $\mrm{Rep}^{\hat{G}}_{\mrm{tor}}(G)$
of torsion $\mbb{Z}_p$-representations of $G$ arising from $(\vphi,\hat{G})$-modules.
\end{corollary}

\noindent
To define maximal (minimal) objects of torsion $(\vphi,\hat{G})$-modules,
we introduce an {\it \'etale $(\vphi,\hat{G})$-module}, which is an \'etale 
$\vphi$-module (in the sense of J.-M. Fontaine \cite{Fo}) equipped with certain Galois action.
Arguments of the theory of $(\vphi,\tau)$-modules of \cite{Ca} give us the fact that 
the category of torsion \'etale $(\vphi,\hat{G})$-modules
is equivalent to $\mrm{Rep}_{\mrm{tor}}(G)$.

Now denote by $e$ the absolute ramification index of $K$.
If $er<p-1$,
then all torsion $(\vphi,\hat{G})$-modules of height $r$ are automatically maximal and minimal. 
Therefore, we have
\begin{corollary}[= Corollary \ref{er}]
Suppose $er<p-1$. Then 
the category $\mrm{Mod}^{r,\hat{G}}_{/\mfS_{\infty}}$ is abelian and artinian.
Furthermore, the functor
$\hat{T}\colon \mrm{Mod}^{r,\hat{G}}_{/\mfS_{\infty}}\to \mrm{Rep}_{\mrm{tor}}(G)$
is exact and fully faithful, and its essential image is stable under taking a subquotient. 
\end{corollary}

\noindent
Such a result on torsion Breuil modules has been 
proved by X. Caruso (cf.\ \cite{Ca-1}, Th\'eor\`eme 1.0.4).\\
We hope our study will be useful to solve the question described in the start of this paper, 
(cf.\ Section 5.7).

Now we describe an organization of this paper.
In Section 2, we recall some results 
on Kisin modules and $(\vphi,\hat{G})$-modules and 
prove some fundamental properties of them which are often used in this paper. 
In Section 3, we prove the Cartier duality theorem of  $(\vphi,\hat{G})$-modules.
In Section 4, we prove Theorem \ref{MThm1}.
Finally in Section 5,
we give a theory of \'etale $(\vphi,\hat{G})$-modules 
and define maximal (minimal) objects for $(\vphi,\hat{G})$-modules, 
and prove 
Theorem \ref{MThm3}.\\

\noindent
{\bf Convention.}
For any $\mbb{Z}$-module $M$, we always use $M_n$ to denote $M/p^nM$ for a positive integer $n$
and $M_{\infty}=M\otimes_{\mbb{Z}_p}\mbb{Q}_p/\mbb{Z}_p$.
We reserve $\vphi$ to represent various Frobenius structures and 
$\vphi_M$ will denote the Frobenius on $M$. 
However, we often drop the subscript if no confusion arises.  
All representations and actions are assumed to be continuous.

\begin{acknowledgements}
The author wants to thank Shin Hattori
who gave him useful advice and comments 
throughout this paper, in particular, Section 5.6.
This work is supported by the Grant-in-Aid for Young Scientists Start-up.
\end{acknowledgements}


\section{Preliminaries}

In this section, 
we recall some notions and results which will be used throughout this paper.

\subsection{Notation}
Let $k$ be a perfect field of 
characteristic $p\ge 2$,
$W(k)$ its ring of Witt vectors, 
$K_0=W(k)[1/p]$, $K$ a finite totally 
ramified extension of $K_0$,
$\bar{K}$ a fixed algebraic closure of $K$
and $G=\mrm{Gal}(\bar K/K)$.
Throughout this paper,
we fix a uniformizer $\pi\in K$ 
and denote by $E(u)$ its 
Eisenstein polynomial over $K_0$.
Let $\mfS=W(k)[\![u]\!]$ equipped 
with a Frobenius endomorphism
$\varphi$ via $u\mapsto u^p$ and 
the natural Frobenius on $W(k)$. 

Let $R=\plim \cO_{\bar K}/p$ 
where $\cO_{\bar K}$ is 
the integer ring of $\bar K$
and the transition maps are 
given by the $p$-th power map.
By the universal property of 
the ring of Witt vectors $W(R)$ of $R$,
there exists a unique surjective projection
map $\theta\colon W(R)\to \wh{\cO}_{\bar K}$
which lifts the projection $R\to \cO_{\bar K}/p$
onto the first factor in the inverse limit,
where $\wh{\cO}_{\bar K}$ is the 
$p$-adic completion of $\cO_{\bar K}$.
For any integer $n\ge 0$,
let $\pi_n\in \bar K$ be 
a $p^n$-th root of $\pi$ such that 
$\pi^p_{n+1}=\pi_n$ and write 
\underbar{$\pi$} $=(\pi_n)_{n\ge 0}\in R$. 
Let $[$\underbar{$\pi$}$]\in W(R)$ 
be the Teichm\"uller
representative of \underbar{$\pi$}.
We embed the $W(k)$-algebra $W(k)[u]$ into $W(R)$
via the map $u\mapsto [$\underbar{$\pi$}$]$.
This embedding extends to an 
embedding $\mfS\hookrightarrow W(R)$, 
which is compatible with Frobenius endomorphisms.

Let $\cO$ be the $p$-adic completion 
of $\mfS[1/u]$, which is 
is a discrete valuation ring 
with uniformizer $p$ and residue field $k(\!(u)\!)$.
Denote by $\E$ the field of fractions of $\cO$.
The inclusion $\mfS\hookrightarrow W(R)$
extends to inclusions $\cO\hookrightarrow W(\mrm{Fr}R)$
and $\E\hookrightarrow W(\mrm{Fr}R)[1/p]$.
Here $\mrm{Fr}R$ is the field of fractions of $R$.
It is not difficult to see that $\mrm{Fr}R$ is algebraically closed.
We denote by $\E^{\mrm{ur}}$ the maximal unramified 
field extension of $\E$ in $W(\mrm{Fr}R)[1/p]$
and $\cOur$ its integer ring.
Let $\wh{\E^{\mrm{ur}}}$ be the 
$p$-adic completion of $\E^{\mrm{ur}}$ and 
$\wh{\cOur}$ its integer ring.
The ring  $\wh{\E^{\mrm{ur}}}$ 
(resp.\ $\wh{\cOur}$)
is equal to the closure of 
$\E^{\mrm{ur}}$ in $W(\mrm{Fr}R)[1/p]$
(resp. the closure of $\cOur$ in $W(\mrm{Fr}R)$). 
Put $\mfS^{\mrm{ur}}=\wh{\cOur}\cap W(R)$.
We regard all these rings as subrings of $W(\mrm{Fr}R)[1/p]$.

Let $K_{\infty}=\cup_{n\ge 0}K(\pi_n)$ and 
$G_{\infty}=\mrm{Gal}(\bar K/K_{\infty})$.
Then $G_{\infty}$ acts on $\mfS^{\mrm{ur}}$ 
and $\E^{\mrm{ur}}$ continuously and fixes the subring $\mfS\subset W(R)$.
We denote by 
$\mrm{Rep}_{\mbb{Z}_p}(G_{\infty})$ 
(resp.\ $\mrm{Rep}_{\mbb{Q}_p}(G_{\infty})$)
the category of 
continuous $\mbb{Z}_p$-linear representations 
of $G_{\infty}$
on finite $\mbb{Z}_p$-modules
(resp.\ the category of 
continuous $\mbb{Q}_p$-linear representations 
of $G_{\infty}$
on finite dimensional $\mbb{Q}_p$-vector spaces).
We denote by 
$\mrm{Rep}_{\mrm{tor}}(G_{\infty})$
(resp.\ $\mrm{Rep}_{\mrm{fr}}(G_{\infty})$)
the full subcategory of 
$\mrm{Rep}_{\mbb{Z}_p}(G_{\infty})$
consisting of 
$\mbb{Z}_p$-modules killed by some power of $p$
(resp.\ finite free $\mbb{Z}_p$-modules).  
Similarly,
we define categories 
$\mrm{Rep}_{\mbb{Z}_p}(G), \mrm{Rep}_{\mbb{Q}_p}(G_{\infty}), \mrm{Rep}_{\mrm{tor}}(G)$
and $\mrm{Rep}_{\mrm{fr}}(G)$
by replacing $G_{\infty}$ with $G$.

\subsection{\'Etale $\vphi$-modules}

In this subsection,
We recall the theory of Fontaine's \'etale  $\vphi$-modules.
For more precise information,
see \cite{Fo} A 1.2.

A finite $\cO$-module $M$ 
is called {\it \'etale}
if $M$ is equipped with 
a $\vphi$-semi-linear map 
$\vphi_M\colon M\to M$ such that 
$\vphi^{\ast}_M$ is an isomorphism, 
where $\vphi^{\ast}_M$ is 
the $\cO$-linearization
$1\otimes \vphi_M\colon \cO
\otimes_{\vphi, \cO} M\to M$
of $\vphi_M$.
A finite $\E$-vector space $M$ 
is called {\it \'etale}
if $M$ is equipped with 
a $\vphi$-semi-linear map 
$\vphi_M\colon M\to M$ and
there exists a $\vphi$-stable $\cO$-lattice $L$  of $M$
that is an \'etale $\cO$-module.
We denote by ${}'\mbf{\Phi M}_{/\cO}$ (resp.\ $\mbf{\Phi M}_{/\E}$) 
the category of finite \'etale $\cO$-modules
(resp.\ the category of finite \'etale $\E$-modules)
with the obvious morphisms.
Note that the extension $K_{\infty}/K$ is a strictly 
APF extension in the sense of \cite{Wi}
and thus $G_{\infty}$ is 
naturally isomorphic to the absolute Galois group of $k(\!(u)\!)$
by the theory of norm fields.
Combining this fact and Fontaine's theory 
in \cite{Fo}, A 1.2.6, we have that
functors
\[
\mcal{T}_{\ast}\colon {}'\mbf{\Phi M}_{/\cO}\to 
\mrm{Rep}_{\mbb{Z}_p}(G_{\infty}),\quad M\mapsto 
(\wh{\cOur}\otimes_{\cO} M)^{\vphi=1} 
\]
and 
\[
\mcal{T}_{\ast}\colon \mbf{\Phi M}_{/\E}\to 
\mrm{Rep}_{\mbb{Q}_p}(G_{\infty}),\quad M\mapsto 
(\wh{\Eur}\otimes_{\E} M)^{\vphi=1} 
\]
are equivalences of abelian categories
and there exist  natural 
$\wh{\cOur}$-linear 
isomorphisms which are compatible with $\vphi$-structures
and $G_{\infty}$-actions:
\begin{equation}
\label{Fon1}
\wh{\cOur}\otimes_{\mbb{Z}_p}\mcal{T}_{\ast}(M) 
\overset{\sim}{\longrightarrow}
\wh{\cOur}\otimes_{\cO} M\quad  \mrm{for}\ M\in {}'\mbf{\Phi M}_{/\cO}   
\end{equation}
and \begin{equation}
\label{Fon1'}
\wh{\Eur}\otimes_{\mbb{Q}_p}\mcal{T}_{\ast}(M) 
\overset{\sim}{\longrightarrow}
\wh{\Eur}\otimes_{\E} M\quad \mrm{for}\ M\in \mbf{\Phi M}_{/\E}. 
\end{equation}
\noindent
On the other hand,
define functors
\[
\mcal{M}_{\ast}\colon \mrm{Rep}_{\mbb{Z}_p}(G_{\infty})
\to
{}'\mbf{\Phi M}_{/\cO},\quad T\mapsto 
(\wh{\cOur}\otimes_{\mbb{Z}_p} T)^{G_{\infty}}
\]
and 
\[
\mcal{M}_{\ast}\colon \mrm{Rep}_{\mbb{Q}_p}(G_{\infty})
\to
\mbf{\Phi M}_{/\E},\quad T\mapsto 
(\wh{\Eur}\otimes_{\mbb{Q}_p} T)^{G_{\infty}}. 
\]
There exist natural 
$\wh{\cOur}$-linear 
isomorphisms which are compatible with $\vphi$-structures
and $G_{\infty}$-actions:
\begin{equation}
\label{Fon2}
\wh{\cOur}\otimes_{\cO}\mcal{M}_{\ast}(T) 
\overset{\sim}{\longrightarrow}
\wh{\cOur}\otimes_{\mbb{Z}_p} T\quad  \mrm{for}\ T\in \mrm{Rep}_{\mrm{Z}_p}(G_{\infty})  
\end{equation}
and \begin{equation}
\label{Fon2'}
\wh{\Eur}\otimes_{\E}\mcal{M}_{\ast}(T) 
\overset{\sim}{\longrightarrow}
\wh{\Eur}\otimes_{\mbb{Q}_p} T\quad  \mrm{for}\ T\in \mrm{Rep}_{\mbb{Q}_p}(G_{\infty}).  
\end{equation}
We denote by $\mbf{\Phi M}_{/\cO_{\infty}}$ 
(resp.\ $\mbf{\Phi M}_{/\cO}$)
the category of finite torsion \'etale $\cO$-modules 
(resp.\ the category of finite free \'etale $\cO$-modules).

\begin{proposition}
\label{Fon}
The functor $\mcal{T}_{\ast}$ induces equivalences 
of categories between $\mbf{\Phi M}_{/\cO_{\infty}}$ 
$($resp.\ $\mbf{\Phi M}_{/\cO}$, resp.\ $\mbf{\Phi M}_{/\E})$
and $\mrm{Rep}_{\mrm{tor}}(G_{\infty})$
$($resp.\ $\mrm{Rep}_{\mrm{fr}}(G_{\infty})$,
resp.\ $\mrm{Rep}_{\mbb{Q}_p}(G_{\infty}))$.
Quasi-inverse functor of $\mcal{T}_{\ast}$ is $\mcal{M}_{\ast}$. 
\end{proposition}

The contravariant version of the functor $\mcal{T}_{\ast}$
is useful for integral theory.
For any $T\in \mrm{Rep}_{\mrm{tor}}(G_{\infty})$,
put
\[
\mcal{M}(T)=\mrm{Hom}_{\mbb{Z}_p[G_{\infty}]}(T,\E^{\mrm{ur}}/\cOur)\quad 
\mrm{if}\ T\ \mrm{is\ killed\ by\ some\ power\ of}\ p, 
\]
\[
\mcal{M}(T)=\mrm{Hom}_{\mbb{Z}_p[G_{\infty}]}(T,\wh{\cOur})\quad 
\mrm{if}\ T\ \mrm{is}\ \mrm{free},
\]
and for any $T\in \mrm{Rep}_{\mbb{Q}_p}(G_{\infty})$, put
\[
\mcal{M}(T)=\mrm{Hom}_{\mbb{Q}_p[G_{\infty}]}(T,\wh{\Eur}).
\]
Then we can check that $\mcal{T}(M)$ is 
the dual representation of $\mcal{T}_{\ast}(M)$.
For any $M\in \mbf{\Phi M}_{/\cO}$,
put
\[
\mcal{T}(M)=\mrm{Hom}_{\cO,\vphi}(M,\E^{\mrm{ur}}/\cOur)\quad 
\mrm{if}\ M\ \mrm{is\ killed\ by\ some\ power\ of}\ p, 
\] 
\[
\mcal{T}(M)=\mrm{Hom}_{\cO,\vphi}(M,\wh{\cOur})\quad 
\mrm{if}\ M\ \mrm{is}\ p\ \mrm{torsion\ free}, 
\]
and for any $M\in \mbf{\Phi M}_{/\E}$,
put
\[
\mcal{T}(M)=\mrm{Hom}_{\E,\vphi}(M,\wh{\Eur}). 
\] 
These formulations give us contravariant  functors $\mcal{T}$ and $\mcal{M}$ (on appropriate categories) 
such that $\mcal{M}\circ \mcal{T}\simeq \mrm{Id}, \mcal{T}\circ \mcal{M}\simeq \mrm{Id}$. 

\subsection{Kisin modules}

A {\it $\vphi$-module} ({\it over $\mfS$}) 
is a $\mfS$-module 
$\mfM$ equipped with a $\vphi$-semi-linear map 
$\vphi\colon \mfM\to \mfM$.
A $\vphi$-module is called a {\it Kisin module}.
A morphism between two $\vphi$-modules 
$(\mfM_1,\vphi_1)$ and $(\mfM_2,\vphi_2)$
is a $\mfS$-linear morphism 
$\mfM_1\to \mfM_2$ compatible 
with Frobenii $\vphi_1$ and $\vphi_2$.
Denote by $'\mrm{Mod}^r_{/\mfS}$
the category of $\vphi$-modules $\mfM$ 
{\it of height $r$}  
in the following sense; 
\begin{itemize}
\item if $r<\infty$, then $\mfM$ is of finite type 
over $\mfS$ and the cokernel of 
$\vphi^{\ast}$ is killed by $E(u)^r$,
where $\vphi^{\ast}$ is the $\mfS$-linearization 
$1\otimes \vphi\colon \mfS\otimes_{\vphi,\mfS}\mfM\to \mfM$
of $\vphi$,

\item if $r=\infty$, then $\mfM$ is of height $r'$ for some integer $0\le r'<\infty$.
In this case, $\mfM$ is called {\it of finite height}.
\end{itemize}

\noindent
Let $\mrm{Mod}^r_{/\mfS_{\infty}}$
be the full subcategory of $'\mrm{Mod}^r_{/\mfS}$
consisting of finite $\mfS$-modules $\mfM$
which satisfy the following:

\begin{itemize}
\item $\mfM$ is killed by some power of $p$,
\item $\mfM$ has a two term resolution by finite free
$\mfS$-modules, that is, there exists an exact sequence 
\[
0\to \mf{N}_1\to \mf{N}_2\to \mfM\to 0 
\]
of $\mfS$-modules where 
$\mf{N}_1$ and $\mf{N}_2$ 
are finite free $\mfS$-modules.
\end{itemize}
Let $\mrm{Mod}^{r}_{/\mfS}$
be the full subcategory of 
$'\mrm{Mod}^r_{/\mfS}$
consisting of finite free $\mfS$-modules.
There exists a useful criterion for  
an object of $'\mrm{Mod}^r_{/\mfS}$ 
whether it is an object of $\mrm{Mod}^r_{/\mfS}$ or not
(see Proposition \ref{torKi}).
To describe the criterion, we need a bit more notion.
A $\vphi$-modules $\mfM$ is called {\it $p'$-torsion free} 
if for all non-zero element $x\in \mfM$,
$\mrm{Ann}_{\mfS}(x)=0$ or $\mrm{Ann}_{\mfS}(x)=p^n\mfS$
for some integer $n$.
This is equivalent to the natural map $\mfM\to \cO\otimes_{\mfS} \mfM$
being injective.
If $\mfM$ is killed by some power of $p$,
then $\mfM$ is $p'$-torsion free if and only if $\mfM$ is $u$-torsion free.
Therefore, if $\mfM\in {}'\mrm{Mod}^r_{/\mfS}$ is killed $p$ and $p'$-etale,
$\mfM$ is finite free as a $k[\![u]\!]$-module by the structure theorem 
of finite $k[\![u]\!]$-modules.
A $\vphi$-module $\mfM$ is called {\it \'etale} if $\mfM$ is $p'$-torsion free
and $\cO\otimes_{\mfS} \mfM$ is an \'etale $\cO$-module.
Since $E(u)$ is a unit of $\cO$,
we see that, for any $\mfM\in {}'\mrm{Mod}^r_{/\mfS}$, 
$\mfM$ is \'etale if and only if $\mfM$ is $p'$-torsion free.
Any object of $\mfM\in \mrm{Mod}^r_{/\mfS}$ is clearly \'etale.

For any $\mfM\in \mrm{Mod}^r_{/\mfS_{\infty}}$,
we define a $\mbb{Z}_p[G_{\infty}]$-module  
via
\[
T_{\mfS}(\mfM)=\mrm{Hom}_{\mfS,\vphi}(\mfM,\mfS_{\infty}^{\mrm{ur}}),
\]
where a $G_{\infty}$-action on 
$T_{\mfS}(\mfM)$ is given by 
$(\sigma.g)(x)=\sigma(g(x))$ 
for $\sigma\in G_{\infty}, g\in T_{\mfS}(\mfM), x\in \mfM$.
The representation $T_{\mfS}(\mfM)$ is an object of 
$\mrm{Rep}_{\mrm{tor}}(G_{\infty})$.

\begin{proposition}[\cite{CL1}, Corollary 2.1.6]
\label{cl1}
The functor 
$T_{\mfS}\colon  \mrm{Mod}^r_{/\mfS_{\infty}}\to 
\mrm{Rep}_{\mrm{tor}}(G_{\infty})$
is exact and faithful.
\end{proposition}
\begin{proof}
The exactness follows from Proposition \ref{FoKi} below and the fact that 
the functor $(\mfM\mapsto \cO\otimes_{\mfS} \mfM)$
from $\mrm{Mod}^r_{/\mfS_{\infty}}$ to $\mbf{\Phi M}_{/\cO}$ is exact
(since $\cO$ is flat over $\mfS$).
\end{proof}

Similarly, for any $\mfM\in \mrm{Mod}^r_{/\mfS}$,
we define a $\mbb{Z}_p[G_{\infty}]$-module via
\[
T_{\mfS}(\mfM)=\mrm{Hom}_{\mfS,\vphi}(\mfM,\mfS^{\mrm{ur}}).
\]
The representation $T_{\mfS}(\mfM)$ is an object of 
$\mrm{Rep}_{\mrm{fr}}(G_{\infty})$
and $\mrm{rank}_{\mbb{Z}_p}T_{\mfS}(\mfM)=\mrm{rank}_{\mfS}\mfM$.

\begin{proposition}[\cite{Ki}, Corollary 2.1.4 and Proposition 2.1.12]
\label{Ki}
The functor 
$T_{\mfS}\colon  \mrm{Mod}^r_{/\mfS}\to 
\mrm{Rep}_{\mrm{fr}}(G_{\infty})$
is exact and fully faithful.
\end{proposition}

Let $\mfM\in \mrm{Mod}^r_{/\mfS_{\infty}}$ 
(resp.\ $\mfM\in \mrm{Mod}^r_{/\mfS}$).
Since $E(u)$ is a unit in $\cO$,
we see that $M=\mfM[1/u]:=\cO\otimes_{\mfS} \mfM$ 
is a finite torsion \'etale $\cO$-module
(resp.\ a finite free \'etale $\cO$-module).
Here a Frobenius structure on $M$ is given by 
$\vphi_{M}=\vphi_{\cO}\otimes \vphi_{\mfM}$. 

\begin{proposition}[\cite{Li1}, Corollary 2.2.2]
\label{FoKi}
Let
$\mfM\in \mrm{Mod}^r_{/\mfS_{\infty}}$ or 
$\mfM\in \mrm{Mod}^r_{/\mfS}$.
Then the natural map
\[
T_{\mfS}(\mfM)\to \mcal{T}(\cO\otimes_{\mfS} \mfM)
\]
is an isomorphism as $\mbb{Z}_p$-representations 
of $G_{\infty}$.
\end{proposition}

\subsection{$(\vphi,\hat{G})$-modules}
Let $S$ be the $p$-adic completion 
of $W(k)[u,\frac{E(u)^i}{i!}]_{i\ge 0}$ and endow $S$
with the following structures:
\begin{itemize}
\item a continuous $\varphi$-semi-linear 
Frobenius $\varphi\colon S\to S$
defined by $\varphi(u)=u^p$.

\item a continuous linear derivation 
      $N\colon S\to S$ defined by $N(u)=-u$.
      
\item a decreasing filtration 
      $(\mrm{Fil}^iS)_{i\ge 0}$ in $S$. Here $\mrm{Fil}^iS$ 
      is the $p$-adic closure of the ideal generated 
      by the divided powers $\gamma_j(E(u))=\frac{E(u)^j}{j!}$ for all $j\ge i$.       

\end{itemize}
Put $S_{K_0}=S[1/p]=K_0\otimes_{W(k)} S$.
The inclusion $W(k)[u]\hookrightarrow W(R)$ 
via the map $u\mapsto [$\underbar{$\pi$}$]$
induces inclusions 
$\mfS\hookrightarrow S\hookrightarrow A_{\mrm{cris}}$
and $S_{K_0}\hookrightarrow B^+_{\mrm{cris}}$.
We regard all these rings as subrings 
in $B^+_{\mrm{cris}}$.

Fix a choice of primitive $p^i$-root 
of unity $\zeta_{p^i}$ for $i\ge 0$
such that $\zeta^p_{p^{i+1}}=\zeta_{p^i}$.
Put \underbar{$\e$} $=(\zeta_{p^i})_{i\ge 0}\in R^{\times}$
and $t=\mrm{log}([$\underbar{$\e$}$])\in A_{\mrm{cris}}$.
Denote by $\nu\colon W(R)\to W(\bar k)$ 
a unique lift of the projection $R\to \bar k$.
Since $\nu(\mrm{Ker}(\theta))$ 
is contained in the set $pW(\bar k)$,
$\nu$ extends to a map 
$\nu\colon A_{\mrm{cris}}\to W(\bar k)$
and $\nu \colon B^+_{\mrm{cris}}\to W(\bar k)[1/p]$.
For any subring $A\subset B^+_{\mrm{cris}}$,
we put 
$I_+A=\mrm{Ker}(\nu\ \mrm{on}\  B^+_{\mrm{cris}})\cap A$.
For any integer $n\ge 0$,
let $t^{\{n\}}=t^{r(n)}\gamma_{\tilde{q}(n)}(\frac{t^{p-1}}{p})$ 
where $n=(p-1)\tilde{q}(n)+r(n)$ with $0\le r(n) <p-1$
and $\gamma_i(x)=\frac{x^i}{i!}$ is 
the standard divided power.

We define a subring $\mcal{R}_{K_0}$ of $B^+_{\mrm{cris}}$
as below:
\[
\mcal{R}_{K_0}=\{\sum^{\infty}_{i=0} f_it^{\{i\}}\mid f_i\in S_{K_0}\
\mrm{and}\ f_i\to 0\ \mrm{as}\ i\to \infty\}.
\]
Put $\wh{\mcal{R}}=\mcal{R}_{K_0}\cap W(R)$
and $I_+=I_+\wh{\mcal{R}}$.

For any field  $F$ over $\mbb{Q}_p$, set 
$F_{p^{\infty}}=\cup^{\infty}_{n\ge 0} F(\zeta_{p^n})$.
Recall $K_{\infty}=\cup_{n\ge 0}K(\pi_n)$ 
and note that 
$K_{\infty, p^{\infty}}=\cup_{n\ge 0} K(\pi_n, \zeta_{p^{\infty}})$
is the Galois closure of 
$K_{\infty}$ over $K$.
Put $H_K=\mrm{Gal}(K_{\infty,p^{\infty}}/K_{\infty}), 
H_{\infty}=\mrm{Gal}(\bar{K}/K_{\infty,p^{\infty}}),
G_{p^{\infty}}=\mrm{Gal}(K_{\infty,p^{\infty}}/K_{p^{\infty}})$ and 
$\hat{G}=\mrm{Gal}(K_{{\infty},p^{\infty}}/K)$.

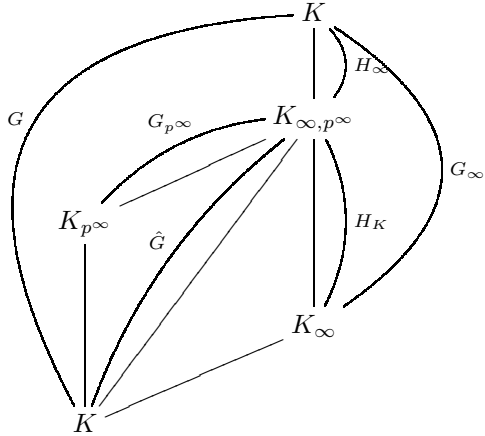
\begin{figure}[htbp]
\begin{center}
\[
\xymatrix{
& & & & & \bar{K} \\
& & & & & K_{\infty,p^{\infty}} \ar@{-}[u] \ar@{-}[u] \ar@/_1pc/@{-}[u]_{H_{\infty}}
\\
& & & K_{p^{\infty}} 
\ar@/^1pc/@{-}[rru] ^{G_{p^{\infty}}}
\ar@{-}[rru] 
& & & & \\
& & & & & 
K_{\infty} \ar@/_1pc/@{-}[uu] _{H_K}
\ar@{-}[uu]
\ar@/_4pc/@{-}[uuu] _{G_{\infty}}
& & & & \\  
& & & K \ar@{-}[rru]  
\ar@/^1pc/@{-}[rruuu] ^{\hat{G}}
\ar@{-}[rruuu] 
\ar@/^6pc/@{-}[rruuuu] ^{G}
\ar@{-}[uu] 
& & & & \\ 
}
\]

\end{center}
\caption{Galois groups of field extensions}
\end{figure}

\begin{proposition}[\cite{Li3}, Lemma 2.2.1]
\label{hatR}
$(1)$ $\wh{\mcal{R}}$ (resp.\ $\mcal{R}_{K_0}$) 
is a $\vphi$-stable $\mfS$-algebra
as a subring in $W(R)$ (resp.\ $B^+_{\mrm{cris}}$).

\noindent
$(2)$ $\wh{\mcal{R}}$ and $I_+$ 
(resp.\ $\mcal{R}_{K_0}$ and $I_+\mcal{R}_{K_0}$) are $G$-stable.
The $G$-action on $\wh{\mcal{R}}$ and $I_+$ 
(resp.\ $\mcal{R}_{K_0}$ and $I_+\mcal{R}_{K_0}$)
factors through $\hat{G}$.

\noindent
$(3)$ There exist natural isomorphisms 
$\mcal{R}_{K_0}/I_+\mcal{R}_{K_0}\simeq K_0$ and 
$\wh{\mcal{R}}/I_+\simeq S/I_+S\simeq \mfS/I_+\mfS\simeq W(k)$.
\end{proposition}

For any Kisin module $(\mfM,\vphi_{\mfM})$ 
of height $r$,
we put $\hat{\mfM}=\wh{\mcal{R}}\otimes_{\vphi, \mfS} \mfM$
and equip $\hat{\mfM}$ with 
a Frobenius $\vphi_{\hat{\mfM}}$ by
$\vphi_{\hat{\mfM}}=
\phi_{\wh{\mcal{R}}}\otimes \vphi_{\hat{\mfM}}$.
It is known that a natural map
\[
\mfM\rightarrow \wh{\mcal{R}}\otimes_{\vphi, \mfS} \mfM=\hat{\mfM}
\]
is an injection (\cite{CL2}, Section 3.1).
By this injection, 
we regard $\mfM$ as a $\vphi(\mfS)$-stable 
submodule in $\hat{\mfM}$.

\begin{definition}
\label{Liumod}
A {\it weak $(\vphi, \hat{G})$-module} $(${\it of height} $r$$)$ 
is a triple $\hat{\mfM}=(\mfM, \vphi_{\mfM}, \hat{G})$ where
\begin{enumerate}
\item[(1)] $(\mfM, \vphi_{\mfM})$ is a 
           Kisin module (of height $r$), 

\vspace{-2mm}
           
\item[(2)] $\hat{G}$ is a $\wh{\mcal{R}}$-semi-linear
           $\hat{G}$-action on $\whR\otimes_{\vphi, \mfS} \mfM$
           which induces a continuous $G$-action on 
            $W(\mrm{Fr}R)\otimes_{\vphi, \mfS} \mfM$
           for the weak topology\footnote{
           Suppose that $\mfM$ is free as a $\mfS$-module. 
           Give $\wh{\mcal{R}}\otimes_{\vphi, \mfS} \mfM$ 
           (resp.\ $W(\mrm{Fr}R)\otimes_{\vphi, \mfS} \mfM$) the weak topology
           using any $\whR$-basis (resp.\ $W(\mrm{Fr}R)$-basis), 
           which is independent of the choice of basis.
           Then we may replace the condition (2) with the following condition (2)':\\
           (2)'  $\hat{G}$ is a continuous $\wh{\mcal{R}}$-semi-linear
           $\hat{G}$-action on $\whR\otimes_{\vphi, \mfS} \mfM$
           for the weak topology.\\
           In fact, if $G$ acts on $\whR\otimes_{\vphi, \mfS} \mfM$ continuously,
           then the $G$-action on $\hat{T}(\hat{\mfM})$ is continuous for the $p$-adic topology
           (the definition for $\hat{T}(\hat{\mfM})$
           is given before Theorem \ref{Li}).
           Since the map $\hat{\iota}$ in Lemma \ref{rel2} (4) is a topological isomorphism for 
           weak topologies on both sides, we see that the $G$-action on 
           $W(\mrm{Fr}R)\otimes_{\vphi, \mfS} \mfM$ is automatically continuous.},
           
\vspace{-2mm}           
           
\item[(3)] the $\hat{G}$-action commutes with $\vphi_{\hat{\mfM}}$,

\vspace{-2mm}

\item[(4)] $\mfM\subset \hat{\mfM}^{H_K}$.

\end{enumerate}

\noindent
A  weak $(\vphi, \hat{G})$-module $\hat{\mfM}$ is called
a {\it $(\vphi, \hat{G})$-module} if it satisfies the additional condition
\begin{enumerate}

\item[(5)] $\hat{G}$ acts on the 
           $W(k)$-module $\hat{\mfM}/I_+\hat{\mfM}$ trivially.           
\end{enumerate}

If $\mfM$ is a torsion 
(resp.\ free) Kisin module of (height $r$),
we call $\hat{\mfM}$ a 
{\it torsion} 
(resp.\ {\it free}) {\it $(\vphi, \hat{G})$-module} 
({\it of} height $r$).
By analogous way, we define the notion of a torsion (resp.\ free) weak $(\vphi, \hat{G})$-module.
If $\hat{\mfM}=(\mfM,\vphi_{\mfM},\hat{G})$ 
is a weak $(\vphi, \hat{G})$-module,
we often abuse of notations by 
denoting $\hat{\mfM}$ the underlying 
module $\wh{\mcal{R}}\otimes_{\vphi, \mfS} \mfM$.
\end{definition}

A morphism 
$f\colon (\mfM, \vphi, \hat{G})\to (\mfM', \vphi', \hat{G})$  
between two weak $(\vphi, \hat{G})$-modules is a morphism 
$f\colon (\mfM,\vphi)\to (\mfM',\vphi')$ of Kisin-modules 
such that 
$\wh{\mcal{R}}\otimes f\colon \hat{\mfM}\to \hat{\mfM}'$
is a $\hat{G}$-equivalent.
We denote by ${}_{\mrm{w}}\mrm{Mod}^{r,\hat{G}}_{/\mfS_{\infty}}$ 
(resp.\ ${}_{\mrm{w}}\mrm{Mod}^{r,\hat{G}}_{/\mfS}$, 
resp.\ $\mrm{Mod}^{r,\hat{G}}_{/\mfS_{\infty}}$,
resp.\ $\mrm{Mod}^{r,\hat{G}}_{/\mfS}$) the 
category of torsion weak  $(\vphi, \hat{G})$-modules
(
resp.\ free weak $(\vphi, \hat{G})$-modules, 
resp.\ torsion $(\vphi, \hat{G})$-modules,
resp.\ free $(\vphi, \hat{G})$-modules).
We regard $\hat{\mfM}$ as a $G$-module 
via the projection $G\twoheadrightarrow \hat{G}$. 
A sequence
$0\to \hat{\mfM}' \to \hat{\mfM} \to \hat{\mfM}''\to 0$
of (weak) $(\vphi,\hat{G})$-modules is exact if it is exact as $\mfS$-modules
and all morphisms are morphisms of (weak) $(\vphi,\hat{G})$-modules.

For a weak $(\vphi, \hat{G})$-module $\hat{\mfM}$,
we define a $\mbb{Z}_p[G]$-module as below:
\[
\hat{T}(\hat{\mfM})=\mrm{Hom}_{\wh{\mcal{R}},\vphi}(\hat{\mfM}, W(R)_{\infty})\quad 
\mrm{if}\ \mfM \ \mrm{is\ killed\ by\ some\ power\ of}\ p
\]
and 
\[
\hat{T}(\hat{\mfM})=\mrm{Hom}_{\wh{\mcal{R}},\vphi}(\hat{\mfM}, W(R))\quad 
\mrm{if}\ \mfM \ \mrm{is\ free}.
\]
Here,  $G$ 
acts on $\hat{T}(\hat{\mfM})$ by $(\sigma.f)(x)=\sigma(f(\sigma^{-1}(x)))$
for $\sigma\in G,\ f\in \hat{T}(\hat{\mfM}),\ x\in \hat{\mfM}$. 

Let $\hat{\mfM}=(\mfM, \vphi_{\mfM}, \hat{G})$ 
be a weak $(\vphi,\hat{G})$-module.
There exists a natural map 
\[
\theta\colon T_{\mfS}(\mfM)\to \hat{T}(\hat{\mfM})
\]
defined by 
\[
\theta(f)(a\otimes m)=a\vphi(f(m))\quad \mrm{for}\ 
f\in T_{\mfS}(\mfM),\ a\in \wh{\mcal{R}}, m\in \mfM,
\]
which is a $G_{\infty}$-equivalent.

Let denote by $\mrm{Rep}^{r}_{\mbb{Z}_p}(G)$ the category
of $G$-stable $\mbb{Z}_p$-lattices in semi-stable 
$p$-adic representations of $G$ with Hodge-Tate weights in $[0,r]$.

\begin{theorem}[\cite{Li3},\cite{CL2}]
\label{Li}
Let $\hat{\mfM}=(\mfM, \vphi_{\mfM}, \hat{G})$ 
be a weak $(\vphi,\hat{G})$-module.

\noindent
$(1)$ The map
$\theta\colon T_{\mfS}(\mfM)\to \hat{T}(\hat{\mfM})$
is an isomorphism of $\mbb{Z}_p[G_{\infty}]$-modules.

\noindent
$(2)$ The functor $\hat{T}$ induces an anti-equivalence 
between  $\mrm{Mod}^{r,\hat{G}}_{\mfS}$ 
and $\mrm{Rep}^{r}_{\mbb{Z}_p}(G)$.
\end{theorem}

\begin{corollary}
The functor $\hat{T}\colon {}_{\mrm{w}}\mrm{Mod}^{r,\hat{G}}_{/\mfS_{\infty}}
\to \mrm{Rep}^{r}_{\mrm{tor}}(G)$
is exact and faithful.
\end{corollary}

\begin{proof}
The exactness of the functor $\hat{T}$ 
follows from \ref{FoKi} and Theorem \ref{Li} (1).
Since  $T_{\mfS}\colon \mrm{Mod}^r_{/\mfS_{\infty}}
\to \mrm{Rep}^{r}_{\mrm{tor}}(G_{\infty})$ is faithful, 
the faithfulness of $\hat{T}$ follows from the following commutative 
diagram:
\begin{center}
$\displaystyle \xymatrix{
\mrm{Hom}_{{}_{\mrm{w}}\mrm{Mod}^{r,\hat{G}}_{/\mfS_{\infty}}}
(\hat{\mfM},\hat{\mfM}')\ar[r] \ar@{^{(}->}[r] \ar_{\hat{T}}[d]& 
\mrm{Hom}_{\mrm{Mod}^r_{/\mfS_{\infty}}}(\mfM,\mfM') \ar@{^{(}->}[r] \ar^{T_{\mfS}\quad }[r]  & 
\mrm{Hom}_{G_{\infty}}(T_{\mfS}(\mfM'),T_{\mfS}(\mfM)) \ar^{\wr}[d]\\
\mrm{Hom}_{G}(\hat{T}(\hat{\mfM}'),\hat{T}(\hat{\mfM})) \ar@{^{(}->}[rr] &
 &
\mrm{Hom}_{G_{\infty}}(\hat{T}(\hat{\mfM}'),\hat{T}(\hat{\mfM})). 
}$
\end{center}

\end{proof}

\subsection{Some fundamental  properties}

In this subsection, 
we give some fundamental, but important, results
on Kisin modules and $(\vphi,\hat{G})$-modules. 
We start with the following proposition which plays an important role throughout this paper.

\begin{proposition}[\cite{Li1}, Proposition 2.3.2]
\label{torKi}
Let $\mfM\in {}'\mrm{Mod}^r_{/\mfS}$ be killed by $p^n$.
The following statements are equivalent:

\noindent
$(1)$ $\mfM\in \mrm{Mod}^r_{/\mfS_{\infty}}$.

\noindent
$(2)$ $\mfM$ is $u$-torsion free.

\noindent
$(3)$ $\mfM$ is \'etale.

\noindent
$(4)$ $\mfM$ has a successive extension of finite free $k[\![u]\!]$-modules in $'\mrm{Mod}^r_{/\mfS_{\infty}}$,
that is,
there exists an extension 
\[
0=\mfM_0\subset \mfM_1\subset \cdots \subset\mfM_k=\mfM
\]
in $'\mrm{Mod}^r_{/\mfS_{\infty}}$ such that 
$\mfM_i/\mfM_{i-1}\in {}'\mrm{Mod}^r_{/\mfS_{\infty}}$ and $\mfM_i/\mfM_{i-1}$
is a finite free $k[\![u]\!]$-module.

\noindent
$(5)$ $\mfM$ is a quotient of two finite free $\mfS$-modules $\mfN'$ and $\mfN''$ with 
$\mfN',\mfN''\in \mrm{Mod}^r_{/\mfS}$. 
\end{proposition}

\begin{remark}
\label{suc}
By Lemma 2.3.1 of \cite{Li1}, it is easy to see
that, for any $i$,  $\mfM_i$ and $\mfM_i/\mfM_{i-1}$ appeared in  
Proposition \ref{torKi} (4) are in fact objects of $\mrm{Mod}^r_{/\mfS_{\infty}}$.
\end{remark}

\begin{corollary}
\label{exact}
Let $A$ be a $\mfS$-algebra without $p$-torsion.
Then $\mrm{Tor}^{\mfS}_1(\mfM,A)=0$ for any Kisin module $\mfM$.
In particular, the functor $\mfM\mapsto A\otimes_{\mfS} \mfM$ is an exact functor 
from the category of Kisin modules to the category of $A$-modules. 
\end{corollary}

\begin{proof}
If $\mfM$ is a free Kisin module, then the fact $\mrm{Tor}^{\mfS}_1(\mfM,A)=0$
is clear.
Let $\mfM$ 
be a torsion Kisin module and let show $\mrm{Tor}^{\mfS}_1(\mfM,A)=0$.
For this proof,
we use Proposition \ref{torKi} (4) and {\it d\'evissage} 
to reduce the  proof to the case that 
$\mfM$ is killed by $p$.
Then it suffices to show 
$\mrm{Tor}^{\mfS}_1(k[\![u]\!],A)=0$. 
The exact sequence 
$0\to \mfS\overset{p}{\rightarrow} \mfS\to k[\![u]\!]\to 0$
induces the exact sequence 
$\mrm{Tor}^{\mfS}_1(\mfS,A)\to 
\mrm{Tor}^{\mfS}_1(k[\![u]\!],A)\to A\overset{p}{\rightarrow} A$,
and since $\mrm{Tor}^{\mfS}_1(\mfS,A)=0$ and $A$ has no $p$-torsion,
we obtain $\mrm{Tor}^{\mfS}_1(\mfM,A)=0$.

\end{proof}

\begin{corollary}
\label{ringext}
Let $\mfM$ be an object of $\mrm{Mod}^r_{/\mfS_{\infty}}$ or $\mrm{Mod}^r_{/\mfS}$.
Let $A\subset B$ be a ring extension of $p$-torsion free $\mfS$-algebras.
Suppose that the natural map
$A_1\to B_1$ is an injection.
Then a natural map $A\otimes_{\mfS} \mfM\to B\otimes_{\mfS} \mfM$ is injective.
\end{corollary}
In this paper, we often regard $A\otimes_{\mfS} \mfM$ (resp.\ $A\otimes_{\vphi,\mfS} \mfM$)
as a submodule of $B\otimes_{\mfS} \mfM$ (resp.\ $ B\otimes_{\vphi,\mfS} \mfM$).

\begin{proof}
The statement is clear if $\mfM$ is free over $\mfS$ or killed by $p$ (since $A_1\subset B_1$).
Suppose that $\mfM$ is killed by some power of $p$.
Take an extension 
$
0=\mfM_0\subset \mfM_1\subset \cdots \subset\mfM_k=\mfM
$
as in Proposition \ref{torKi} (4).
Note that $\mfM_i$ and $\mfM_i/\mfM_{i-1}$ are in $\mrm{Mod}^r_{/\mfS_{\infty}}$
(cf.\ Remark \ref{suc}).
Since two horizontal sequences of the diagram 
\begin{center}
$\displaystyle \xymatrix{
0\ar[r] & 
A\otimes_{\mfS} \mfM_{i-1}\ar[r] \ar[d]& 
A\otimes_{\mfS}\mfM_i \ar[r] \ar[d]&
A\otimes_{\mfS} \mfM_{i-1}\ar[r] \ar@{^{(}->}[d]& 
0 \\
0\ar[r] & 
B\otimes_{\mfS} \mfM_{i-1}\ar[r] & 
B\otimes_{\mfS}\mfM_i \ar[r] & 
B\otimes_{\mfS} \mfM_{i-1}\ar[r] &
0
}$
\end{center}
are exact (see Corollary \ref{exact}), 
an induction on $i$ induces the desired result.

\begin{corollary}
\label{modext}
Let $\mfM$ be an object of $\mrm{Mod}^r_{/\mfS_{\infty}}$
and $\mfN$ any $\vphi$-module over $\mfS$ with $\mfM\subset \mfN$.
Let $\mfS\subset A\subset W(\mrm{Fr}R)$ be a ring extension 
such that $\mfS_1\to \mrm{Fr}R$ is injective.

\noindent
$(1)$ A natural map $A\otimes_{\mfS} \mfM\to A\otimes_{\mfS} \mfN$
is injective.

\noindent
$(2)$ If $A$ is $\vphi$-stable, then a natural map 
$A\otimes_{\vphi, \mfS} \mfM\to A\otimes_{\vphi, \mfS} \mfN$
is injective.
\end{corollary}

\begin{proof}
We only prove (2) (the proof for (1) is similar). 
See the following commutative diagram:
\begin{center}
$\displaystyle \xymatrix{
A\otimes_{\vphi, \mfS} \mfM \ar[r] \ar[d] & 
A\otimes_{\vphi, \mfS} \mfN \ar[d] \\
W(\mrm{Fr}R)\otimes_{\vphi, \mfS} \mfM \ar[r] &
W(\mrm{Fr}R)\otimes_{\vphi, \mfS} \mfN. 
}$
\end{center}
The left vertical map is injective by Corollary \ref{ringext}
and the bottom horizontal map is also injective since $\vphi\colon \mfS\to W(\mrm{Fr}R)$
is flat.
Hence we obtain the desired result.
\end{proof}

\end{proof}

\begin{remark} 
\label{subsets}
Let $n>0$ be an integer and $\mfS\subset A\subset B\subset W(\mrm{Fr}R)$ a ring extension.

\noindent
$(1)$ 
If $A$ satisfies the condition that 
the natural map $A_n\to W_n(\mrm{Fr}R)$ is injective,
then for any $A\subset B\subset W(\mrm{Fr}R)$, the map $A_n\to B_n$ is also injective. 

\noindent
$(2)$ (cf.\ \cite{CL2}, Lemma 3.1.1 and \cite{Fo}, Proposition 1.8.3)We have the following injections:
\begin{center}
$\displaystyle \xymatrix{
\whR_n \ar@{^{(}->}[r] & W_n(R)\ar@{^{(}->}[r] & W_n(\mrm{Fr}R) \\
\mfS_n\ar@{^{(}->}[r] \ar@{^{(}->}[u] & \mfS^{\mrm{ur}}_n\ar@{^{(}->}[r] \ar@{^{(}->}[u] & 
\cOur_n. \ar@{^{(}->}[u]. 
}$
\end{center}
\end{remark}

\begin{corollary}
Let $\mfM$ be an object of $\mrm{Mod}^r_{/\mfS_{\infty}}$ and $n\ge 0$ an integer.
Then $p^nT_{\mfS}(\mfM)=0$ if and only if $p^n\mfM=0$.
\end{corollary}

\begin{proof}
The sufficiency is clear from the definition of $T_{\mfS}$.
Suppose $p^nT_{\mfS}(\mfM)=0$.
First we prove the case where $n=0$.
By Proposition \ref{torKi} and Remark \ref{suc}, 
there exists an extension 
\[
0=\mfM_0\subset \mfM_1\subset \cdots \subset\mfM_k=\mfM
\]
in $\mrm{Mod}^r_{/\mfS_{\infty}}$ such that 
$\mfM_{i+1}/\mfM_i\in {}\mrm{Mod}^r_{/\mfS_{\infty}}$ and $\mfM_{i+1}/\mfM_i$
is a finite free $k[\![u]\!]$-module.
Taking $T_{\mfS}$ to the exact sequence
$0\to \mfM_i\to \mfM_{i+1}\to \mfM_{i+1}/\mfM_i\to0$,
we obtain an exact sequence 
$0\to T_{\mfS}(\mfM_{i+1}/\mfM_i)\to T_{\mfS}(\mfM_i)\to T_{\mfS}(\mfM_{i+1}) \to0$
of $\mbb{Z}_p[G_{\infty}]$-modules.
Since $T_{\mfS}(\mfM_k)=T_{\mfS}(\mfM)=0$, we obtain $T_{\mfS}(\mfM_{i+1}/\mfM_i)=0$.
By Lemma 2.1.2 of \cite{Ki}, this implies $\mfM_k=\mfM_{k-1}$ and in particular, $T_{\mfS}(\mfM_{k-1})=0$.
Inductively, we obtain $\mfM_k=\mfM_{k-1}=\cdots =\mfM_0=0$.
For general $n\ge 0$, we consider an exact sequence 
$
0\to \mrm{ker}(p^n)\to \mfM\overset{p^n}{\to} \mfM
$
in $\mrm{Mod}^r_{/\mfS_{\infty}}$.
Since $p^nT_{\mfS}(\mfM)=0$, if we take $T_{\mfS}$ to this sequence, 
we have $T_{\mfS}(\mfM)\simeq T_{\mfS}(\mrm{ker}(p^n))$.
Therefore, taking $T_{\mfS}$ to an exact sequence 
$
0\to \mrm{ker}(p^n)\overset{p^n}{\to} \mfM\to \mfM/\mrm{ker}(p^n)\to 0
$
in $\mrm{Mod}^r_{/\mfS_{\infty}}$,
we obtain $T_{\mfS}(\mfM/\mrm{ker}(p^n))=0$ and then $\mfM/\mrm{ker}(p^n)=0$.
\end{proof}

\begin{lemma}
\label{ptor}
Let $\mfM$ be a finite $\mfS$-module.
If $\mfM$ is $p'$-torsion free, then $\mfM/p\mfM$ is also.
\end{lemma}
\begin{proof}
We may suppose that $\mfM\not= 0$.
By an elementary ring theory,
we obtain
$\sqrt{\mathstrut \mrm{Ann}_{\mfS}(\mfM/p\mfM)}= 
\sqrt{\mathstrut \mrm{Ann}_{\mfS}(\mfM)+p\mfS}=
p\mfS$
and thus 
$\mrm{Ann}_{\mfS}(\mfM/p\mfM)=p\mfS$.
\end{proof}

\begin{proposition}
\label{tensorKi}
Let $\mfM\in \mrm{Mod}^r_{/\mfS_{\infty}}$ and $\mfM\in \mrm{Mod}^{r'}_{/\mfS_{\infty}}$
for some $r,r'\in \{0,1,\dots ,\infty\}$.
Then $\frac{\mfM\otimes_{\mfS}\mfM'}{u{\rm \mathchar`-tor}}$ is an object of $\mrm{Mod}^{r+r'}_{/\mfS_{\infty}}$.
If we put $\mfM\otimes \mfM'=\frac{\mfM\otimes_{\mfS}\mfM'}{u{\rm \mathchar`-tor}}$,
then there exists a canonical isomorphism
$T_{\mfS}(\mfM\otimes\mfM')\simeq T_{\mfS}(\mfM)\otimes_{\mbb{Z}_p}T_{\mfS}(\mfM')$
 of $\mbb{Z}_p[G_{\infty}]$-modules.
Furthermore,
if $\mfM$ or $\mfM'$ is killed by $p$,
then $\mfM\otimes_{\mfS}\mfM'$ is $u$-torsion free. 
\end{proposition}

\begin{proof}
To check $\frac{\mfM\otimes_{\mfS}\mfM'}{u{\rm \mathchar`-tor}}\in \mrm{Mod}^{r+r'}_{/\mfS_{\infty}}$
is not difficult.
Putting $M=\mfM[1/u]$ and $M'=\mfM'[1/u]$, 
we have $\frac{\mfM\otimes_{\mfS}\mfM'}{u{\rm \mathchar`-tor}}[1/u]\simeq M\otimes M'$.
By Proposition \ref{FoKi},
we obtain
$T_{\mfS}(\mfM\otimes \mfM')\simeq T(M\otimes_{\cO}M')\simeq T(M)\otimes_{\mbb{Z}_p}T(M')\simeq 
T_{\mfS}(\mfM)\otimes_{\mbb{Z}_p} T_{\mfS}(\mfM').$
The last assertion follows from Lemma \ref{ptor}.
\end{proof}

\begin{proposition}[Scheme-theoretic closure, \cite{Li1}, Lemma 2.3.6]
\label{stc}
Let $f\colon \mfM\to L$ be a morphism of $\vphi$-modules
over $\mfS$.
Suppose that $\mfM$ and $L$ are $p'$-torsion free and $\mfM\in {}'\mrm{Mod}^r_{/\mfS}$.
Then $\mrm{ker}(f)$ and $\mrm{im}(f)$ are \'etale and belong to $'\mrm{Mod}^r_{/\mfS}$.
In particular, 
if $\mfM\in \mrm{Mod}^r_{/\mfS_{\infty}}$,
then  $\mrm{ker}(f)$ and $\mrm{im}(f)$ are also in $\mrm{Mod}^r_{/\mfS_{\infty}}$.
\end{proposition}

\noindent
There exists a $(\vphi,\hat{G})$-analogue of the above proposition.

\begin{corollary}
\label{imker}
Let $\hat{\mfM}$ and $\hat{\mfM}'$ be in ${}_{\mrm{w}}\mrm{Mod}^{r,\hat{G}}_{/\mfS_{\infty}}$  
$($resp.\ $\mrm{Mod}^{r,\hat{G}}_{/\mfS_{\infty}})$. 
Let $f\colon \hat{\mfM}\to \hat{\mfM}'$ be a morphism of 
$(\vphi,\hat{G})$-modules.
Then, $\mrm{ker}(f)$ and $\mrm{im}(f)$ as $\vphi$-modules are in 
$\mrm{Mod}^{r}_{/\mfS_{\infty}}$,
a $\hat{G}$-action on $\hat{\mfM}$
gives $\mrm{ker}(f)$ a structure of 
a weak $(\vphi,\hat{G})$-module
$($resp.\ a $(\vphi,\hat{G})$-module$)$ 
and a $\hat{G}$-action on $\hat{\mfM}'$
gives a structure of 
a weak $(\vphi,\hat{G})$-module
$($resp.\ a $(\vphi,\hat{G})$-module$)$.
\end{corollary}

\begin{proof}
It is enough to prove only the case where 
$\hat{\mfM}, \hat{\mfM}'\in \mrm{Mod}^{r,\hat{G}}_{/\mfS_{\infty}}$. 
By Proposition \ref{stc},
$\mrm{ker}(f)$ and $\mrm{im}(f)$ as $\vphi$-modules are in 
$\mrm{Mod}^{r}_{/\mfS_{\infty}}$.
Consider the image of $f$.
Let $\hat{f}\colon \whR\otimes_{\vphi, \mfS} \mfM\to \whR\otimes_{\vphi, \mfS} \mfM'$ 
be a morphism induces from $f$.
Since 
$\whR\otimes_{\vphi, \mfS} \mrm{im}(f)= \hat{f}(\whR\otimes_{\vphi, \mfS} \mfM)$ 
and $\hat{f}$ is compatible with $\hat{G}$-action,
we can define $\hat{G}$-action on $\whR\otimes_{\vphi, \mfS} \mrm{im}(f)$
such that $\whR\otimes_{\vphi, \mfS} \mfM\to \whR\otimes_{\vphi, \mfS} \mrm{im}(f)$
induced from $f$ is $\hat{G}$-equivalent.
Since $\whR\otimes_{\vphi, \mfS} \mfM/I_+(\whR\otimes_{\vphi, \mfS} \mfM)
\to \whR\otimes_{\vphi, \mfS} \mrm{im}(f)/I_+(\whR\otimes_{\vphi, \mfS} \mrm{im}(f))$ 
is surjective, it is a routine work to check that 
$\wh{\mrm{im}(f)}=(\mrm{im}(f), \vphi, \hat{G})$ satisfies conditions 
to being a $(\vphi,\hat{G})$-module. 
The assertion for the kernel of $f$
follows from the fact that,
two exact sequences
$
0\to \whR\otimes_{\vphi,\mfS} \mrm{ker}(f)\to 
\whR\otimes_{\vphi,\mfS}\mfM \overset{\hat{f}}{\to} 
\whR\otimes_{\vphi,\mfS}\mrm{im}(f)\to 0
$
and 
$
0\to (\whR/I_+)\otimes_{\vphi,\mfS} \mrm{ker}(f)\to 
(\whR/I_+)\otimes_{\vphi,\mfS}\mfM \overset{\hat{f}}{\to} 
(\whR/I_+)\otimes_{\vphi,\mfS}\mrm{im}(f)\to 0
$
arising from the exact sequence $0\to \mrm{ker}(f)\to \mfM\to \mrm{im}(f)\to 0$
are exact by Corollary \ref{exact} and the fact that $\whR/I_+\simeq W(k)$ is $p$-torsion free. 
\end{proof}

\begin{corollary}
\label{subweak}
Let $0\to \hat{\mfM}'\to \hat{\mfM}\to \hat{\mfM}''\to 0$ be an exact sequence in
${}_\mrm{w}\mrm{Mod}^{r,\hat{G}}_{/\mfS_{\infty}}$.
If $\hat{\mfM}\in \mrm{Mod}^{r,\hat{G}}_{/\mfS_{\infty}}$,
then $\hat{\mfM}'$ and $\hat{\mfM}''$ are also in $\mrm{Mod}^{r,\hat{G}}_{/\mfS_{\infty}}$.
\end{corollary}

\begin{proof}
This immediately follows from Corollary \ref{exact}.
\end{proof}


\section{Cartier duality for $(\varphi, \hat{G})$-modules}

In this subsection,
we give the Cartier duality on  $(\vphi,\hat{G})$-modules. 
Throughout this section, we fix an integer $r<\infty$.

\subsection{Cartier duality for Kisin modules}

In this subsection, 
we recall Liu's results on duality theorems 
for Kisin modules
(\cite{Li1}, Section 3).

\begin{example}
Let $\mfS^{\vee}=\mfS\cdot \mfrak{f}^r$ 
be the rank-$1$ free $\mfS$-module 
with $\vphi(\mfrak{f}^r)=c_0^{-r}E(u)^r\cdot \mfrak{f}^r$ where 
$pc_0$ is the constant coefficient of $E(u)$.
We denote by $\vphi^{\vee}$ this Frobenius $\vphi$.
Then $(\mfS^{\vee},\vphi^{\vee})$ is a free Kisin module of height $r$ and   
there exists an isomorphism 
$T_{\mfS}(\mfS^{\vee})\simeq \mbb{Z}_p(r)$
as $\mbb{Z}_p[G_{\infty}]$-modules
(see \cite{Li1}, Example 2.3.5).
Put $\mfS^{\vee}_{\infty}
=\mbb{Q}_p/\mbb{Z}_p\otimes_{\mbb{Z}_p} \mfS^{\vee}
=\mfS_{\infty}\cdot \mfrak{f}^r$
(resp.\ $\mfS^{\vee}_{n}
=\mbb{Z}_p/p^n\mbb{Z}_p\otimes_{\mbb{Z}_p} \mfS^{\vee}
=\mfS_{n}\cdot \mfrak{f}^r$ for any integer $n\ge 0$).
The Frobenius $\vphi$ on $\mfS^{\vee}$ induces 
Frobenii $\vphi^{\vee}$ on $\mfS_{\infty}^{\vee}$ 
and $\mfS_n^{\vee}$.

Put $\E^{\vee}=\E\otimes_{\mfS} \mfS^{\vee}=\E\cdot \mfrak{f}^r$
and equip $\E^{\vee}$ with a Frobenius $\vphi^{\vee}$ 
arising from that of
$\E$ and $\mfS^{\vee}$.
Similarly, we put 
$\cOdual=\cO \cdot \mfrak{f}^r,
\cOdual_{\infty}=\cO_{\infty}\cdot \mfrak{f}^r,
\cOdual_n=\cO_n\cdot \mfrak{f}^r$
and equip them with Frobenii $\vphi^{\vee}$
which arise from that of $\E^{\vee}$. 
We define 
$\cOurdual,\ \cOurdual_{\infty}$
$\cOurdual_n$, and Frobenii $\vphi^{\vee}$ on them by the analogous way.
\end{example}

Let $\mfM$ be a  Kisin module of height $r$
and denote by $M=\cO\otimes_{\mfS} \mfM$
the corresponding \'etale $\vphi$-module.
Put 
\[
\mfM^{\vee}=\mrm{Hom}_{\mfS}(\mfM, \mfS_{\infty}),\
M^{\vee}=\mrm{Hom}_{\cO,\vphi}(M,\cO_{\infty})
\quad \mrm{if}\ \mfM\ \mrm{is\ killed\ by\ some\ power\ of}\ p
\]
and 
\[
\mfM^{\vee}=\mrm{Hom}_{\mfS}(\mfM, \mfS),\
M^{\vee}=\mrm{Hom}_{\cO,\vphi}(M, \cO)
\quad \mrm{if}\ \mfM\ \mrm{is\ free}.
\]
We then have natural pairings
\[
\langle \cdot , \cdot \rangle \colon \mfM \times \mfM^{\vee}\to
\mfS_{\infty}^{\vee},\
\langle \cdot , \cdot \rangle \colon 
M\times M^{\vee}\to \cOdual_{\infty}
\quad \mrm{if}\ \mfM\ \mrm{is\ killed\ by\ some\ power\ of}\ p
\]
and 
\[
\langle \cdot , \cdot \rangle \colon \mfM \times 
\mfM^{\vee}\to \mfS^{\vee},\
\langle \cdot , \cdot \rangle \colon 
M\times M^{\vee}\to \cOdual
\quad \mrm{if}\ \mfM\ \mrm{is\ free}.
\]
The Frobenius $\vphi_{\mfM}^{\vee}$ on $\mfM^{\vee}$ 
(resp.\  $\vphi_{M}^{\vee}$ on $M^{\vee}$ )
is defined to be 
\[
\langle \vphi_{\mfM}(x), 
\vphi_{\mfM}^{\vee}(y) \rangle = 
\vphi^{\vee}(\langle x,y \rangle) 
\quad \mrm{for}\ x\in \mfM, y\in \mfM^{\vee}.
\]
\[
(\mrm{resp.}\ 
\langle \vphi_{M}(x), \vphi_{M}^{\vee}(y) \rangle = 
\vphi^{\vee}(\langle x,y \rangle) \quad \mrm{for}\ x\in M, y\in M^{\vee}.)
\]

\begin{theorem}[\cite{Li1}]
\label{DualKi}
Let $\mfM$ be a 
Kisin module of height $r$,
$M=\cO\otimes_{\mfS} \mfM$ 
the corresponding \'etale $\vphi$-module
and 
$\langle \cdot , \cdot \rangle$
the paring as above.

\noindent
$(1)$ $(\mfM^{\vee}, \vphi_{\mfM}^{\vee})$
is a Kisin module of height $r$.
Similarly, 
$M^{\vee}$ is an \'etale $\vphi$-module.

\noindent
$(2)$ 
A natural map $\cO\otimes_{\mfS} \mfM^{\vee}\to M^{\vee}$
is an isomorphism and 
$\vphi^{\vee}_{M}=
\vphi_{\cO}\otimes \vphi^{\vee}_{\mfM}$.

\noindent
$(3)$ 
The assignment 
$\mfM \mapsto \mfM^{\vee}$
is an anti-equivalence on the category 
of torsion Kisin-modules
$($resp.\ free Kisin-modules$)$ 
and a natural map 
$\mfM\to (\mfM^{\vee})^{\vee}$
is an isomorphism.

\noindent
$(4)$ All parings $\langle \cdot , \cdot \rangle$ 
appeared in the above are perfect.

\noindent
$(5)$ Taking a dual preserves a short exact sequence of 
torsion Kisin modules
$($resp.\ free Kisin modules, resp.\ torsion \'etale $\vphi$-modules
resp.\ free \'etale $\vphi$-modules$)$. 

\end{theorem}   
\begin{remark}
The assertion (2) of the above theorem says that there exists a natural isomorphism 
$\cO\otimes_{\mfS} {\mfM}^{\vee}\simeq (\cO\otimes_{\mfS} {\mfM})^{\vee}=M^{\vee}$
which is compatible with $\vphi$-structures.
In fact, the paring 
$\langle \cdot , \cdot \rangle$
for $M$ is equal to the pairing 
which is obtained by tensoring $\cO$
to the pairing $\langle \cdot , \cdot \rangle$ for $\mfM$.    
\end{remark}

\subsection{Construction of dual objects}
Put
\[
\wh{\mfS}^{\vee}
=\wh{\mcal{R}}\otimes_{\vphi,\mfS} \mfS^{\vee}
=\wh{\mcal{R}}\otimes_{\vphi,\mfS} (\mfS\cdot \mfrak{f}^r)
=\wh{\mcal{R}}\cdot \mfrak{f}^r,
\]
\[
\wh{\mfS}_n^{\vee}
=\mbb{Z}_p/p^n\mbb{Z}_p\otimes_{\mbb{Z}_p} \wh{\mfS}^{\vee}
=\wh{\mcal{R}}\otimes_{\vphi,\mfS} \mfS_n^{\vee}
=\wh{\mcal{R}}\otimes_{\vphi,\mfS} (\mfS_n\cdot \mfrak{f}^r)
=\wh{\mcal{R}}_n\cdot  \mfrak{f}^r
\quad \mrm{for\ any\ integer}\ n\ge 0
\]
and
\[
\wh{\mfS}_{\infty}^{\vee}
=\mbb{Q}_p/\mbb{Z}_p\otimes_{\mbb{Z}_p} \wh{\mfS}^{\vee}
=\wh{\mcal{R}}\otimes_{\vphi,\mfS} \mfS_{\infty}^{\vee}
=\wh{\mcal{R}}\otimes_{\vphi,\mfS} (\mfS_{\infty}\cdot \mfrak{f}^r)
=\wh{\mcal{R}}_{\infty}\cdot  \mfrak{f}^r,
\]
and we equip them with natural Frobenii arising from those of $\wh{\mcal{R}}$ and $\mfS^{\vee}$.
By Theorem \ref{Li},
we can define a unique $\hat{G}$-action on $\wh{\mfS}^{\vee}$ such that 
$\wh{\mfS}^{\vee}$ has a structure as a $(\vphi,\hat{G})$-module 
of height $r$ and there exists an isomorphism
\begin{equation}
\label{ga}
\hat{T}(\wh{\mfS}^{\vee})\simeq \mbb{Z}_p(r)
\end{equation}
as $\mbb{Z}_p[G]$-modules.
This $\hat{G}$-action on $\wh{\mfS}^{\vee}$ induces 
$\hat{G}$-actions on $\wh{\mfS}_n^{\vee}$
and  $\wh{\mfS}_{\infty}^{\vee}$. 
Then it is not difficult to see that $\wh{\mfS}_n^{\vee}$ has a structure as a torsion 
$(\vphi,\hat{G})$-module of height $r$
and there exists an isomorphism
\begin{equation}
\label{ga2}
\hat{T}(\wh{\mfS}_n^{\vee})\simeq \mbb{Z}_p/p^n\mbb{Z}_p(r)
\end{equation}
as $\mbb{Z}_p[G]$-modules.
We may say that $\wh{\mfS}^{\vee}$ (resp.\ $\wh{\mfS}_n^{\vee}$ ) 
is a dual $(\vphi,\hat{G})$-module
of $\wh{\mfS}$ (resp.\ $\wh{\mfS}_n$ )
since (\ref{ga}) and (\ref{ga2}) hold.

\begin{remark}
If 
$K_{p^{\infty}}\cap K_{\infty}=K$ (which is automatically hold in the case $p>2$), 
then $\hat{G}$-actions on $\wh{\mfS}^{\vee},\wh{\mfS}_n^{\vee}$
and $\wh{\mfS}_{\infty}^{\vee}$  
can be written explicitly as follows
(see Example 3.2.3 of \cite{Li3}):
If $K_{p^{\infty}}\cap K_{\infty}=K$,
we have $\hat{G}=G_{p^{\infty}}\rtimes H_K$ (see Lemma 5.1.2 in \cite{Li2}). 
Fixing a topological generator $\tau\in G_{p^{\infty}}$,
we define $\hat{G}$-actions on the above three modules 
by the relation $\tau(\mfrak{f}^r)=\hat{c}^r\cdot \mfrak{f}^r$.
Here $\hat{c}=\frac{c}{\tau(c)}=
\prod^{\infty}_{n=1}\vphi^n(\frac{E(u)}{\tau(E(u))}),\ 
c=\prod^{\infty}_{n=0}\vphi^n(\frac{\vphi(c_0^{-1}E(u))}{p})$.
Example 3.2.3 of \cite{Li3} says that $c\in A_{\mrm{cris}}^{\times}$ 
and $\hat{c}\in \whR^{\times}$.
It follows from straightforward calculations  that 
$\wh{\mfS}^{\vee}$ and 
$\wh{\mfS}_n^{\vee}$ 
are $(\vphi,\hat{G})$-modules of height $r$.
\end{remark}
\begin{lemma}
\label{lemmacar}
Let $A$ be a $\mfS$-algebra 
with characteristic coprime to $p$.
Let $\mfM\in \mrm{Mod}^r_{/\mfS_{\infty}}$ 
$($resp.\ $\mfM\in \mrm{Mod}^r_{/\mfS})$.
Then there exists a natural isomorphism:
\[
A\otimes_{\vphi, \mfS} \mfM^{\vee} 
\overset{\sim}{\longrightarrow} 
\mrm{Hom}_A(A\otimes_{\vphi, \mfS} \mfM, A_{\infty})\quad 
\mrm{if}\ \mfM\ {\it is\ killed\ by\ some\ power\ of}\ p,
\]
\[
(resp.\ \quad
A\otimes_{\vphi, \mfS} \mfM^{\vee} 
\overset{\sim}{\longrightarrow} 
\mrm{Hom}_A(A\otimes_{\vphi, \mfS} \mfM, A)\quad 
{\it if}\ \mfM\ {\it is\ free}).
\]
\end{lemma}
\begin{proof}
If $\mfM$ is free, the statement is clear.
If $p\mfM=0$, then we may regard $\mfM$ as a 
finite free $\mfS_1$-module 
and thus the statement is clear.
Suppose that $\mfM$ is a 
(general) torsion Kisin module
of height $r$. 
By Proposition \ref{torKi} of \cite{Li1},
there exists an extension of 
$\vphi$-modules
\[
0=\mfM_0\subset \mfM_1\subset \cdots \subset \mfM_n=\mfM 
\]
such that, for all $1\le i\le n$,
$\mfM_i/\mfM_{i-1}\in \mrm{Mod}^r_{/\mfS_{\infty}}$ and 
$\mfM_i/\mfM_{i-1}$ is a finite free $\mfS/p\mfS=k[\![u]\!]$-module.
Furthermore, we have $\mfM_i\in \mrm{Mod}^r_{/\mfS_{\infty}}$
by Lemma 2.3.1 in \cite{Li1}. 
We show that the natural map
\[
A\otimes_{\vphi, \mfS} \mfM_i^{\vee} 
\longrightarrow
\mrm{Hom}_A(A\otimes_{\vphi, \mfS} 
\mfM_i, A_{\infty}),\quad
a\otimes f\mapsto (a\otimes x\mapsto af(x))
\]
where $a\in A, f\in \mfM_i^{\vee}$ and  $x\in \mfM_i$,
is an isomorphism by induction for $i$.
For $i=0$, it is obvious. 
Suppose that the above map is an isomorphism for $i-1$. 
We have an exact sequence of $\mfS$-modules
\begin{equation}
\label{ex1}
0\to \mfM_{i-1}\to \mfM_i\to \mfM_i/\mfM_{i-1}\to 0. 
\end{equation}
By Corollary 3.1.5 of \cite{Li1}, we know that 
the sequence
\[
0\to (\mfM_i/\mfM_{i-1})^{\vee}\to 
\mfM_i^{\vee}\to \mfM_{i-1}^{\vee}\to 0. 
\]
is also an exact sequence of $\mfS$-modules.
Therefore, we have the following 
exact sequence of $A$-modules:
\begin{equation}
\label{ex2}
A\otimes_{\vphi,\mfS}(\mfM_i/\mfM_{i-1})^{\vee}\to 
A\otimes_{\vphi,\mfS}\mfM_i^{\vee}\to 
A\otimes_{\vphi,\mfS}\mfM_{i-1}^{\vee}\to 0. 
\end{equation}
On the other hand, the exact sequence  
(\ref{ex1}) induces an exact sequence of $A$-modules
\begin{equation}
\label{ex3}
0\to 
\mrm{Hom}_{A}(A\otimes_{\vphi, \mfS} \mfM_i/\mfM_{i-1}, A_{\infty})\to 
\mrm{Hom}_{A}(A\otimes_{\vphi, \mfS} \mfM_i, A_{\infty})\to 
\mrm{Hom}_{A}(A\otimes_{\vphi, \mfS} \mfM_{i-1}, A_{\infty}).
\end{equation}
Combining sequences (\ref{ex2}) and (\ref{ex3}), 
we obtain the following commutative 
diagram of A-modules:
\begin{center}
$\displaystyle \xymatrix@C=2mm@R=3mm{
 & A\otimes_{\vphi,\mfS}(\mfM_i/\mfM_{i-1})^{\vee}\ar[r] \ar[d] & 
A\otimes_{\vphi,\mfS}\mfM_i^{\vee}\ar[r] \ar[d]
& A\otimes_{\vphi,\mfS}\mfM_{i-1}^{\vee}\ar[r] \ar[d]& 0 \\
0\ar[r] & \mrm{Hom}_A(A\otimes_{\vphi, \mfS} \mfM_i/\mfM_{i-1}, A_{\infty})\ar[r] & 
\mrm{Hom}_A(A\otimes_{\vphi, \mfS} \mfM_i, A_{\infty})\ar[r]
& \mrm{Hom}_A(A\otimes_{\vphi, \mfS} \mfM_{i-1}, A_{\infty})  &  }$
\end{center}
where the two rows are exact. 
Furthermore, 
first and third columns are isomorphisms
by the induction hypothesis.
By the snake lemma, 
we obtain that the second column is an isomorphism, too.
\end{proof}

Let $\hat{\mfM}=(\mfM,\vphi_{\mfM},\hat{G})$ be a torsion 
(resp.\ free) weak $(\vphi,\hat{G})$-module of height $r$
and $(\mfM^{\vee},\vphi^{\vee}_{\mfM})$ 
the dual Kisin module of $(\mfM,\vphi_{\mfM})$.
By Lemma \ref{lemmacar},
we have isomorphisms
\begin{equation}
\label{act1}
\whR\otimes_{\vphi, \mfS} \mfM^{\vee} 
\overset{\sim}{\longrightarrow} 
\mrm{Hom}_{\wh{\mcal{R}}}(\wh{\mcal{R}}\otimes_{\vphi, \mfS} \mfM, \wh{\mfS}_{\infty}^{\vee})\quad 
\mrm{if}\ \mfM\ \mrm{is\ killed\ by\ some\ power\ of}\ p,
\end{equation}  
\begin{equation}
\label{act2}
\wh{\mcal{R}}\otimes_{\vphi, \mfS} \mfM^{\vee} 
\overset{\sim}{\longrightarrow} 
\mrm{Hom}_{\wh{\mcal{R}}}(\wh{\mcal{R}}\otimes_{\vphi, \mfS} \mfM, \wh{\mfS}^{\vee})\quad 
\mrm{if}\ \mfM\ \mrm{is\ free}.
\end{equation}  
We define a $\hat{G}$-action on
$\mrm{Hom}_{\wh{\mcal{R}}}(\wh{\mcal{R}}\otimes_{\vphi, \mfS} \mfM, \wh{\mfS}_{\infty}^{\vee})$
(resp.\ 
$\mrm{Hom}_{\wh{\mcal{R}}}(\wh{\mcal{R}}\otimes_{\vphi, \mfS} \mfM, \wh{\mfS}^{\vee})$)
by
\[
(\sigma.f)(x)=\sigma(f(\sigma^{-1}(x)))
\]
for $\sigma\in \hat{G}, x\in \wh{\mcal{R}}\otimes_{\vphi,\mfS}\mfM$ and 
$f\in \mrm{Hom}_{\wh{\mcal{R}}}(\wh{\mcal{R}}\otimes_{\vphi, \mfS} \mfM, \wh{\mfS}_{\infty}^{\vee})$
(resp.\ $f\in \mrm{Hom}_{\wh{\mcal{R}}}(\wh{\mcal{R}}\otimes_{\vphi, \mfS} \mfM, \wh{\mfS}^{\vee}))$
and equip a $\hat{G}$-action on $\wh{\mcal{R}}\otimes_{\vphi, \mfS} \mfM^{\vee}$
via an isomorphism (\ref{act1}) (resp.\ (\ref{act2})).

\begin{theorem}
\label{dual}
Let $\hat{\mfM}=(\mfM,\vphi_{\mfM},\hat{G})$ 
be a torsion (resp.\ free) weak $(\vphi,\hat{G})$-module 
of height $r$
and equip  a $\hat{G}$-action on 
$\wh{\mcal{R}}\otimes_{\vphi,\mfS} \mfM^{\vee}$ as the above.
Then the triple 
$\hat{\mfM}^{\vee}=(\mfM^{\vee}, \vphi^{\vee}_{\mfM}, \hat{G})$ 
is a torsion (resp.\ free) weak
$(\vphi,\hat{G})$-module of height $r$.
If $\hat{\mfM}$ is a $(\vphi,\hat{G})$-module of height $r$,
then $\hat{\mfM}^{\vee}$ is also.
\end{theorem}
\begin{definition}
Let $\hat{\mfM}$  be a weak $(\vphi,\hat{G})$-module (resp.\ a $(\vphi,\hat{G})$-module). 
We call $\hat{\mfM}^{\vee}$ as Theorem \ref{dual} 
the {\it Cartier dual of $\hat{\mfM}$}.
\end{definition}

To prove Theorem \ref{dual},
we need the following easy property for $\wh{\mcal{R}}_{\infty}=\wh{\mcal{R}}[1/p]/\wh{\mcal{R}}$. 

\begin{lemma}
\label{lemm}
$(1)$ For any integer $n$, we have 
\[
\whR[1/p]\cap p^nW(\mrm{Fr}R)=\whR\cap p^nW(R)=p^n\whR.
\]
\noindent
$(2)$ 
The following properties for an $a\in \whR[1/p]$ are equivalent:

$(\mrm{i})$ If $x\in \whR[1/p]$ satisfies that $ax=0$ in $\whRi$, then $x=0$ in $\whRi$.

$(\mrm{ii})$ $a\notin p\whR$.

$(\mrm{iii})$ $a\notin pW(R)$. 

$(\mrm{iv})$ $a\notin pW(\mrm{Fr}R)$. 

\end{lemma}

\begin{proof}
(1) The result follows from the relations
\[
\whR[1/p]\cap p^nW(\mrm{Fr}R)=
\whR[1/p]\cap (W(R)[1/p]\cap p^nW(\mrm{Fr}R))=
\whR[1/p]\cap p^nW(R)
\]
and 
\[
p^n\whR\subset \whR[1/p]\cap p^nW(R)\subset \whR_{K_0}\cap p^nW(R)=
p^n(\whR_{K_0}\cap W(R))=p^n\whR.
\]

\noindent
(2) The equivalence of (ii), (iii) and (iv) follows from the assertion (1).
Suppose the condition (iv) holds.
Take any $x\in \whR[1/p]$ such that $ax\in \whR$.
Then we have
\[
\frac{1}{a}\whR\cap \whR[1/p]\subset 
\frac{1}{a}W(\mrm{Fr}R)\cap W(\mrm{Fr}R)[1/p]\subset W(\mrm{Fr}R)
\]
since $a\notin pW(\mrm{Fr}R)$.
Thus we obtain
\[
x\in \frac{1}{a}\whR\cap \whR[1/p]=
\frac{1}{a}\whR\cap \whR[1/p]\cap W(\mrm{Fr}R)\subset 
\whR[1/p]\cap W(\mrm{Fr}R)=\whR,
\]
which implies the assertion (i) (the last equality follows from (1)).
Suppose the condition (ii) does not hold, that is, $a\in p\whR$.
Then $\whR[1/p]\cap \frac{1}{a}\whR\supset 
\frac{1}{p}\whR\supsetneq \whR$ and this implies that (i) does not hold.
\end{proof}

\begin{proof}[Proof of Theorem \ref{dual}]
We only prove the case where $\hat{\mfM}$ is a torsion $(\vphi,\hat{G})$-module
(the free case can be checked by almost all the same method).

We check the properties $(1)$ to $(5)$ 
of Definition \ref{Liumod} for $\hat{\mfM}^{\vee}$.
It is clear that $(1)$ and $(2)$ hold for $\hat{\mfM}^{\vee}$.
Take any $f\in \mfM^{\vee}$.
Regard $\mfM^{\vee}$ 
as a submodule of $\wh{\mcal{R}}\otimes_{\vphi, \mfS} \mfM^{\vee}$.
Then, in $\wh{\mcal{R}}\otimes_{\vphi, \mfS} \mfM^{\vee}$, 
we see that $f$ is equal to the map 
\[
\hat{f}\colon \wh{\mcal{R}}\otimes_{\vphi, \mfS} \mfM\to \wh{\mcal{R}}\cdot \mfrak{f}^r
\]
given by $a\otimes x\mapsto a\vphi(f(x))\cdot \mfrak{f}^r$ 
for $a\in \wh{\mcal{R}}$ and $x\in \mfM$.
Since $\mfM\subset (\wh{\mcal{R}}\otimes_{\vphi,\mfS} \mfM)^{H_K}$, 
we have 
\begin{align*}
(\sigma.\hat{f})(a\otimes x)&=\sigma(\hat{f}(\sigma^{-1}(a\otimes x)))
=\sigma(\hat{f}(\sigma^{-1}(a)(1\otimes x)))
=\sigma((\sigma^{-1}(a)\hat{f}(1\otimes x)))\\
&=a\sigma(\hat{f}(1\otimes x))
=a\sigma(\vphi(f(x))\cdot \mfrak{f}^r)
=a\vphi(f(x))\cdot \mfrak{f}^r
=\hat{f}(a\otimes x).
\end{align*}
for any $a\in \wh{\mcal{R}}, x\in \mfM$ and $\sigma\in H_K$.
This implies $\mfM^{\vee}\subset (\wh{\mcal{R}}\otimes_{\vphi,\mfS} \mfM^{\vee})^{H_K}$
and hence $(4)$ holds for $\hat{\mfM}^{\vee}$.
Check the property $(5)$, that is, the condition that 
$\hat{G}$ acts trivially on $\hat{\mfM}/I_+\hat{\mfM}$.
By Lemma \ref{lemmacar}, we know that 
there exists the following natural isomorphism:
\[
\wh{\mcal{R}}\otimes_{\vphi, \mfS} \mfM^{\vee}/
I_+(\wh{\mcal{R}}\otimes_{\vphi, \mfS} \mfM^{\vee}) 
\overset{\sim}{\longrightarrow} 
\mrm{Hom}_{\wh{\mcal{R}}}
(\wh{\mcal{R}}\otimes_{\vphi, \mfS} \mfM/
I_+(\wh{\mcal{R}}\otimes_{\vphi, \mfS} \mfM), 
\wh{\mfS}^{\vee}_{\infty}/I_+\wh{\mfS}^{\vee}_{\infty}),
\]
which is in fact $\hat{G}$-equivalent by the definition of $\hat{G}$-action on 
$\wh{\mcal{R}}\otimes_{\vphi, \mfS} \mfM^{\vee}$.
Since $\hat{G}$ acts on $\wh{\mcal{R}}\otimes_{\vphi, \mfS} \mfM/
I_+(\wh{\mcal{R}}\otimes_{\vphi, \mfS} \mfM)$ and 
$\wh{\mfS}^{\vee}_{\infty}/I_+\wh{\mfS}^{\vee}_{\infty}$
trivially,
we obtain the desired result.

Finally we prove the property $(3)$ for $\hat{\mfM}^{\vee}$.
First we note that, if we take any  $f\in \mfM^{\vee}=\mrm{Hom}_{\mfS}(\mfM, \mfS_{\infty})$
and regard $f$ as a map which has values in $\mfS_{\infty}^{\vee}$, 
then we have 
\begin{equation}
\label{aaa}
\vphi^{\vee}(f)\circ \vphi_{\mfM}=\vphi^{\vee}\circ f\colon \mfM\to \mfS_{\infty}^{\vee}.
\end{equation}
Recall that there exists a natural isomorphism
\[
\wh{\mcal{R}}\otimes_{\vphi, \mfS} \mfM^{\vee} 
\simeq \mrm{Hom}_{\wh{\mcal{R}}}(\wh{\mcal{R}}\otimes_{\vphi, \mfS} \mfM, \wh{\mfS}_{\infty}^{\vee})
\]
by Lemma \ref{lemmacar}. 
We equip $\mrm{Hom}_{\wh{\mcal{R}}}(\wh{\mcal{R}}\otimes_{\vphi, \mfS} \mfM, \wh{\mfS}_{\infty}^{\vee})$
with a $\vphi$-structure $\vphi^{\vee}$ 
via this isomorphism.
Then it is enough to show that 
$\sigma \vphi^{\vee}=\vphi^{\vee}\sigma$ on 
$\mrm{Hom}_{\wh{\mcal{R}}}(\wh{\mcal{R}}\otimes_{\vphi, \mfS} \mfM, \wh{\mfS}_{\infty}^{\vee})$ 
for any $\sigma\in \hat{G}$.
Take any 
$\hat{f}\in \mrm{Hom}(\wh{\mcal{R}}\otimes_{\vphi, \mfS} \mfM, \wh{\mfS}_{\infty}^{\vee})$
and consider the following diagram:
\begin{equation}
\label{com2}
\displaystyle \xymatrix{
  \wh{\mcal{R}}\otimes_{\vphi,\mfS} \mfM 
  \ar^{\vphi_{\hat{\mfM}}}[r] \ar_{\hat{f}}[d]
& \wh{\mcal{R}}\otimes_{\vphi,\mfS} \mfM 
  \ar^{\vphi^{\vee}(\hat{f})}[d]\\
  \wh{\mfS}_{\infty}^{\vee}
  \ar_{\vphi^{\vee}}[r]
& \wh{\mfS}_{\infty}^{\vee}
. }
\end{equation}
By (\ref{aaa}), 
we obtain that the diagram (\ref{com2}) is also commutative. 
To check the relation $\sigma(\vphi^{\vee}(\hat{f}))=\vphi^{\vee}(\sigma(\hat{f}))$,
it suffices to show that 
$\sigma(\vphi^{\vee}(\hat{f}))(\vphi_{\hat{\mfM}}(x))=
\vphi^{\vee}(\sigma(\hat{f}))(\vphi_{\hat{\mfM}}(x))$
for any $x\in \wh{\mcal{R}}\otimes_{\vphi, \mfS} \mfM$ since 
$\mfM$ is of finite $E(u)$-height and, for any $a\in \wh{\mcal{R}}_{\infty}$, 
$\vphi(E(u))a=0$ if and only if $a=0$ by Lemma \ref{lemm}.
By (\ref{com2}), we have
\[
\sigma(\vphi^{\vee}(\hat{f}))(\vphi_{\hat{\mfM}}(x))
=\sigma(\vphi^{\vee}(\hat{f})(\sigma^{-1}(\vphi_{\hat{\mfM}}(x))))
=\sigma(\vphi^{\vee}(\hat{f})(\vphi_{\hat{\mfM}}(\sigma^{-1}(x))))
=\sigma(\vphi^{\vee}(\hat{f}(\sigma^{-1}(x)))).
\]
By replacing $\hat{f}$ with $\sigma(\hat{f})$ in the diagram (\ref{com2}), we have 
\[
\vphi^{\vee}(\sigma(\hat{f}))(\vphi_{\hat{\mfM}}(x))
=\vphi^{\vee}(\sigma(\hat{f}))(x)
=\vphi^{\vee}(\sigma(\hat{f}(\sigma^{-1}(x))
=\sigma(\vphi^{\vee}(\hat{f}(\sigma^{-1}(x))))
\]
and this finishes the proof.
\end{proof}

\subsection{Cartier duality theorem}

Let $\hat{\mfM}$ be a weak $(\vphi,\hat{G})$-module of height $r$.
We have natural pairings 
\begin{equation}
\label{du1}
\langle \cdot , \cdot \rangle \colon (\whR\otimes_{\vphi, \mfS}\mfM) 
\times (\whR\otimes_{\vphi, \mfS}\mfM^{\vee})\to
\wh{\mfS}_{\infty}^{\vee}
\quad \mrm{if}\ \mfM\ \mrm{is\ killed\ by\ some\ power\ of}\ p
\end{equation}
and
\begin{equation}
\label{du2}
\langle \cdot , \cdot \rangle \colon (\whR\otimes_{\vphi, \mfS}\mfM) 
\times (\whR\otimes_{\vphi, \mfS}\mfM^{\vee})\to
\wh{\mfS}^{\vee}
\quad \mrm{if}\ \mfM\ \mrm{is\ free}.
\end{equation}
It is not difficult to see that these parings commute with Frobenii and $\hat{G}$-actions.

Here we describe the Cartier duality theorem for $(\vphi,\hat{G})$-modules.

\begin{theorem}[Cartier duality theorem]
\label{DualLi}
Let $\hat{\mfM}$ be a weak $(\vphi,\hat{G})$-module $($resp.\ a $(\vphi,\hat{G})$-module$)$
of height $r$.

\noindent
$(1)$ The assignment 
$\hat{\mfrak{M}}\mapsto \hat{\mfrak{M}}^{\vee}$
is an anti-equivalence on the category 
of torsion weak $(\varphi, \hat{G})$-modules
$($resp.\ free weak $(\varphi, \hat{G})$-modules,
resp.\ torsion $(\varphi, \hat{G})$-modules,
resp.\ free weak $(\varphi, \hat{G})$-modules$)$ 
and a natural map 
$\hat{\mfrak{M}}\to (\hat{\mfrak{M}}^{\vee})^{\vee}$
is an isomorphism.
 
\noindent
$(2)$ Parings $($\ref{du1}$)$ and $($\ref{du2}$)$ are perfect.

\noindent
$(3)$ Taking a dual preserves a short exact sequence  
of torsion weak $(\varphi, \hat{G})$-modules
$($resp.\ free weak $(\varphi, \hat{G})$-modules,
resp.\ torsion $(\varphi, \hat{G})$-modules,
resp.\ free weak $(\varphi, \hat{G})$-modules$)$. 
\end{theorem}

\begin{proof}
By Theorem \ref{DualKi} (3),
we have already known that a natural map
$\mfM\to (\mfM^{\vee})^{\vee}$
is an isomorphism as $\vphi$-modules.
Furthermore, 
straightforward calculations show that 
the map $\mfM\to (\mfM^{\vee})^{\vee}$ is compatible 
with Galois action after tensoring $\whR$.
Thus we obtain that 
$\hat{\mfrak{M}}\to (\hat{\mfrak{M}}^{\vee})^{\vee}$
is an isomorphism, and the assertion (1) follows immediately.
The assertion (3) follows from Theorem \ref{DualKi} (5). 
Consequently,
we have to show the assertion (2).
We leave the proof to the next section.
\end{proof}

\subsection{Compatibility with Galois actions}
The goal of this subsection is to prove the following
which is equivalent to Theorem \ref{DualLi} (2):
\begin{proposition}
\label{Gal}
Let $\hat{\mfM}$ be a weak $(\vphi,\hat{G})$-module.
Then we have
\begin{equation}
\hat{T}(\hat{\mfM}^{\vee})\simeq \hat{T}^{\vee}(\hat{\mfM})(r)
\end{equation}
as $\mbb{Z}_p[G]$-modules 
where $\hat{T}^{\vee}(\hat{\mfM})$ is the dual representation of 
$\hat{T}(\hat{\mfM})$ and
the symbol ``$(r)$'' is for the $r$-th Tate twist.
\end{proposition}

First we construct a covariant functor
for the category of weak $(\vphi, \hat{G})$-modules. 
Recall that, if $\hat{\mfM}=(\mfM,\vphi_{\mfM},\hat{G})$ 
is a weak $(\vphi, \hat{G})$-module,
we often abuse of notations by 
denoting $\hat{\mfM}$ the underlying 
module $\wh{\mcal{R}}\otimes_{\vphi, \mfS} \mfM$.
\begin{proposition}
\label{prop1}
Let $\hat{\mfM}$ be a weak $(\vphi, \hat{G})$-module.
Then the natural $W(\mrm{Fr}R)$-linear map
\begin{equation}
\label{fund1}
W(\mrm{Fr}R)\otimes_{\mbb{Z}_p}
(W(\mrm{Fr}R)\otimes_{\wh{\mcal{R}}} \hat{\mfM})^{\vphi=1}
\to 
W(\mrm{Fr}R)\otimes_{\wh{\mcal{R}}} 
\hat{\mfM},\quad a\otimes x\mapsto ax, 
\end{equation}
for any  $a\in W(\mrm{Fr}R)$ and $x\in(W(\mrm{Fr}R)\otimes_{\wh{\mcal{R}}} \hat{\mfM})^{\vphi=1}$,
is an isomorphism, which is compatible 
with $\vphi$-structures and $G$-actions.
\end{proposition}

\begin{proof}
A non-trivial assertion of this proposition is 
only the bijectivity of the map (\ref{fund1}).
First we note the following natural $\vphi$-equivariant isomorphisms:
\begin{align*}
        W(\mrm{Fr}R)\otimes_{\wh{\mcal{R}}} \hat{\mfM}
&\simeq W(\mrm{Fr}R)\otimes_{\vphi,\mfS} \mfM \\
&\simeq W(\mrm{Fr}R)\otimes_{\cO}(\cO\otimes_{\vphi,\mfS}M)\\
&\overset{1\otimes \vphi^{\ast}_M}{\longrightarrow} W(\mrm{Fr}R)\otimes_{\cO} M
\end{align*}
where $M=\cO\otimes_{\mfS} \mfM$ 
is the \'etale $\vphi$-module corresponding to $\mfM$.
Here the bijectivity of $1\otimes \vphi^{\ast}_M$, where $\vphi^{\ast}_M$ 
is the $\cO$-linearization of $\vphi_M$, 
follows from the \'etaleness of $M$. 
Combining the above isomorphisms and the relation (\ref{Fon1}),
we obtain the following natural $\vphi$-equivalent bijective  maps
\begin{equation}
\label{hosoku}
W(\mrm{Fr}R)\otimes_{\wh{\mcal{R}}} \hat{\mfM}
\overset{\sim}{\longrightarrow}  
W(\mrm{Fr}R)\otimes_{\cO} M
\overset{\sim}{\longleftarrow}
W(\mrm{Fr}R)\otimes_{\mbb{Z}_p}
(\wh{\cOur}\otimes_{\cO} M)^{\vphi=1} 
\end{equation}
and hence we obtain 
\begin{equation}
\label{is1}
(W(\mrm{Fr}R)\otimes_{\wh{\mcal{R}}} \hat{\mfM})^{\vphi=1}
\simeq
(\wh{\cOur}\otimes_{\cO} M)^{\vphi=1}. 
\end{equation}
By (\ref{hosoku}) and (\ref{is1}),
we obtain an isomorphism
\[
W(\mrm{Fr}R)\otimes_{\mbb{Z}_p}
(W(\mrm{Fr}R)\otimes_{\wh{\mcal{R}}} \hat{\mfM})^{\vphi=1}
\overset{\sim}{\longrightarrow} 
W(\mrm{Fr}R)\otimes_{\wh{\mcal{R}}} \hat{\mfM}
\]  
and the desired result follows from the fact that this isomorphism coincides with the natural map (\ref{fund1}).
\end{proof}

For any weak  $(\vphi, \hat{G})$-module $\hat{\mfM}$,
we set
\[
\hat{T}_{\ast}(\hat{\mfM})=(W(\mrm{Fr}R)\otimes_{\hat{\mcal{R}}}\hat{\mfM})^{\vphi=1}.
\]
Since the Frobenius action on 
$W(\mrm{Fr}R)\otimes_{\hat{\mcal{R}}}\hat{\mfM}$ 
commutes with $G$-action,
we see that $G$ acts on  $\hat{T}_{\ast}(\hat{\mfM})$ stable. 
We have shown in the proof of Proposition \ref{prop1} (see (\ref{is1}))
that 
\[
\hat{T}_{\ast}(\hat{\mfM})\simeq \mcal{T}_{\ast}(M)
\]
as $\mbb{Z}_p[G_{\infty}]$-modules
for $M=\cO\otimes_{\mfS} \mfM$ (the functor $\mcal{T}_{\ast}$ is defined in Section 2.2).
In particular, 
if $\hat{\mfM}$ is free and $d=\mrm{rank}_{\mfS}(\mfM)$, 
$\hat{T}_{\ast}(\hat{\mfM})$ is free of 
rank $d$ as a $\mbb{Z}_p$-module.
The association $\hat{\mfM}\mapsto \hat{T}_{\ast}(\hat{\mfM})$ 
is a covariant functor from the category of 
$(\vphi, \hat{G})$-modules of height $r$
to the category $\mrm{Rep}_{\mbb{Z}_p}(G)$ 
of finite $\mbb{Z}_p[G]$-modules.
By the exactness of the functor $\mcal{T}_{\ast}$,
the functor $\hat{T}_{\ast}$ is an exact functor.

\begin{corollary}
\label{cov}
The $\mbb{Z}_p$-representation $\hat{T}_{\ast}(\hat{\mfM})$ of $G$ is the dual of 
$\hat{T}(\hat{\mfM})$, that is,
\[
\hat{T}^{\vee}(\hat{\mfM})\simeq \hat{T}_{\ast}(\hat{\mfM})
\]
as $\mbb{Z}_p[G]$-modules where
$\hat{T}^{\vee}(\hat{\mfM})$
is the dual representation of $\hat{T}(\hat{\mfM})$. 
\end{corollary}

\begin{proof}
Suppose $\hat{\mfM}$ is killed by some power of $p$.
By Proposition \ref{prop1} and 
the relation $W(\mrm{Fr}R)_{\infty}^{\vphi=1}=\mbb{Q}_p/\mbb{Z}_p$,
we have 
\begin{align*}
\mrm{Hom}_{\mbb{Z}_p}(\hat{T}_{\ast}(\hat{\mfM}),\mbb{Q}_p/\mbb{Z}_p)
&\simeq 
\mrm{Hom}_{W(\mrm{Fr}R),\vphi}
(W(\mrm{Fr}R)\otimes_{\mbb{Z}_p}
(W(\mrm{Fr}R)\otimes_{\wh{\mcal{R}}} \hat{\mfM})^{\vphi=1},
W(\mrm{Fr}R)_{\infty})\\
&\simeq
\mrm{Hom}_{W(\mrm{Fr}R),\vphi}(
W(\mrm{Fr}R)\otimes_{\wh{\mcal{R}}} \hat{\mfM},
W(\mrm{Fr}R)_{\infty})\\
&\simeq
\mrm{Hom}_{\wh{\mcal{R}},\vphi}(
\hat{\mfM},
W(\mrm{Fr}R)_{\infty})=\hat{T}(\hat{\mfM}).
\end{align*}
The last equality follows from 
the proof of Lemma 3.1.1 of \cite{Li3},
but we include
a proof here for the sake of completeness.
Take any 
$h\in \mrm{Hom}_{\wh{\mcal{R}},\vphi}(
\hat{\mfM},
W(\mrm{Fr}R)_{\infty})$.
It is enough to prove that 
$h$ has in fact values in $W(R)_{\infty}$.
Put $g=h|_{\mfM}$.
Since $g$ is a $\vphi(\mfS)$-linear morphism 
from $\mfM$ to $W(R)_{\infty}=\vphi(W(R)_{\infty})$,
there exists a $\mfS$-linear morphism 
$\mfrak{g}\colon \mfM\to W(\mrm{Fr}R)_{\infty}$ 
such that $\vphi(\mfrak{g})=g$.
Furthermore, we see that $\mfrak{g}$ is $\vphi$-equivariant.
Note that $\mfrak{g}(\mfM)\subset W(\mrm{Fr}R)_{\infty}$ 
is a $\mfS$-finite type $\vphi$-stable submodule
and of $E(u)$-height $r$.
By \cite{Fo}, Proposition B.1.8.3, we have 
$\mfrak{g}(\mfM)\subset \mfS_{\infty}^{\mrm{ur}}$. 
Since 
\[
h(a\otimes x)=a\vphi(\mfrak{g}(x))
\]
for any $a\in \wh{\mcal{R}}$ and $x\in \mfM$, we obtain that 
$h$ has values in $W(R)_{\infty}$.

The case $\hat{\mfM}$ is free, 
we obtain the desired result 
by the same proof as above 
if we replace $W(\mrm{Fr}R)_{\infty}$ 
(resp.\ $\mbb{Q}_p/\mbb{Z}_p$) with 
$W(\mrm{Fr}R)$ (resp.\ $\mbb{Z}_p$). 
\end{proof}

In the rest of this subsection, 
we prove Proposition \ref{Gal}.
We only prove the case where 
$\mfM$ is killed by $p^n$ for some integer $n\ge 1$
(we can prove the free case 
by an analogous way and the free case is 
easier than the torsion case).

First we consider natural pairings
\begin{equation}
\label{pair1}
\langle \cdot , \cdot \rangle \colon \mfM \times \mfM^{\vee}\to
\mfS_n^{\vee}
\end{equation}
and 
\begin{equation}
\label{pair2}
\langle \cdot , \cdot \rangle \colon M\times M^{\vee}\to \cOdual_n
\end{equation}
which are perfect and compatible with $\vphi$-structures.
Here $M=\cO\otimes_{\mfS} \mfM$ is the \'etale $\vphi$-module
corresponding to $\mfM$.
We can extend the pairing (\ref{pair2}) to the $\vphi$-equivalent perfect pairing  
\[
(\cOur\otimes_{\cO} M)
\times (\cOur\otimes_{\cO}M^{\vee})
\to \cOurdual_n.
\]
Since the above pairing is $\vphi$-equivariant and 
$(\cOurdual_n)^{\vphi=1}\simeq \mbb{Z}_p/p^n\mbb{Z}_p(-r)$,
we have a pairing 
\begin{equation}
\label{pair23}
(\cOur\otimes_{\cO} M)^{\vphi=1}
\times (\cOur\otimes_{\cO}M^{\vee})^{\vphi=1}
\to \mbb{Z}_p/p^n\mbb{Z}_p(-r)
\end{equation}
compatible with $G_{\infty}$-actions.
Liu showed in the proof of Lemma 3.1.2 in \cite{Li1} 
that this pairing is perfect.
By a similar way, we have the following paring
\begin{equation}
\label{pair233}
(W(\mrm{Fr}R)\otimes_{\cO} M)^{\vphi=1}
\times (W(\mrm{Fr}R)\otimes_{\cO}M^{\vee})^{\vphi=1}
\to \mbb{Z}_p/p^n\mbb{Z}_p(-r).
\end{equation} 
On the other hand,
the pairing (\ref{pair1}) induces a pairing 
\begin{equation}
\label{pair3}
(\wh{\mcal{R}}\otimes_{\vphi, \mfS}\mfM) 
\times (\wh{\mcal{R}}\otimes_{\vphi, \mfS}\mfM^{\vee})\to
\wh{\mfS}_n^{\vee}.
\end{equation}
We can extend the pairing (\ref{pair3}) to the $\vphi$-equivalent perfect pairing
\[
(W(\mrm{Fr}R)\otimes_{\wh{\mcal{R}}}(\wh{\mcal{R}}\otimes_{\vphi, \mfS}\mfM)) 
\times (W(\mrm{Fr}R)\otimes_{\wh{\mcal{R}}}(\wh{\mcal{R}}\otimes_{\vphi, \mfS}\mfM^{\vee}))\to
W(\mrm{Fr}R)\otimes_{\wh{\mcal{R}}} \wh{\mfS}_n^{\vee}.
\]
Since the above pairing is $\vphi$-equivariant and 
$(W(\mrm{Fr}R)\otimes_{\wh{\mcal{R}}} \wh{\mfS}_n^{\vee})^{\vphi=1}\simeq \mbb{Z}_p/p^n\mbb{Z}_p(-r)$,
we have a pairing 
\begin{equation}
\label{pair4}
(W(\mrm{Fr}R)\otimes_{\wh{\mcal{R}}}(\wh{\mcal{R}}\otimes_{\vphi, \mfS}\mfM))^{\vphi=1} 
\times (W(\mrm{Fr}R)\otimes_{\wh{\mcal{R}}}(\wh{\mcal{R}}\otimes_{\vphi, \mfS}\mfM^{\vee}))^{\vphi=1}\to
\mbb{Z}_p/p^n\mbb{Z}_p(-r)
\end{equation}
compatible with $G$-actions.
Since we have the natural isomorphism 
$\cOur\otimes_{\mbb{Z}_p}(\cOur\otimes_{\cO} M)^{\vphi=1}
\overset{\sim}{\longrightarrow}\cOur\otimes_{\cO} M$,
we obtain the $\vphi$-equivariant isomorphisms
\begin{equation}
\label{isom1}
W(\mrm{Fr}R)\otimes_{\wh{\mcal{R}}}\hat{\mfM} \overset{\sim}{\longrightarrow}
W(\mrm{Fr}R)\otimes_{\cO} M \overset{\sim}{\longleftarrow}
W(\mrm{Fr}R)\otimes_{\mbb{Z}_p}(\cOur\otimes_{\cO} M)^{\vphi=1}.
\end{equation}
Therefore, combining (\ref{pair23}), (\ref{pair233}), (\ref{pair4}) and (\ref{isom1}), 
we have the following diagram
\begin{center}
$\displaystyle \xymatrix{
  \! \! \!(W(\mrm{Fr}R)\otimes_{\wh{\mcal{R}}}\hat{\mfM})^{\vphi=1}\! \! \!
  \ar[d]_{\wr} 
& \! \! \! \times \! \! \!
& \! \! \!(W(\mrm{Fr}R)\otimes_{\wh{\mcal{R}}}\hat{\mfM}^{\vee})^{\vphi=1}
  \ar[d]_{\wr} \ar[r]
& \mbb{Z}_p/p^n\mbb{Z}_p(-r)
  \ar@{=}[d] \\
  \! \! \!(W(\mrm{Fr}R)\otimes_{\cO}M)^{\vphi=1}\! \! \!
& \! \! \! \times \! \! \!
& \! \! \!(W(\mrm{Fr}R)\otimes_{\cO}M^{\vee})^{\vphi=1}
  \ar[r]
& \mbb{Z}_p/p^n\mbb{Z}_p(-r) \\
  (\cOur\otimes_{\cO} M)^{\vphi=1}
  \ar[u]^{\wr}
& \! \! \! \times \! \! \!
& \! \! \!(\cOur\otimes_{\cO} M^{\vee})^{\vphi=1}
  \ar[u]^{\wr} \ar[r]
& \mbb{Z}_p/p^n\mbb{Z}_p(-r) 
\ar@{=}[u]
. }$
\end{center}
It is a straightforward calculation to check that 
the above diagram is commutative.
Since the bottom pairing is perfect, 
we see that the top pairing is also perfect.
This implies $\hat{T}_{\ast}(\hat{\mfM}^{\vee})\simeq \hat{T}_{\ast}(\hat{\mfM})(-r)$ and 
therefore, we have the desired result by Corollary \ref{cov}.


\section{Category of representations arising from torsion $(\vphi,\hat{G})$-modules} 
\subsection{Relations between $(\vphi,\hat{G})$-modules and their representations} 

Select a $\mfrak{t}\in \mfS^{\mrm{ur}}$ such that $\mfrak{t}\notin p\mfS^{\mrm{ur}}$
and $\vphi(\mfrak{t})=c_0^{-1}E(u)\mfrak{t}$ where $pc_0=E(0)$.
Such $\mfrak{t}$ is unique up to units of $\mbb{Z}_p$, see Example 2.3.5 in \cite{Li1} for details.

Let $\mfM$ be in 
$\mrm{Mod}^{r}_{/\mfS_{\infty}}$.
We construct a map $\iota_{\mfS}$ for $\mfM$ which connects 
$\mfM$ to $T_{\mfS}(\mfM)$ (cf.\ \cite{Li1}, Section 3.2). 
First observe that there exists a natural isomorphism of 
$\mbb{Z}_p[G_{\infty}]$-modules
\[
T_{\mfS}(\mfM)=\mrm{Hom}_{\mfS,\vphi}(\mfM,\mfS^{\mrm{ur}}_{\infty})\simeq
\mrm{Hom}_{\mfS^{\mrm{ur}},\vphi}(\mfS^{\mrm{ur}}\otimes_{\mfS}\mfM,\mfS^{\mrm{ur}}_{\infty})
\]
where $G_{\infty}$ acts on 
$\mrm{Hom}_{\mfS^{\mrm{ur}},\vphi}(\mfS^{\mrm{ur}}\otimes_{\mfS}\mfM,\mfS^{\mrm{ur}}_{\infty})$
by $(\sigma.f)(x)=\sigma(f(\sigma^{-1}(x)))$ for 
$\sigma\in G_{\infty},
f\in \mrm{Hom}_{\mfS^{\mrm{ur}},\vphi}(\mfS^{\mrm{ur}}\otimes_{\mfS}\mfM,\mfS^{\mrm{ur}}_{\infty}),
x\in \mfS^{\mrm{ur}}\otimes_{\mfS}\mfM$
and $G_{\infty}$ acts on $\mfM$ trivial.
Thus we can define a morphism 
$\iota'_{\mfS}\colon \mfS^{\mrm{ur}}\otimes_{\mfS} \mfM\to 
\mrm{Hom}_{\mbb{Z}_p}(T_{\mfS}(\mfM),\mfS^{\mrm{ur}}_{\infty})$
by
\[
x\mapsto (f\mapsto f(x)),\quad 
x\in \mfS^{\mrm{ur}}\otimes_{\mfS}\mfM, f\in T_{\mfS}(\mfM). 
\]
Since $T_{\mfS}(\mfM)\simeq \oplus_{i\in I}\mbb{Z}_p/p^{n_i}\mbb{Z}_p$ 
as finite $\mbb{Z}_p$-modules,
we have a natural isomorphism 
$\mrm{Hom}_{\mbb{Z}_p}(T_{\mfS}(\mfM),\mfS^{\mrm{ur}}_{\infty})\simeq 
\mfS^{\mrm{ur}}\otimes_{\mbb{Z}_p}T^{\vee}_{\mfS}(\mfM)$
where $T^{\vee}_{\mfS}(\mfM)=\mrm{Hom}_{\mbb{Z}_p}(T_{\mfS}(\mfM),\mbb{Q}_p/\mbb{Z}_p)$
is the dual representation of $T_{\mfS}(\mfM)$.
Composing this isomorphism with $\iota'_{\mfS}$, we obtain a map 
\[
\iota_{\mfS}\colon \mfS^{\mrm{ur}}\otimes_{\mfS} \mfM\to 
\mfS^{\mrm{ur}}\otimes_{\mbb{Z}_p}T^{\vee}_{\mfS}(\mfM).
\]
\noindent
For $\mfM\in \mrm{Mod}^{r}_{/\mfS}$, 
we also construct $\iota_{\mfS}\colon \mfS^{\mrm{ur}}\otimes_{\mfS} \mfM\to 
\mfS^{\mrm{ur}}\otimes_{\mbb{Z}_p}T^{\vee}_{\mfS}(\mfM)$ by the same way
except only replacing $\mfS^{\mrm{ur}}_{\infty}$ with $\mfS^{\mrm{ur}}$.

\begin{lemma}
\label{rel1}
Let $A$ be a ring with $\mfS^{\mrm{ur}}\subset A \subset W(\mrm{Fr}R)$
which yields a ring extension $A_1 \subset \mrm{Fr}R$.
Let $\mfM$ be in 
$\mrm{Mod}^{r}_{/\mfS_{\infty}}$ or
$\mrm{Mod}^{r}_{/\mfS}$.
Let $\iota_{\mfS}$ be as above.

\noindent
$(1)$ $\iota_{\mfS}$  is $G_{\infty}$-equivalent and  
$\vphi$-equivalent.
Furthermore, 
$A\otimes_{\mfSur}\iota_{\mfS}$ is injective.

\noindent
$(2)$ If $r<\infty$, 
then $\mfrak{t}^r(A\otimes_{\mbb{Z}_p} T_{\mfS}^{\vee}(\mfM))\subset 
(A\otimes_{\mfSur}\iota_{\mfS})(A\otimes_{\mfS} \mfM)$.
If $r=\infty$,
then $\mfrak{t}^{r'}(A\otimes_{\mbb{Z}_p} T_{\mfS}^{\vee}(\mfM))\subset 
(A\otimes_{\mfSur}\iota_{\mfS})(A\otimes_{\mfS} \mfM)$ for $r'>0$
such that $\mfM$ is of height $r'$.

\noindent
$(3)$ The map 
\[
W(\mrm{Fr}R)\otimes_{\mfS^{\mrm{ur}}} \iota_{\mfS}\colon 
W(\mrm{Fr}R)\otimes_{\mfS} \mfM\to 
W(\mrm{Fr}R)\otimes_{\mbb{Z}_p} T_{\mfS}^{\vee}(\mfM)
\] 
is bijective. 
\end{lemma}

\begin{proof}
We may suppose that $r<\infty$.
The assertion that $\iota_{\mfS}$  is $G_{\infty}$-equivalent and 
$\vphi$-equivalent is a result of Theorem 3.2.2 in \cite{Li1}.
Liu showed in {\it loc}.\ {\it cit},\ that there exists a map 
$\iota^{\vee}_{\mfS}\colon \mfS^{\mrm{ur}}
\otimes_{\mbb{Z}_p} T^{\vee}_{\mfS}(\mfM)\to \mfSur\otimes_{\mfS}\mfM$
such that 
$\iota^{\vee}_{\mfS}\circ \iota_{\mfS}=\mfrak{t}^r$, in particular,
$(A\otimes_{\mfS^{\mrm{ur}}}  \iota^{\vee}_{\mfS})\circ 
(A\otimes_{\mfS^{\mrm{ur}}} \iota_{\mfS})=\mfrak{t}^r$.
Moreover, in the proof of {\it loc}.\ {\it cit},\
Liu also showed that the composite
$(\cOur\otimes_{\mfSur} \iota_{\mfS})\circ 
(\cOur\otimes_{\mfSur} \iota^{\vee}_{\mfS})\colon 
\cOur\otimes_{\mfSur} (\mfSur\otimes_{\mbb{Z}_p} T^{\vee}_{\mfS}(\mfM))
\to \cOur\otimes_{\mfSur} (\mfSur\otimes_{\mbb{Z}_p} T^{\vee}_{\mfS}(\mfM))$
is equal to the map $\cOur\otimes_{\mfSur}(\mfrak{t}^r\otimes_{\mbb{Z}_p} \mrm{Id})$.
Hence we obtain 
$\iota_{\mfS}\circ \iota^{\vee}_{\mfS}=\mfrak{t}^r$
and then the assertion (2) follows.
We show the injectivity of $A\otimes_{\mfS^{\mrm{ur}}}  \iota_{\mfS}$.
Since $(A\otimes_{\mfS^{\mrm{ur}}}  \iota^{\vee}_{\mfS})\circ 
(A\otimes_{\mfS^{\mrm{ur}}} \iota_{\mfS})=\mfrak{t}^r$,
if $\mfM$ is free over $\mfS$, we see that  $A\otimes_{\mfS^{\mrm{ur}}}  \iota_{\mfS}$ is 
injective.
Next we suppose that $\mfM$ is killed by $p$.
In this case, the proof is almost the same as the free case, 
except one need to note that 
$\mfM$ is free as a $k[\![u]\!]$-module, $\mfrak{t}\not= 0$ in $A_1$ 
(since $\mfS^{\mrm{ur}}_1\subset A_1$; see Remark \ref{subsets}) 
and $A_1$ is a domain 
(since $A_1\subset \mrm{Fr}R$).
Suppose that  $\mfM$
is killed by some power of $p$.
By Proposition \ref{torKi} (4) and Remark \ref{suc},
there exists an extension 
\[
0=\mfM_0\subset \mfM_1\subset \cdots \subset\mfM_k=\mfM
\]
in $\mrm{Mod}^r_{/\mfS_{\infty}}$ such that 
$\mfM_i,\ \mfM_{i+1}/\mfM_i\in\mrm{Mod}^r_{/\mfS_{\infty}}$ and $\mfM_{i+1}/\mfM_i$
is a finite free $k[\![u]\!]$-module.
We have a commutative diagram
\begin{center}
$\displaystyle \xymatrix{
0\ar[r] & A\otimes_{\mfS} \mfM_{i-1} \ar[r]  \ar_{A\otimes_{\mrm{\mfS^{\mrm{ur}}}}\iota_{\mfS,i-1}}[d]
& A\otimes_{\mfS} \mfM_i \ar[r] \ar_{A\otimes_{\mrm{\mfS^{\mrm{ur}}}}\iota_{\mfS,i}}[d]
& A\otimes_{\mfS} \mfM_i/\mfM_{i-1} \ar[r] \ar_{A\otimes_{\mrm{\mfS^{\mrm{ur}}}}\iota_{\mfS,i,i-1}}[d]
& 0 \\
0\ar[r] & A\otimes_{\mbb{Z}_p} T^{\vee}_{\mfS}(\mfM_{i-1}) \ar[r] & 
A\otimes_{\mbb{Z}_p} T^{\vee}_{\mfS}(\mfM_i) \ar[r]
& A\otimes_{\mbb{Z}_p} T^{\vee}_{\mfS}(\mfM_i/\mfM_{i-1}) \ar[r] & 0 }$
\end{center}
where $\iota_{\mfS,i-1}, \iota_{\mfS,i}$ and $\iota_{\mfS,i,i-1}$
are maps $\iota_{\mfS}$ for $\mfM_i, \mfM_{i-1}$ and $\mfM_i/\mfM_{i-1}$, respectively.
By corollary \ref{exact} and the exactness of $T_{\mfS}$, 
two horizontal sequences are exact.
By induction on $i$, we see that 
$A\otimes_{\mrm{\mfS^{\mrm{ur}}}}\iota_{\mfS}$ (for $\mfM$) is injective.

Finally,
if we put $A=W(\mrm{Fr}R)$, we see the bijectivity
of $W(\mrm{Fr}R)\otimes_{\mrm{\mfS^{\mrm{ur}}}}\iota_{\mfS}$
from (1), (2) and $\mfrak{t}\in W(\mrm{Fr}R)^{\times}$. 
\end{proof}

Let $\hat{\mfM}$ be in 
${}_{\mrm{w}}\mrm{Mod}^{r,\hat{G}}_{/\mfS_{\infty}}$.
We construct a map $\hat{\iota}$ for $\hat{\mfM}$ which connects 
$\hat{\mfM}$ to $\hat{T}(\hat{\mfM})$ (cf.\ \cite{Li2}, Section 3.1).
First, we recall that  we abuse of notations by denoting $\hat{\mfM}$
the underlying module $\whR\otimes_{\vphi,\mfS} \mfM$.
Observe that there exists a natural isomorphism of 
$\mbb{Z}_p[G]$-modules
\[
\hat{T}(\hat{\mfM})=\mrm{Hom}_{\whR,\vphi}(\hat{\mfM},W(R)_{\infty})
\simeq
\mrm{Hom}_{W(R),\vphi}(W(R)\otimes_{\whR}\hat{\mfM},W(R)_{\infty})
\]
where $G$ acts on 
$\mrm{Hom}_{W(R),\vphi}(W(R)\otimes_{\whR}\hat{\mfM},W(R)_{\infty})$
by $(\sigma.f)(x)=\sigma(f(\sigma^{-1}(x)))$ for 
$\sigma\in G,
f\in \mrm{Hom}_{W(R),\vphi}(W(R)\otimes_{\whR}\hat{\mfM},W(R)_{\infty}),
x\in W(R)\otimes_{\whR}\hat{\mfM}$.
Thus we can define a morphism 
$\hat{\iota}'\colon W(R)\otimes_{\whR}\hat{\mfM} \to 
\mrm{Hom}_{\mbb{Z}_p}(\hat{T}(\hat{\mfM}),W(R)_{\infty})$
by
\[
x\mapsto (f\mapsto f(x)),\quad 
x\in W(R)\otimes_{\whR}\hat{\mfM}, f\in \hat{T}(\hat{\mfM}). 
\]
Since $\hat{T}(\hat{\mfM})\simeq \oplus_{i\in I}\mbb{Z}_p/p^{n_i}\mbb{Z}_p$ 
as finite $\mbb{Z}_p$-modules,
we have a natural isomorphism 
$\mrm{Hom}_{\mbb{Z}_p}(\hat{T}(\hat{\mfM}),W(R)_{\infty})\simeq 
W(R)\otimes_{\mbb{Z}_p}\hat{T}^{\vee}(\hat{\mfM})$.
Recall that $\hat{T}^{\vee}(\hat{\mfM})=\mrm{Hom}_{\mbb{Z}_p}(\hat{T}(\hat{\mfM}),\mbb{Q}_p/\mbb{Z}_p)$
is the dual representation of $\hat{T}(\hat{\mfM})$.
Composing this isomorphism with $\hat{\iota}'$, we obtain a map 
\[
\hat{\iota}\colon W(R)\otimes_{\whR}\hat{\mfM}\to 
W(R)\otimes_{\mbb{Z}_p}\hat{T}^{\vee}(\hat{\mfM}).
\]
\noindent
For $\hat{\mfM}\in {}_{\mrm{w}}\mrm{Mod}^{r,\hat{G}}_{/\mfS}$, 
we also construct $\hat{\iota}\colon  W(R)\otimes_{\whR}\hat{\mfM}\to 
W(R)\otimes_{\mbb{Z}_p}\hat{T}^{\vee}(\hat{\mfM})$ by the same way
except only replacing $W(R)_{\infty}$ with $W(R)$.

\begin{lemma}
\label{rel2}
Let $A$ be a ring with $\mfS^{\mrm{ur}}\subset A \subset W(\mrm{Fr}R)$
which yields a ring extension $A_1 \subset \mrm{Fr}R$.
Suppose that $A$ is $\vphi_{W(\mrm{Fr}R)}$-stable.
Let $\hat{\mfM}$ be in 
${}_{\mrm{w}}\mrm{Mod}^{r,\hat{G}}_{/\mfS_{\infty}}$ 
or ${}_{\mrm{w}}\mrm{Mod}^{r,\hat{G}}_{/\mfS}$.
Let $\hat{\iota}$ be as above.

\noindent
$(1)$ $\hat{\iota}\simeq W(R)\otimes_{\vphi,\mfS^{\mrm{ur}}} \iota_{\mfS}$,
that is, the following diagram commutes:
\begin{center}
$\displaystyle \xymatrix{
W(R)\otimes_{\vphi,\mfS} \mfM  \ar^{\hat{\iota}}[rr]  
& &  
W(R)\otimes_{\mbb{Z}_p} \hat{T}^{\vee}(\hat{\mfM}) \\
W(R)\otimes_{\vphi,\mfS^{\mrm{ur}}} (\mfS^{\mrm{ur}}\otimes_{\mfS} \mfM) 
\ar^{W(R)\otimes_{\vphi,\mfS^{\mrm{ur}}} \iota_{\mfS}\quad }[rr] 
\ar^{\alpha\otimes \mrm{Id}_{\mfM}}_{\wr}[u]
& &
W(R)\otimes_{\vphi,\mfS^{\mrm{ur}}}(\mfS^{\mrm{ur}}\otimes_{\mbb{Z}_p} 
T^{\vee}_{\mfS}{\mfM}). 
\ar^{\alpha\otimes (\theta^{\vee})^{-1}}_{\wr}[u]
}$
\end{center}
Here, 
$\alpha\colon W(R)\otimes_{\vphi,\mfSur} \mfSur  \to W(R)$
is the isomorphism given by 
$\alpha(\sum_i a_i\otimes b_i)=\sum_i a_i\vphi(b_i)$ with 
$a_i\in W(R), b_i\in \mfSur$.

\noindent
$(2)$ $\hat{\iota}$  is $G$-equivalent and
$\vphi$-equivalent. 
Furthermore, 
$A\otimes_{W(R)} \hat{\iota}$ is injective.

\noindent
$(3)$ If $r<\infty$, 
then $\vphi(\mfrak{t})^r(A\otimes_{\mbb{Z}_p} \hat{T}^{\vee}(\hat{\mfM}))\subset 
(A\otimes_{W(R)} \hat{\iota})(A\otimes_{\whR} \hat{\mfM})$.
If $r=\infty$,
then $\vphi(\mfrak{t})^{r'}(A\otimes_{\mbb{Z}_p} \hat{T}^{\vee}(\hat{\mfM}))\subset 
(A\otimes_{W(R)} \hat{\iota})(A\otimes_{\whR} \hat{\mfM})$ for $r'>0$
such that $\mfM$ is of $E(u)$-height $r'$.

\noindent
$(4)$ The map 
\[
W(\mrm{Fr}R)\otimes_{W(R)} \hat{\iota}\colon 
W(\mrm{Fr}R)\otimes_{\whR} \hat{\mfM}\to 
W(\mrm{Fr}R)\otimes_{\mbb{Z}_p} \hat{T}^{\vee}(\hat{\mfM})
\] 
is bijective. 
\end{lemma}

\begin{proof}
The statement (1) follows from the proof as same as that of Proposition 3.1.3 (2) of \cite{Li2}.
To see that $A\otimes_{W(R)} \hat{\iota}$ is injective, by (1),
it is enough to check that $A\otimes_{\vphi, \mfSur} \iota_{\mfS}\colon 
A\otimes_{\vphi, \mfS}\mfM\to A\otimes_{\mbb{Z}_p} T^{\vee}_{\mfS}(\mfM)$
is injective. 
This can be checked by almost the same method as the proof of Lemma \ref{rel1} (1).
The rest statements follow from (1) and Lemma \ref{rel1}.  
\end{proof}

Let $\hat{\mfM}$ be in 
${}_{\mrm{w}}\mrm{Mod}^{r,\hat{G}}_{/\mfS_{\infty}}$ 
or ${}_{\mrm{w}}\mrm{Mod}^{r,\hat{G}}_{/\mfS}$.
Then $T_{\mfS}(\hat{\mfM})$ has a natural $G$-action via 
$\theta\colon T_{\mfS}(\mfM)\overset{\sim}{\rightarrow}\hat{T}(\hat{\mfM})$ 
(see Theorem \ref{Li}).
\begin{corollary}
Let $\hat{\mfM}$ and $\hat{\mfM}'$ be in 
${}_{\mrm{w}}\mrm{Mod}^{r,\hat{G}}_{/\mfS_{\infty}}$ 
$($resp.\ ${}_{\mrm{w}}\mrm{Mod}^{r,\hat{G}}_{/\mfS})$.
Let $f\colon \mfM'\to \mfM$ be a morphism in $\mrm{Mod}^r_{/\mfS_{\infty}}$.
If $T_{\mfS}(f)$ is $G$-equivalent,
then $f$ is in fact a morphism in  
${}_{\mrm{w}}\mrm{Mod}^{r,\hat{G}}_{/\mfS_{\infty}}$ 
$($resp.\ ${}_{\mrm{w}}\mrm{Mod}^{r,\hat{G}}_{/\mfS})$.
\end{corollary}

\begin{proof}
Consider the commutative diagram
$$
\xymatrix{
W(R)\otimes_{\mbb{Z}_p}\hat{T}^{\vee}(\hat{\mfM}')\ar[r] & 
W(R)\otimes_{\mbb{Z}_p}\hat{T}^{\vee}(\hat{\mfM}) \\
W(R)\otimes_{\whR}(\whR\otimes_{\vphi,\mfS}\mfM')\ar[r] \ar@{^{(}->}^{\hat{\iota}}_{\wr}[u] & 
W(R)\otimes_{\whR}(\whR\otimes_{\vphi,\mfS}\mfM)\ar@{^{(}->}^{\hat{\iota}}_{\wr}[u]
.}
$$
where the top and bottom arrows are  morphisms induced from $f$.
By our assumption of $f$ and the result that $\hat{\iota}$ is injective, 
we see that the bottom arrow commutes with $G$-action
and then we have done. 
\end{proof}

\subsection{Proof of Theorem \ref{MThm1}} 

\begin{lemma}
\label{MThm2''}
Let $0\to T'\to T\to T''\to 0$ be an exact sequence in $\mrm{Rep}_{\mrm{tor}}(G_{\infty})$.
Let $\mfM\in \mrm{Mod}^r_{/\mfS_{\infty}}$   and
$\psi$ an isomorphism 
$T_{\mfS}(\mfM)\overset{\sim}{\longrightarrow} T$
of $\mbb{Z}_p$-representations of $G_{\infty}$. 
Then there exists an exact sequence 
$0\to \mfM''\to \mfM\to \mfM'\to 0$
in $\mrm{Mod}^r_{/\mfS_{\infty}}$ 
which makes the following commutative diagram:
\begin{center}
$\displaystyle \xymatrix{
0\ar[r] & T'\ar[r] & T\ar[r] 
& T''\ar[r] & 0 \\
0\ar[r] & T_{\mfS}(\mfM')\ar[r] \ar_{\wr}[u] & 
T_{\mfS}(\mfM)\ar[r]\ar^{\psi}_{\wr}[u] 
& T_{\mfS}(\mfM'')\ar[r] \ar_{\wr}[u] & 0. }$
\end{center}
\end{lemma}

\begin{proof}
Put $M=\cO\otimes_{\mfS} \mfM$ and 
$\Psi$ an isomorphism defined by the composite 
$\mcal{T}(M)\simeq T_{\mfS}\overset{\psi}{\to} T$.
By Proposition \ref{Fon},
there exists an exact sequence 
$0\to M''\to M\overset{g}{\to} M'\to 0$ in $\mbf{\Phi M}_{/\cO_{\infty}}$  
which makes the following commutative diagram:
\begin{center}
$\displaystyle \xymatrix{
0\ar[r] & T'\ar[r] & T\ar[r] 
& T''\ar[r] & 0 \\
0\ar[r] & \mcal{T}(M')\ar[r] \ar_{\wr}[u] & 
\mcal{T}(M)\ar[r]\ar^{\Psi}_{\wr}[u] 
& \mcal{T}(M'')\ar[r] \ar_{\wr}[u] & 0. }$
\end{center}
By abuse of notation we denote by $g$ the composite $\mfM\hookrightarrow M\overset{g}{\to} M'$.
Put $\mfM''=\mfM\cap M''$ and $\mfM'=g(\mfM)$.
Since $\mfM\in \mrm{Mod}^r_{/\mfS_{\infty}}$ and $M'$ is $p'$-torsion free,
it follows from Proposition \ref{stc} that $\mfM'$ and 
$\mfM''$ are in $\mrm{Mod}^r_{/\mfS_{\infty}}$.
The inclusion map $\mfM\hookrightarrow M$ induces an injection $\mfM''\hookrightarrow M''$
and thus we have the following commutative diagram
\begin{center}
$\displaystyle \xymatrix{
0\ar[r] & M''\ar[r] & M\ar[r] 
& M'\ar[r] & 0 \\
0\ar[r] & \cO\otimes_{\mfS}\mfM''\ar[r] \ar[u] & 
\cO\otimes_{\mfS}\mfM\ar[r]\ar@{=}[u]
& \cO\otimes_{\mfS}\mfM' \ar[r]  \ar[u] & 0 }$
\end{center}
where two horizontal sequences of \'etale $\vphi$-modules are exact.
By a diagram chasing, we see that  the map $\cO\otimes_{\mfS} \mfM'\to M'$ 
is surjective.
Since $\mfM'\subset M'$ is $\vphi$-stable and finite as a $\mfS$-module,
we know that the map $\cO\otimes_{\mfS} \mfM'\to M'$ is 
injective (cf.\ \cite{Fo}, B. 1.4.2) and thus, it is  bijective. 
By the snake lemma, we know that the left vertical arrow of the above diagram is also bijective.
Taking the functor $T$ to the above diagram, Proposition \ref{FoKi} gives  the desired result.
\end{proof}

\begin{theorem}
\label{MThm2}
Let $0\to T'\to T\to T''\to 0$ be an exact sequence of 
finite torsion $\mbb{Z}_p$-representations of $G$.
Suppose that there exist a torsion 
$(\vphi, \hat{G})$-module $\hat{\mfM}$ of height $r$
and an isomorphism 
$\psi\colon \hat{T}(\hat{\mfM})\overset{\sim}{\longrightarrow} T$
of $\mbb{Z}_p$-representations of $G$. 
Then there exists an exact sequence 
$0\to \hat{\mfM}''\to \hat{\mfM}\to \hat{\mfM}'\to 0$
in $\mrm{Mod}^{r,\hat{G}}_{/\mfS_{\infty}}$ 
which makes the following commutative diagram:
\begin{center}
$\displaystyle \xymatrix{
0\ar[r] & T'\ar[r] & T\ar[r] 
& T''\ar[r] & 0 \\
0\ar[r] & \hat{T}(\hat{\mfM}')\ar[r] \ar_{\wr}[u] & 
\hat{T}(\hat{\mfM})\ar[r]\ar^{\psi}_{\wr}[u] 
& \hat{T}(\hat{\mfM}'')\ar[r] \ar_{\wr}[u] & 0. }$
\end{center}
\end{theorem}

\begin{proof}
By a short argument shows that we may suppose 
$T=\hat{T}(\hat{\mfM})$
and $\psi$
is the identity map for $T$. 
Take 
\[
\theta\colon T_{\mfS}(\mfM)\to \hat{T}(\hat{\mfM})
\]
defined by 
\[
\theta(f)(a\otimes m)=a\vphi(f(m))\quad \mrm{for}\ 
f\in T_{\mfS}(\mfM),\ a\in \wh{\mcal{R}}, m\in \mfM,
\]
as appeared in Section 2.4, which is $G_{\infty}$-equivalent.
By Lemma \ref{MThm2''},
we have an exact sequence 
$0\to \mfM''\to \mfM\to \mfM'\to 0$
in $\mrm{Mod}^r_{/\mfS_{\infty}}$ 
which makes the following commutative diagram:
\begin{center}
$\displaystyle \xymatrix{
0\ar[r] & T'\ar[r] & T\ar[r] 
& T''\ar[r] & 0 \\
0\ar[r] & T_{\mfS}(\mfM')\ar[r] \ar_{\wr}[u] & 
T_{\mfS}(\mfM)\ar[r]\ar^{\theta}_{\wr}[u] 
& T_{\mfS}(\mfM'')\ar[r] \ar_{\wr}[u] & 0. }$
\end{center}
We want to equip $\mfM'$ and $\mfM''$ with structures of $(\vphi,\hat{G})$-modules.
Combining the  diagram with Lemma \ref{rel2},
we obtain the following diagram whose all squares commute:

\begin{tiny}
$$
\xymatrix@C=2mm@R=3mm{
 & 
W(R)\otimes_{\vphi,\mfS} \mfM \ar@{^{(}->}[rr] \ar^{\hat{\iota}}[rr] & 
 & 
W(R)\otimes_{\mbb{Z}_p} T^{\vee} \ar^{\wr}_{\alpha^{-1}\otimes \theta^{\vee}}[dd]\\
W(R)\otimes_{\vphi,\mfS} \mfM'' \ar@{^{(}->}[ur]  & 
 & 
W(R)\otimes_{\mbb{Z}_p} (T'')^{\vee}  \ar^{\wr}[dd] \ar@{^{(}->}[ur]
& \\
&
W(R)\otimes_{\vphi,\mfSur}(\mfSur\otimes_{\mfS} \mfM) 
\ar_{\wr}^{\alpha\otimes \mrm{Id}_{\mfM}}[uu] 
\ar@{^{(}-}[r]  &
\ar^{W(R)\otimes \iota_{\mfS}\qquad \qquad}[r] &
W(R)\otimes_{\vphi,\mfSur} (\mfSur\otimes_{\mbb{Z}_p} T^{\vee}_{\mfS}(\mfM)) 
 \\
W(R)\otimes_{\vphi,\mfSur}(\mfSur\otimes_{\mfS} \mfM'') \ar@{^{(}->}[rr] 
\ar^{W(R)\otimes \iota_{\mfS}}[rr]  \ar_{\wr}^{\alpha\otimes \mrm{Id}_{\mfM''}}[uu] \ar@{^{(}->}[ur]
&  
& 
W(R)\otimes_{\vphi,\mfSur} (\mfSur\otimes_{\mbb{Z}_p} T^{\vee}_{\mfS}(\mfM'')). \ar@{^{(}->}[ur] & 
}
$$
\end{tiny}
Here, 
$\alpha\colon W(R)\otimes_{\vphi,\mfSur} \mfSur  \to W(R)$
is the isomorphism given by 
$\alpha(\sum_i a_i\otimes b_i)=\sum_i a_i\vphi(b_i)$ with 
$a_i\in W(R), b_i\in \mfSur$.
Define a map $W(R)\otimes_{\vphi,\mfS}\mfM'' \to W(R)\otimes_{\mbb{Z}_p} (T'')^{\vee}$
such that all squares in the above diagram commute.
Tensoring $W(\mrm{Fr}R)$ to the ceiling, 
we obtain a diagram whose all maps are injective (cf.\ Corollary \ref{exact} and \ref{ringext}): 
\begin{center}
$\displaystyle \xymatrix{
\wh{\mfM} \ar@{^{(}->}[r] & 
W(\mrm{Fr}R)\otimes_{\whR} \hat{\mfM} \ar@{^{(}->}[r] \ar^{\iota\quad }[r]& 
W(\mrm{Fr}R)\otimes_{\mbb{Z}_p} \hat{T}^{\vee}(\hat{\mfM}) 
\\
\whR\otimes_{\vphi,\mfS}\mfM'' \ar@{^{(}->}[r] \ar@{^{(}->}[u] & 
W(\mrm{Fr}R)\otimes_{\whR} (\whR\otimes_{\vphi,\mfS}\mfM'') \ar@{^{(}->}[r] 
\ar^{\qquad \iota''}[r] \ar@{^{(}->}[u] & 
W(\mrm{Fr}R)\otimes_{\mbb{Z}_p}(T'')^{\vee} \ar@{^{(}->}[u]. 
}$
\end{center}
Moreover, the map $\iota=W(\mrm{Fr}R)\otimes_{W(R)}\hat{\iota}$ is bijective by Lemma \ref{rel2} (4), and 
the map $\iota''$ is also bijective by Lemma \ref{rel1} (3).  
Define a $G$-action on $W(\mrm{Fr}R)\otimes_{\whR} (\whR\otimes_{\vphi,\mfS}\mfM'')$
via $\iota''$. Then the injection 
$W(\mrm{Fr}R)\otimes_{\whR} (\whR\otimes_{\vphi,\mfS}\mfM'')
\hookrightarrow 
W(\mrm{Fr}R)\otimes_{\whR} \hat{\mfM}$
is automatically $G$-equivalent.
On the other hand,
see the diagram
\begin{center}
$\displaystyle \xymatrix{
0\ar[r] & \mfS\otimes_{\vphi,\mfS} \mfM''\ar[r] \ar[d] & 
\mfS\otimes_{\vphi,\mfS} \mfM\ar[r] \ar[d] & 
\mfS\otimes_{\vphi,\mfS} \mfM'\ar[r] \ar[d] & 
0 \\
0\ar[r] & \whR\otimes_{\vphi,\mfS} \mfM''\ar[r] \ar[d] & 
\whR\otimes_{\vphi,\mfS} \mfM\ar[r] \ar[d] &
\whR\otimes_{\vphi,\mfS} \mfM'\ar[r] \ar[d]
& 0\\
0\ar[r] & W(\mrm{Fr}R)\otimes_{\vphi,\mfS} \mfM''\ar[r] & 
W(\mrm{Fr}R)\otimes_{\vphi,\mfS} \mfM\ar[r]  &
W(\mrm{Fr}R)\otimes_{\vphi,\mfS} \mfM'\ar[r] & 0. }$
\end{center}
By Corollary \ref{exact} and \ref{ringext}, we see that 
all horizontal sequences are exact and all vertical arrows are injective.
Hence we may regard $\whR\otimes_{\vphi,\mfS} \mfM, \whR\otimes_{\vphi,\mfS} \mfM''$
 and $W(\mrm{Fr}R)\otimes_{\vphi,\mfS} \mfM''$
as submodules of $W(\mrm{Fr}R)\otimes_{\vphi,\mfS} \mfM=W(\mrm{Fr}R)\otimes_{\whR} \hat{\mfM}$.
In particular, 
we have 
\begin{equation}
\label{equa}
\whR\otimes_{\vphi,\mfS} \mfM''=
(\whR\otimes_{\vphi,\mfS} \mfM) \cap (W(\mrm{Fr}R)\otimes_{\vphi,\mfS} \mfM'').
\end{equation}
\noindent
Since $G$-actions on $\whR\otimes_{\vphi,\mfS} \mfM$ and $W(\mrm{Fr}R)\otimes_{\vphi,\mfS} \mfM''$
are restrictions of the $G$-action on $W(\mrm{Fr}R)\otimes_{\vphi,\mfS} \mfM=W(\mrm{Fr}R)\otimes_{\whR} \hat{\mfM}$,
the equation (\ref{equa}) gives a well-defined $G$-action on $\whR\otimes_{\vphi,\mfS} \mfM''$. 
Since the $G$-action on $\whR\otimes_{\vphi,\mfS} \mfM$ factors through $\hat{G}$,
the $G$-action on $\whR\otimes_{\vphi,\mfS} \mfM''$ also factors through $\hat{G}$.
We also define $\hat{G}$-action on $\whR\otimes_{\vphi,\mfS} \mfM'$ via a natural isomorphism
$\whR\otimes_{\vphi,\mfS} \mfM''\simeq (\whR\otimes_{\vphi,\mfS} \mfM'')/(\whR\otimes_{\vphi,\mfS} \mfM'')$.
It is not difficult to check that 
triples $\hat{\mfM}'=(\mfM',\vphi,\hat{G})$ and  $\hat{\mfM}''=(\mfM'',\vphi,\hat{G})$
are weak $(\vphi,\hat{G})$-modules.
Obviously, we have the exact sequences 
\begin{equation}
\label{desired}
0\to \hat{\mfM}''\to \hat{\mfM}\to \hat{\mfM}'\to 0
\end{equation}
of weak $(\vphi,\hat{G})$-modules.
By Corollary \ref{subweak},
we know that   $\hat{\mfM}'$ and $\hat{\mfM}''$  are  in fact $(\vphi,\hat{G})$-modules.
Now we check that the exact sequence (\ref{desired}) satisfies the desired property.
Projections $\mfM\to \mfM'$ and $\hat{\mfM}\to \hat{\mfM}'$ induce injections 
$T_{\mfS}(\mfM')\hookrightarrow T_{\mfS}(\mfM)$ of $\mbb{Z}_p[G_{\infty}]$-modules and 
$\hat{T}(\mfM')\hookrightarrow \hat{T}(\mfM)$ of $\mbb{Z}_p[G]$-modules.
Furthermore, the diagram below is commutative:
\begin{center}
$\displaystyle \xymatrix{
\hat{T}(\hat{\mfM}')  \ar@{^{(}->}[d] & 
T_{\mfS}(\mfM')\ar_{\sim}^{\theta}[l] \ar@{^{(}->}[d] \ar^{\sim}[r] & 
T' \ar@{^{(}->}[d]\\
\hat{T}(\hat{\mfM})  & 
T_{\mfS}(\mfM) \ar_{\sim}^{\theta}[l]  \ar_{\theta}^{\sim}[r] &
T. 
}$
\end{center}
This induces the commutative diagram 
\begin{center}
$\displaystyle \xymatrix{
T'  \ar@{^{(}->}[r] & 
T \ar@{=}[d]\\  
\hat{T}(\hat{\mfM}')  \ar@{^{(}->}[r] \ar_{\wr}[u] & 
\hat{T}(\hat{\mfM})=T
}$
\end{center}
and thus we see that the left vertical arrow in just the above square is $G$-equivalent.
The desired result follows from this.
\end{proof}

\begin{remark}
By using the theory of \'etale $(\vphi,\hat{G})$-modules,
we will know a canonical understanding for  the sequence $0\to \hat{\mfM}''\to \hat{\mfM}\to \hat{\mfM}'\to 0$
appeared in Theorem \ref{MThm2}, see Remark \ref{E1.1}.
\end{remark}

By Theorem \ref{MThm2}, 
the essential image of the functor $\hat{T}\colon \mrm{Mod}^{r,\hat{G}}_{/\mfS_{\infty}}\to \mrm{Rep}_{\mrm{tor}}(G)$
is stable under talking a subquotient. 
In particular, we see that the category $\mrm{Rep}^{\hat{G}}_{\mrm{tor}}(G)$ 
is also stable under taking a subquotient.
Clearly, the category $\mrm{Rep}^{\hat{G}}_{\mrm{tor}}(G)$ 
is also stable under taking a direct sum.
We show that $\mrm{Rep}^{\hat{G}}_{\mrm{tor}}(G)$  is stable under taking a dual and a tensor product.
\begin{lemma}
The full subcategory $\mrm{Rep}^{\hat{G}}_{\mrm{tor}}(G)$ of $\mrm{Rep}_{\mrm{tor}}(G)$
is stable under taking a dual.
\end{lemma}
\begin{proof}
Let $T\in \mrm{Rep}^{\hat{G}}_{\mrm{tor}}(G)$  and 
take some  $\hat{\mfM}\in \mrm{Mod}^{r,\hat{G}}_{/\mfS_{\infty}}$ (for some $r<\infty$) 
such that $T=\hat{T}(\hat{\mfM})$. 
Take an integer $n\ge 0$ such that $\mfM$ is killed by $p^n$.
For any integer $k\ge 0$, denote by  $\hat{\mfS}_n(k)$ 
the Cartier dual of the trivial $(\vphi,\hat{G})$-module $\hat{\mfS}_n$ 
in $\mrm{Mod}^{k,\hat{G}}_{/\mfS_{\infty}}$ and by 
$\mfS_n(k)$ its underlying $\vphi$-module.
Then it can be seen immediately that 
$\mfM\otimes_{\mfS} \mfS_n(k)$ has a structure of a $(\vphi,\hat{G})$-module
of height $r+k$, and if we denote it by $\hat{\mfM}(k)$, then 
$\hat{T}(\hat{\mfM}(k))=\hat{T}(\hat{\mfM})(k)$.
Take an integer $m>r$ which is divided  by $p-1$.
Then 
\begin{align*}
T^{\vee}&=\hat{T}(\hat{\mfM}^{\vee})\otimes_{\mbb{Z}_p} \mbb{F}_p(-r)
=\hat{T}(\hat{\mfM}^{\vee})\otimes_{\mbb{Z}_p} \mbb{F}_p(m-r)\\
&=\hat{T}(\hat{\mfM}^{\vee})\otimes_{\mbb{Z}_p} \hat{T}(\hat{\mfS}_n(m-r))
=\hat{T}(\hat{\mfM}^{\vee}(m-r))
\end{align*}
and we have done.
\end{proof}

Finally we consider the assertion related with a tensor product of Theorem \ref{MThm1}.
It is enough to prove the following lemma.
\begin{lemma}
\label{tensorLi}
Let $\hat{\mfM}\in {}_{\mrm{w}}\mrm{Mod}^{r,\hat{G}}_{/\mfS_{\infty}}$ 
$($resp.\ $\hat{\mfM}\in \mrm{Mod}^{r,\hat{G}}_{/\mfS_{\infty}})$
and 
$\hat{\mfM}'\in {}_{\mrm{w}}\mrm{Mod}^{r',\hat{G}}_{/\mfS_{\infty}}$ 
$($resp.\ $\hat{\mfM}\in \mrm{Mod}^{r',\hat{G}}_{/\mfS_{\infty}})$
for some $r,r'\in \{0,1,\dots ,\infty\}$.
Then $\frac{\mfM\otimes_{\mfS}\mfM'}{u{\rm \mathchar`-tor}}$ is an object of $\mrm{Mod}^{r+r'}_{/\mfS_{\infty}}$
and has a structure of a weak $(\vphi,\hat{G})$-modules
$($resp.\ a $(\vphi,\hat{G})$-modules$)$.
If we put $\hat{\mfM}\otimes \hat{\mfM}'=\wh{\frac{\mfM\otimes_{\mfS}\mfM'}{u{\rm \mathchar`-tor}}}$,
then there exists a canonical isomorphism
$\hat{T}(\hat{\mfM}\otimes \hat{\mfM}')\simeq \hat{T}(\hat{\mfM})\otimes_{\mbb{Z}_p}\hat{T}(\hat{\mfM}')$
 of $\mbb{Z}_p[G]$-modules.
\end{lemma}

\begin{proof}
Since $\frac{\mfM\otimes_{\mfS}\mfM'}{u{\rm \mathchar`-tor}}$ is $u$-torsion free,
we see 
$\frac{\mfM\otimes_{\mfS}\mfM'}{u{\rm \mathchar`-tor}}\in \mrm{Mod}^{r+r'}_{/\mfS_{\infty}}$
by Proposition \ref{torKi}.
We equip a $\hat{G}$-action (resp.\ a $G$-action) on $\whR\otimes_{\vphi,\mfS}(\mfM\otimes_{\mfS}\mfM')$
(resp.\ $W(\mrm{Fr}R)\otimes_{\vphi,\mfS}(\mfM\otimes_{\mfS}\mfM')$)
via a canonical isomorphism 
$\whR\otimes_{\vphi,\mfS}(\mfM \otimes_{\mfS}\mfM')\simeq 
(\whR\otimes_{\vphi,\mfS}\mfM) \otimes_{\mfS}(\whR \otimes_{\vphi,\mfS}(\mfM'))$.
(resp.\ 
$W(\mrm{Fr}R)\otimes_{\vphi,\mfS}(\mfM \otimes_{\mfS}\mfM')\simeq 
(W(\mrm{Fr}R)\otimes_{\whR}(\whR\otimes_{\vphi,\mfS}\mfM)) 
\otimes_{\mfS}
(W(\mrm{Fr}R)\otimes_{\whR}(\whR\otimes_{\vphi,\mfS}\mfM'))$).
If we denote by $(u{\rm \mathchar`-tor})$ by the $u$-torsion part of 
$\mfM\otimes_{\mfS}\mfM'$, then we obtain an exact sequence
\[
\whR\otimes (u{\rm \mathchar`-tor}) \to \whR\otimes_{\vphi,\mfS}(\mfM \otimes_{\mfS}\mfM')
\overset{\eta}{\rightarrow} \whR\otimes_{\vphi,\mfS}(\frac{\mfM \otimes_{\mfS}\mfM'}{u{\rm \mathchar`-tor}})\to 0
\]
as $\whR$-modules.
Note that $u$ is a unit of $W(\mrm{Fr}R)$.
Since a natural map 
$\whR\otimes_{\vphi,\mfS}\frac{\mfM \otimes_{\mfS}\mfM'}{u{\rm \mathchar`-tor}}
\to W(\mrm{Fr}R)\otimes_{\vphi,\mfS}\frac{\mfM \otimes_{\mfS}\mfM'}{u{\rm \mathchar`-tor}}
=W(\mrm{Fr}R)\otimes_{\vphi,\mfS}(\mfM\otimes_{\mfS}\mfM')$
is injective (cf.\ Corollary \ref{ringext}), 
we see that the equality 
$\mrm{ker}(\eta)=\mrm{ker}
(\whR\otimes_{\vphi,\mfS}(\mfM\otimes_{\mfS}\mfM')
\to
W(\mrm{Fr}R)\otimes_{\vphi,\mfS}(\mfM\otimes_{\mfS}\mfM'))$
and thus $\mrm{ker}(\eta)$ is stable under the $\hat{G}$-action on 
$\whR\otimes_{\vphi,\mfS}(\mfM\otimes_{\mfS}\mfM')$.
Therefore,
we can equip a $\hat{G}$-action on 
$\whR\otimes_{\vphi,\mfS}\frac{\mfM \otimes_{\mfS}\mfM'}{u{\rm \mathchar`-tor}}$
via a canonical isomorphism
$\whR\otimes_{\vphi,\mfS}\frac{\mfM \otimes_{\mfS}\mfM'}{u{\rm \mathchar`-tor}}
\simeq (\whR\otimes_{\vphi,\mfS}(\mfM\otimes_{\mfS}\mfM'))/\mrm{ker}(\eta)$.
Then it is not difficult to see that 
$\frac{\mfM \otimes_{\mfS}\mfM'}{u{\rm \mathchar`-tor}}$
has a structure of a $(\vphi,\hat{G})$-module. 
Finally we prove 
$\hat{T}(\hat{\mfM}\otimes \hat{\mfM}')\simeq \hat{T}(\hat{\mfM})\otimes_{\mbb{Z}_p}\hat{T}(\hat{\mfM}')$. 
By Proposition \ref{prop1},
we obtain $\vphi$-equivalent and $G$-compatible isomorphisms 
\begin{align*}
W(\mrm{Fr}R)\otimes_{\mbb{Z}_p}(\hat{T}_{\ast}(\hat{\mfM})\otimes \hat{T}_{\ast}(\hat{\mfM}'))
&\simeq W(\mrm{Fr}R)\otimes_{\whR}(\hat{\mfM}\otimes_{\whR} \hat{\mfM}')\\
&\simeq  W(\mrm{Fr}R)\otimes_{\whR}(\whR\otimes_{\vphi,\mfS}(\mfM \otimes_{\mfS} \mfM'))\\
&\simeq W(\mrm{Fr}R)\otimes_{\whR}(\whR\otimes_{\vphi,\mfS}(\frac{\mfM \otimes_{\mfS} \mfM'}{u{\rm \mathchar`-tor}})).
\end{align*}
Seeing ``$\vphi=1$''-part of the above modules,
we have that $\hat{T}_{\ast}(\hat{\mfM})\otimes \hat{T}_{\ast}(\hat{\mfM}')\simeq 
\hat{T}_{\ast}(\hat{\mfM}\otimes \hat{\mfM})$.
Taking the dual of both sides,
we obtain the desired result.
\end{proof}


\section{Maximal objects and minimal objects}
Caruso and Liu defined maximal objects for Kisin modules 
and Breuil modules in \cite{CL1} and
they proved that the category of maximal objects can be regarded as  a full subcategory
of $\mrm{Rep}_{\mrm{tor}}(G_{\infty})$.
In this section, we discuss maximal objects for $(\vphi,\hat{G})$-modules and 
prove that the category of them can be regarded as  a full subcategory
of $\mrm{Rep}_{\mrm{tor}}(G)$.

\subsection{Maximal objects and minimal objects for Kisin modules}

In this subsection, we recall the theory of maximal (minimal) objects given in \cite{CL1}.
For $M\in \mbf{\Phi M}_{/\cO_{\infty}}$,
we denote by $F^r_{\mfS}(M)$  the (partially) ordered set (by inclusion)
of $\mfM\in \mrm{Mod}^r_{/\mfS_{\infty}}$ contained in $M$ such that $\mfM[1/u]=M$.
Then 
$F^r_{\mfS}(M)$ has a greatest element and a smallest element (cf.\ \cite{CL1}, Corollary 3.2.6).
\begin{definition}
Let $\mfM\in \mrm{Mod}^r_{/\mfS_{\infty}}$.
We denote by $\mrm{Max}^r(\mfM)$ the greatest element of 
$F^r_{\mfS}(\mfM[1/u])$.
It is endowed with a homomorphism $\iota^{\mfM}_{\mrm{max}}\colon \mfM\to \mrm{Max}^r(\mfM)$   
in $\mrm{Mod}^r_{/\mfS_{\infty}}$.
\end{definition}
Maximal objects are characterized by the following universality
(\cite{CL1}, Proposition 3.3.5):
Let $\mfM\in  \mrm{Mod}^r_{/\mfS_{\infty}}$.
The couple $(\Max^r(\mfM), \iotamax^{\mfM})$
is characterized by the following universal property:
\begin{itemize}
\item The morphism $T_{\mfS}(\iotamax^{\mfM})$ is an isomorphism.
\item For each couple $(\mfM',f)$ where $\mfM'\in \mrm{Mod}^r_{/\mfS_{\infty}}$ 
and $f\colon \mfM\to \mfM'$ becomes an isomorphism under $T_{\mfS}$,
there exists a unique map $g\colon \mfM'\to \Max^r(\mfM)$ such that 
$g\circ f=\iotamax^{\mfM}$.
\end{itemize}
This property gives rise to a functor 
$\Max^r\colon \mrm{Mod}^r_{/\mfS_{\infty}} \to \mrm{Mod}^r_{/\mfS_{\infty}}$.
If we denote by $\Max^r_{/\mfS_{\infty}}$ its essential image, 
Caruso and Liu proved that

\begin{theorem}[\cite{CL1}, Theorem 3.3.8]
The category $\Max^r_{/\mfS_{\infty}}$ is abelian.
Moreover, kernels, cokernels, images and coimages  
in the abelian category $\Max^r_{/\mfS_{\infty}}$
have explicit descriptions.  
\end{theorem}
\noindent
The restriction $T_{\mfS}$ on $\Max^r_{/\mfS_{\infty}}$
is exact and fully faithful (cf. \cite{CL1}, Corollary 3.3.10).
\begin{center}
$\displaystyle \xymatrix{
\mrm{Mod}^r_{/\mfS_{\infty}}  \ar^{T_{\mfS}}[rr] \ar^{\Max^r}[rd] \ar@{->>}[rd] &
 &  
\mrm{Rep}_{\mrm{tor}}(G_{\infty}) \\  
 & 
\Max^r_{/\mfS_{\infty}}.
\ar@{^{(}->}[ur]  \ar^{T_{\mfS}}[ur] & 
}$
\end{center}
The theory for minimal objects proceeds if we take a ``dual''
to the above theory.
By Proposition 5.6 of \cite{CL2}, 
if $r=\infty$,
the functor $T_{\mfS}$ is an anti-equivalence of categories:
\[
T_{\mfS}\colon \Max^{\infty}_{/\mfS_{\infty}}\overset{\sim}{\longrightarrow} 
\mrm{Rep}_{\mrm{tor}}(G_{\infty}).
\]
For more precise properties,
see Section 3 of \cite{CL1}.

\subsection{\'Etale $(\vphi,\hat{G})$-modules}
In this subsection,
we give a notion of \'etale $(\vphi,\hat{G})$-modules.
The idea in this subsection follows from that of the $(\vphi,\tau)$-theory given in \cite{Ca}.
As one of the main theorem in \cite{Ca},
we prove that the category of 
various \'etale $(\vphi,\hat{G})$-modules are
equivalent to the category of various $\mbb{Z}_p$-representations of $G$,
including the case where $p=2$.

Here\footnote{
In \cite{Ca}, rings $\cOG$ and $\EG$ are denoted by 
$\E^{\mrm{int}}_{\tau}$ and $\E_{\tau}$, respectively.},
we put $\cOG=W(\mrm{Fr}R)^{H_{\infty}}$,
which is absolutely unramified and a complete discrete valuation ring 
with perfect residue field $\mrm{Fr}R^{H_{\infty}}$.
Furthermore $\cOG$ is a closed subring of $W(\mrm{Fr}R)$ for the weak topology.
Put $\EG=\mrm{Fr}\cOG=\cOG[1/p]$.
By definition, 
$\vphi_{W(\mrm{Fr}R)[1/p]}$ is stable on  $\cOG$ and $\EG$ 
which is bijective on themselves.
Furthermore,
$\hat{G}$ acts on $\cOG$ and $\EG$ continuously.
Since an inclusion $\cO \to \cOG$ (resp.\ $\E \to \EG$) is faithfully flat, 
for any \'etale $\vphi$-module $M$ over $\cO$ (resp.\ over $\E$),
a natural map $M\to \cOG\otimes_{\cO} M$ (resp.\ $M\to \EG\otimes_{\E} M$)
is an injection. By this embedding, 
we regard $M$ as a sub $\cO$-module of $\cOG\otimes_{\cO} M$
(resp.\ a sub $\E$-module of $\EG\otimes_{\E} M$). 
Similarly, a natural map 
$M\to \cOG\otimes_{\vphi, \cO} M$ (resp.\ $M\to \EG\otimes_{\vphi, \E} M$)
is an injection and by this embedding we regard $M$ as 
a sub $\vphi(\cOG)$-module of $\cOG\otimes_{\vphi, \cO} M$
(resp.\ a sub $\vphi(\EG)$-module of $\EG\otimes_{\vphi, \E} M$).
\begin{definition}
An {\it \'etale $(\vphi,\hat{G})'$-module over $\cO$} 
(resp.\ an {\it \'etale $(\vphi,\hat{G})$-module over $\cO$}) 
is a triple ${}'\hat{M}=(M,\vphi_M, \hat{G})$ (resp.\ $\hat{M}=(M,\vphi_M, \hat{G})$) 
where 

\vspace{-2mm}
\begin{enumerate}
\item[(1)] $(M,\vphi_M)$ is an \'etale  $\vphi$-module over $\cO$, 

\vspace{-2mm}

\item[(2)] $\hat{G}$ is a continuous  $\cOG$-semi-linear 
$\hat{G}$-action on $\cOG\otimes_{\cO} M$ (resp.\ $\cOG\otimes_{\vphi, \cO} M$)
for the weak topology,

\vspace{-2mm}

\item[(3)] the $\hat{G}$-action commutes with $\vphi_{\cOG}\otimes \vphi_M$,

\vspace{-2mm}

\item[(4)] $M\subset (\cOG\otimes_{\cO} M)^{H_K}$ 
(resp.\ $M\subset (\cOG\otimes_{\vphi, \cO} M)^{H_K}$).
\end{enumerate}

\vspace{-2mm}
If $M$ is killed by some power of $p$ (resp.\ free over $\cO$),
then ${}'\hat{M}$ (resp.\ $\hat{M}$) is called a {\it torsion \'etale  $(\vphi,\hat{G})$-module}
(resp.\ a {\it free \'etale  $(\vphi,\hat{G})$-module}).
By replacing $\cO$ and $\cOG$ with $\E$ and $\EG$, respectively,
we define an {\it \'etale $(\vphi,\hat{G})'$-module over $\E$} 
and an {\it \'etale $(\vphi,\hat{G})$-module over $\E$}.
\end{definition}

Denote by $'\mbf{\Phi M}^{\hat{G}}_{/\cO_{\infty}}$ 
(resp.\ $'\mbf{\Phi M}^{\hat{G}}_{/\cO}$,\ resp.\ $'\mbf{\Phi M}^{\hat{G}}_{/\E}$)
the category of torsion \'etale $(\vphi,\hat{G})'$-modules over $\cO$ 
(resp.\ the category of free \'etale $(\vphi,\hat{G})'$-modules over $\cO$, 
resp.\  the category of \'etale $(\vphi,\hat{G})'$-modules over $\E$).
Similarly, 
we denote by 
$\mbf{\Phi M}^{\hat{G}}_{/\cO_{\infty}}$ 
(resp.\ $\mbf{\Phi M}^{\hat{G}}_{/\cO}$,\ resp.\ $\mbf{\Phi M}^{\hat{G}}_{/\E}$)
the category of torsion \'etale $(\vphi,\hat{G})$-modules over $\cO$ 
(resp.\ the category of free \'etale $(\vphi,\hat{G})$-modules over $\cO$, 
resp.\  the category of \'etale $(\vphi,\hat{G})$-modules over $\E$).

If ${}'\hat{M}$ is an \'etale $(\vphi,\hat{G})'$-module over $\cO$,
then $\hat{G}$ acts on $\cOG\otimes_{\vphi,\cOG}(\cOG\otimes_{\cO} M)$ 
by a natural way.  
We obtain $\hat{G}$-action on $\cOG\otimes_{\vphi, \cO} M$ via 
\[
\cOG\otimes_{\vphi,\cOG}(\cOG\otimes_{\cO} M)\simeq
\cOG\otimes_{\vphi, \cO} M,\quad a\otimes(b\otimes x)\mapsto a\vphi(b)\otimes x
\]
where $a,b\in \cOG,\ x\in M$.
This $\hat{G}$-action equips $M$ with a structure of  an \'etale $(\vphi,\hat{G})$-module over $\cO$.
Conversely, if $\hat{M}$ is an \'etale $(\vphi,\hat{G})$-module over $\cO$,
we obtain $\hat{G}$-action on $\cOG\otimes_{\cO} M$ via 
\[
\cOG\otimes_{\vphi^{-1},\cOG}(\cOG\otimes_{\vphi, \cO} M)\simeq
\cOG\otimes_{\cO} M,\quad a\otimes(b\otimes x)\mapsto a\vphi^{-1}(b)\otimes x
\]
where $a,b\in \cOG,\ x\in M$.
This $\hat{G}$-action equips $M$ with a structure of  an \'etale $(\vphi,\hat{G})'$-module over $\cO$.
Consequently,
we have canonical equivalences of categories
\begin{equation}
\label{epmod1}
'\mbf{\Phi M}^{\hat{G}}_{/\cO_{\infty}}\simeq 
\mbf{\Phi M}^{\hat{G}}_{/\cO_{\infty}},\quad 
'\mbf{\Phi M}^{\hat{G}}_{/\cO}\simeq \mbf{\Phi M}^{\hat{G}}_{/\cO}.
\end{equation}
By the same way, we obtain
\begin{equation}
\label{epmod2}
'\mbf{\Phi M}^{\hat{G}}_{/\E}\simeq \mbf{\Phi M}^{\hat{G}}_{/\E}.
\end{equation}

In the following proposition, 
$\mcal{M}$ and $\mcal{T}$ are functors defined in Section 2.2.

\begin{lemma}
\label{Calem}
$(1)$ For all finite torsion $\mbb{Z}_p$-representations $T$ of $G_{\infty}$
(resp.\ finite free $\mbb{Z}_p$-representations $T$ of $G_{\infty}$,\ resp.\ 
finite $\mbb{Q}_p$-representations $T$ of $G_{\infty}$),
a natural map 
\begin{center}
$\cOG\otimes_{\cO} \mcal{M}(T)\to \mrm{Hom}_{\mbb{Z}_p[H_{\infty}]}(T, W(\mrm{Fr}R)_{\infty})$\\
$(\mrm{resp}.\ 
\cOG\otimes_{\cO} \mcal{M}(T)\to \mrm{Hom}_{\mbb{Z}_p[H_{\infty}]}(T, W(\mrm{Fr}R)),$\\
$\mrm{resp}.\
\EG\otimes_{\E} \mcal{M}(T)\to \mrm{Hom}_{\mbb{Q}_p[H_{\infty}]}(T, W(\mrm{Fr}R)[1/p]))$
\end{center}
is an isomorphism.

\noindent
$(2)$ For all torsion \'etale $\vphi$-modules $M$ over $\cO$
(resp.\ free \'etale $\vphi$-modules $M$ over $\cO$,\ resp.\ 
\'etale $\vphi$-modules $M$ over $\E$), a natural map
\begin{center}
$\mcal{T}(M)\to \mrm{Hom}_{\cOG, \vphi}(\cOG\otimes_{\cO} M, W(\mrm{Fr}R)_{\infty})$\\
$(\mrm{resp}.\ 
\mcal{T}(M)\to \mrm{Hom}_{\cOG, \vphi}(\cOG\otimes_{\cO} M, W(\mrm{Fr}R)),$\\
$\mrm{resp}.\
\mcal{T}(M)\to \mrm{Hom}_{\EG, \vphi}(\EG\otimes_{\E} M, W(\mrm{Fr}R)[1/p]))$
\end{center}
is an isomorphism.
\end{lemma}


\begin{proof}
We only prove the torsion case. The rest cases can be checked by a similar manner.
First we consider (1).
Taking a tensor product $W(\mrm{Fr}R)$ over $\wh{\cOur}$ to (\ref{Fon2})
and picking up a $H_{\infty}$-fixed parts, 
we obtain a natural bijective map
\begin{equation}
\label{is0}
\cOG\otimes_{\cO} (\cOur\otimes_{\mbb{Z}_p} T)^{G_{\infty}}\to 
(W(\mrm{Fr}R)\otimes_{\mbb{Z}_p} T)^{H_{\infty}}. 
\end{equation}
If we replace $T$ in (\ref{is0}) with its dual representation,
we obtain the desired result.
Using (\ref{Fon1}), we can check (2) by a similar way.

\end{proof}

We define a contravariant functor 
${}'\hat{\mcal{M}}\colon 
\mrm{Rep}_{\mrm{tor}}(G)\to 
{}'\mbf{\Phi M}^{\hat{G}}_{/\cO_{\infty}}$ 
as below:
For any $T\in \mrm{Rep}_{\mrm{tor}}(G)$,
define
\[
{}'\hat{\mcal{M}}(T)=\mcal{M}(T)=
\mrm{Hom}_{G_{\infty}}(T,\E^{\mrm{ur}}/\cOur)
\]
and we equip a $\hat{G}$-action on $\cOG\otimes_{\cO} \mcal{M}(T)$
via the isomorphism 
$\cOG\otimes_{\cO} \mcal{M}(T)\simeq \mrm{Hom}_{\mbb{Z}_p[H_{\infty}]}(T, W(\mrm{Fr}R)_{\infty})$
(cf.\ Lemma \ref{Calem} (1)).
Here $\hat{G}$ acts on the right hand side by the formula 
$(\sigma.f)(x)=\hat{\sigma}(f(\hat{\sigma}^{-1}(x)))$ for 
$\sigma\in \hat{G}$ and $\hat{\sigma}\in G$ any lift of $\sigma$, 
$f\in \mrm{Hom}_{\mbb{Z}_p[H_{\infty}]}(T, W(\mrm{Fr}R)_{\infty}),
x\in T$.

On the other hand,
we define a contravariant  functor 
${}'\hat{\mcal{T}}\colon 
'\mbf{\Phi M}^{\hat{G}}_{/\cO_{\infty}}\to \mrm{Rep}_{\mrm{tor}}(G)$ 
as below:
For any ${}'\hat{M}\in {}'\mbf{\Phi M}^{\hat{G}}_{/\cO_{\infty}}$,
define
\[
{}'\hat{\mcal{T}}({}'\hat{M})=\mcal{T}(M)=\mrm{Hom}_{\cO,\vphi}(M,\E^{\mrm{ur}}/\cOur)
\]
and we equip a $G$-action on ${}'\hat{\mcal{T}}({}'\hat{M})$
via the isomorphism 
$\mcal{T}(M)\simeq \mrm{Hom}_{\cOG, \vphi}(\cOG\otimes_{\cO} M, 
W(\mrm{Fr}R)_{\infty})$ 
(cf.\ Lemma \ref{Calem} (2)). 
Here $G$ acts on the right hand side by the formula 
$(\sigma.f)(x)=\sigma(f(\sigma^{-1}(x)))$ for 
$\sigma\in G, 
f\in \mrm{Hom}_{\cOG, \vphi}(\cOG\otimes_{\cO} M, W(\mrm{Fr}R)_{\infty}),
x\in \cOG\otimes_{\cO} M$.

We also define a contravariant functor 
${}'\hat{\mcal{M}}\colon 
\mrm{Rep}_{\mrm{fr}}(G)\to 
{}'\mbf{\Phi M}^{\hat{G}}_{/\cO}$ 
(resp.\ ${}'\hat{\mcal{M}}\colon 
\mrm{Rep}_{\mbb{Q}_p}(G)\to 
{}'\mbf{\Phi M}^{\hat{G}}_{/\E}$)
and 
${}'\hat{\mcal{T}}\colon 
'\mbf{\Phi M}^{\hat{G}}_{/\cO_{\infty}}\to \mrm{Rep}_{\mrm{fr}}(G)$ 
(resp.\ 
${}'\hat{\mcal{T}}\colon 
'\mbf{\Phi M}^{\hat{G}}_{/\E}\to \mrm{Rep}_{\mbb{Q}_p}(G)$)
by a similar manner.

Combining ${}'\hat{\mcal{T}},{}'\hat{\mcal{M}}$ with (\ref{epmod1}) or (\ref{epmod2}),
we obtain contravariant functors 
\[
\hat{\mcal{M}}\colon 
\mrm{Rep}_{\mrm{tor}}(G)\to 
\mbf{\Phi M}^{\hat{G}}_{/\cO_{\infty}},\
\hat{\mcal{M}}\colon 
\mrm{Rep}_{\mrm{fr}}(G)\to 
\mbf{\Phi M}^{\hat{G}}_{/\cO},\
\hat{\mcal{M}}\colon 
\mrm{Rep}_{\mbb{Q}_p}(G)\to 
\mbf{\Phi M}^{\hat{G}}_{/\E}
\]
and
\[
\hat{\mcal{T}}\colon 
\mbf{\Phi M}^{\hat{G}}_{/\cO_{\infty}}\to \mrm{Rep}_{\mrm{tor}}(G),\
\hat{\mcal{T}}\colon 
\mbf{\Phi M}^{\hat{G}}_{/\cO}\to \mrm{Rep}_{\mrm{fr}}(G),\
\hat{\mcal{T}}\colon 
\mbf{\Phi M}^{\hat{G}}_{/\E}\to \mrm{Rep}_{\mbb{Q}_p}(G).
\]

\begin{proposition}
\label{et-rep}
The contravariant functor $\hat{\mcal{T}}$
is an anti-equivalence of categories between
$\mbf{\Phi M}^{\hat{G}}_{/\cO_{\infty}}$ 
$($resp.\ $\mbf{\Phi M}^{\hat{G}}_{/\cO}$,
resp.\ $\mbf{\Phi M}^{\hat{G}}_{/\E}$)
and  $\mrm{Rep}_{\mrm{tor}}(G)$ 
$($resp.\ $\mrm{Rep}_{\mrm{fr}}(G)$,
resp.\  $\mrm{Rep}_{\mbb{Q}_p}(G))$).
Furthermore, $\hat{\mcal{M}}$ is a quasi-inverse of 
$\hat{\mcal{T}}$.
\end{proposition}

\begin{proof}
By Proposition \ref{Fon}, we have already known that,
for an \'etale $(\vphi,\hat{G})$-module $\hat{M}$
and a representation $T$ of $G$, 
canonical morphisms
$M\to \mcal{M}(\mcal{T}(M))$ and $T\to \mcal{T}(\mcal{M}(T))$
are isomorphisms as \'etale $\vphi$-modules and $G_{\infty}$-representations, respectively. 
It is enough to prove that the former is compatible with $\hat{G}$-action 
and the latter is $G$-equivalent.
In the following,
we only prove the torsion case; the same proof proceeds for rest cases.
It is enough to prove that functors ${}'\hat{\mcal{T}}$ and ${}'\hat{\mcal{M}}$
are inverses of each other.
Take any ${}'\hat{M}\in {}'\mbf{\Phi M}^{\hat{G}}_{/\cO_{\infty}}$.
We show a canonical isomorphism
\[
\eta\colon \cOG\otimes_{\cO} M\to \cOG\otimes_{\cO} \mcal{M}(\mcal{T}(M))
\]
is $\hat{G}$-equivalent.
By definitions of functors ${}'\hat{\mcal{T}}$ and ${}'\hat{\mcal{M}}$,
the following composition map
\begin{align*}
\cOG\otimes_{\cO} \mcal{M}(\mcal{T}(M))
&\overset{\sim}{\longrightarrow} 
\mrm{Hom}_{\mbb{Z}_p[H_{\infty}]}(\mcal{T}(M), W(\mrm{Fr}R)_{\infty})\\
&\overset{\sim}{\longrightarrow} 
\mrm{Hom}_{\mbb{Z}_p[H_{\infty}]}
(\mrm{Hom}_{\cOG, \vphi}(\cOG\otimes_{\cO} M, W(\mrm{Fr}R)_{\infty}), W(\mrm{Fr}R)_{\infty})
\end{align*}
is $\hat{G}$-equivalent. By composing this map with $\eta$,
we obtain the map
\[
\tilde{\eta}\colon \cOG\otimes_{\cO} M
\overset{\sim}{\longrightarrow}  
\mrm{Hom}_{\mbb{Z}_p[H_{\infty}]}
(\mrm{Hom}_{\cOG, \vphi}(\cOG\otimes_{\cO} M, W(\mrm{Fr}R)_{\infty}), W(\mrm{Fr}R)_{\infty})
\]
which is given by $x\mapsto (f\mapsto f(x))$ 
for $x\in \cOG\otimes_{\cO} M,\ 
f\in \mrm{Hom}_{\cOG, \vphi}(\cOG\otimes_{\cO} M, W(\mrm{Fr}R)_{\infty})$.
It is a straightforward calculation to check that $\tilde{\eta}$
is compatible with $\hat{G}$-actions, 
and thus $\eta$ is also.
Consequently,
we obtain the result that 
${}'\hat{\mcal{M}}\circ{}'\hat{\mcal{T}}\simeq \mrm{Id}$.
By a similar argument we can obtain  
${}'\hat{\mcal{T}}\circ{}'\hat{\mcal{M}}\simeq \mrm{Id}$
and this finishes the proof.

\end{proof}

\begin{remark}
By definitions of $\hat{\mcal{T}}$ and  $\hat{\mcal{M}}$ and the theory of Fontaine's
\'etale $\vphi$-modules,
we see that these functors  preserves various structures of categories.
For example, these functors are exact and 
commute with the formation of tensor products and duals.
Here the notion of the tensor product of \'etale $(\vphi,\hat{G})$-modules and 
that of dual \'etale $(\vphi,\hat{G})$-modules are defined by natural manners.
\end{remark}

\subsection{Link between Liu's $(\vphi,\hat{G})$-modules and \'etale $(\vphi,\hat{G})$-modules}

In this subsection,
we connect 
the theory of  Liu's $(\vphi,\hat{G})$-modules and 
the theory of our \'etale $(\vphi,\hat{G})$-modules.

Let $\hat{\mfM}=(\mfM,\vphi,\hat{G})$ be a $(\vphi, \hat{G})$-module, 
or a weak $(\vphi, \hat{G})$-module, 
in the sense of Definition \ref{Liumod}.
Extending a $\hat{G}$-action on $\whR \otimes_{\vphi,\mfS} \mfM$ to
$\cOG\otimes_{\whR} (\whR \otimes_{\vphi,\mfS} \mfM)$
by a natural way,
we see that $\mfM[1/u]=\cO\otimes_{\mfS} \mfM$ has a natural structure of 
an \'etale $(\vphi, \hat{G})$-module over $\cO$
(recall that $G$ acts on $W(\mrm{Fr}R)\otimes_{\vphi,\mfS} \mfM$ 
continuously for the weak topology by Definition \ref{Liumod}). 
This is the reason why a $\hat{G}$-action in the definition of an \'etale $(\vphi,\hat{G})$-module
is defined not on 
$\cOG\otimes_{\cO} M$ but on $\cOG\otimes_{\vphi, \cO} M$.
In the below, 
we denote by $\wh{\mfM[1/u]}$ the \'etale $(\vphi, \hat{G})$-module over $\cO$ obtained as the above.
Note that 
there exists a canonical isomorphism of $\mbb{Z}_p$-representations of $G$:
\[
\hat{T}(\hat{\mfM})\simeq \hat{\mcal{T}}(\wh{\mfM[1/u]}).
\]
In fact,
we have canonical isomorphisms
\begin{align*}
\hat{\mcal{T}}(\wh{\mfM[1/u]})&\simeq 
\mrm{Hom}_{\cOG,\vphi}(\cOG\otimes_{\vphi,\cO}(\mfM[1/u]), W(\mrm{Fr}R)_{\infty})\\
& \simeq \mrm{Hom}_{\whR,\vphi}(\whR\otimes_{\vphi,\mfS}\mfM, W(\mrm{Fr}R)_{\infty})\\
& \simeq \mrm{Hom}_{\whR,\vphi}(\whR\otimes_{\vphi,\mfS}\mfM, W(R)_{\infty})=\hat{T}(\hat{\mfM})
\end{align*} 
by Lemma \ref{Calem} (1) and Proposition B. 1.8.3 of \cite{Fo} (see also the proof of Corollary \ref{cov}).

In the below,
we want to use various morphisms between 
Liu's $(\vphi,\hat{G})$-modules and \'etale $(\vphi,\hat{G})$-modules.
To do this, we need to define some notions. 
Let $\mrm{Mod}(\vphi,\hat{G})$ be the category
whose objects are $\vphi$-modules $\mfM$ over $\mfS$ killed by a power of $p$ equipped with a $\cOG$-semilinear
$\hat{G}$-action on  $\cOG\otimes_{\vphi,\mfS} \mfM$.
Morphisms in $\mrm{Mod}(\vphi,\hat{G})$ are  defined by a natural manner.
Then categories 
${}_{\mrm{w}}\mrm{Mod}^{r,\hat{G}}_{/\mfS_{\infty}}$,  $\mrm{Mod}^{r,\hat{G}}_{/\mfS_{\infty}}$
and
$\mbf{\Phi M}^{\hat{G}}_{/\cO_{\infty}}$ 
can be regarded as full subcategories of $\mrm{Mod}(\vphi,\hat{G})$. 
We call a morphism $f\colon \mfM\to M$ in the category 
$\mrm{Mod}(\vphi,\hat{G})$ 
a {\it morphism of $(\vphi,\hat{G})$-modules}, 
and we often denote $f$ by $f\colon \hat{\mfM}\to \hat{M}$.

\begin{definition}
Let $\hat{\mfM}$ be in ${}_{\mrm{w}}\mrm{Mod}^{r,\hat{G}}_{/\mfS_{\infty}}$ 
or $\mrm{Mod}^{r,\hat{G}}_{/\mfS_{\infty}}$,
and $\hat{M}\in \mbf{\Phi M}^{\hat{G}}_{/\cO_{\infty}}$ equipped with 
a morphism $f\colon \hat{\mfM}\to \hat{M}$  of $(\vphi,\hat{G})$-modules. 
If $f$ is an injection as a $\mfS$-module morphism,
then $\hat{\mfM}$ can be regarded as a subobject of  $\hat{M}$
in the category $\mrm{Mod}(\vphi,\hat{G})$.
In this case,  (the image of) $\hat{\mfM}$ is called a sub $(\vphi,\hat{G})$-module of $\hat{M}$.

\end{definition}

\begin{proposition}[Analogue of scheme theoretic closure]
\label{astc}
Let $\hat{\mfM}$ be in ${}_{\mrm{w}}\mrm{Mod}^{r,\hat{G}}_{/\mfS_{\infty}}$  
$($resp.\ $\mrm{Mod}^{r,\hat{G}}_{/\mfS_{\infty}})$ and 
$\hat{M}\in \mbf{\Phi M}^{\hat{G}}_{/\cO_{\infty}}$. 
Let $f\colon \hat{\mfM}\to \hat{M}$ be a morphism of 
$(\vphi,\hat{G})$-modules.
Then, $\mrm{ker}(f)$ and $\mrm{im}(f)$ as $\vphi$-modules are in 
$\mrm{Mod}^{r}_{/\mfS_{\infty}}$.
Furthermore, the
$\hat{G}$-action on $\hat{\mfM}$
gives $\mrm{ker}(f)$ a structure of 
a weak $(\vphi,\hat{G})$-module
$($resp.\ a $(\vphi,\hat{G})$-module$)$ 
and the $\hat{G}$-action on $\hat{M}$
gives $\mrm{im}(f)$ a structure of 
a weak $(\vphi,\hat{G})$-module
$($resp.\ a $(\vphi,\hat{G})$-module$)$.

In this paper, 
we often denote $\wh{\mrm{im}(f)}$ by $f(\hat{\mfM})$ or $\wh{f(\mfM)}$.  
\end{proposition}  

\begin{proof}
The proof is same as that of Corollary \ref{imker}. 
\end{proof}  
  
The above proposition gives us a result on a successive extension 
for $(\vphi,\hat{G})$-modules, which is an analogue of Proposition \ref{torKi} (4).  
  
\begin{corollary}
\label{sucLi}
Let $\hat{\mfM}$ be in ${}_{\mrm{w}}\mrm{Mod}^{r,\hat{G}}_{/\mfS_{\infty}}$  
$($resp.\ $\mrm{Mod}^{r,\hat{G}}_{/\mfS_{\infty}})$.
Then there exists an extension 
\[
0=\mfM_0\subset \mfM_1\subset \cdots \subset \mfM_k=\mfM
\]
in $\mrm{Mod}^r_{/\mfS_{\infty}}$
which satisfies the following; for any $i$,

\noindent
$(\mrm{i})$ $\mfM_i/\mfM_{i-1}$ is a finite free $k[\![u]\!]$-module,

\noindent
$(\mrm{ii})$ $\mfM_i$ and $\mfM_i/\mfM_{i-1}$ have structures of 
weak $(\vphi,\hat{G})$-modules of height $r$
$($resp.\ $(\vphi,\hat{G})$-modules of height $r)$
which make a canonical exact sequence
\[
0\to \hat{\mfM}_{i-1}\to \hat{\mfM}_i\to \wh{\mfM_i/\mfM_{i-1}} \to0
\] 
in ${}_{\mrm{w}}\mrm{Mod}^{r,\hat{G}}_{/\mfS_{\infty}}$  
$($resp.\ $\mrm{Mod}^{r,\hat{G}}_{/\mfS_{\infty}})$.
\end{corollary}  
  
\begin{proof}
Putting $M=\mfM[1/u]$, we have seen that  
$\hat{M}=\wh{\mfM[1/u]}$ is an \'etale $(\vphi,\hat{G})$-module.
We see that $pM$ and $M/pM$ have structures of  \'etale $(\vphi,\hat{G})$-modules
and then there exists a natural exact sequence 
$0\to \wh{pM}\to \hat{M}\overset{\mrm{pr}}{\to} \wh{M/pM}\to 0$
of \'etale $(\vphi,\hat{G})$-modules.
We also denote by $\mrm{pr}$ a composition $\hat{\mfM}\to \hat{M} \overset{\mrm{pr}}{\to} \wh{M/pM}$
which is a morphism of $(\vphi,\hat{G})$-modules.
By Proposition \ref{astc},
we know that 
$\mfM'=\mrm{ker}(\mrm{pr|_{\mfM}})$ and $\mfM''=\mrm{pr}(\mfM)$ 
have structures of weak $(\vphi,\hat{G})$-modules of height $r$
(resp.\ $(\vphi,\hat{G})$-modules of height $r$)
and a canonical sequence 
$0\to \hat{\mfM}'\to \hat{\mfM}\to \hat{\mfM}''\to 0$
is exact in ${}_{\mrm{w}}\mrm{Mod}^{r,\hat{G}}_{/\mfS_{\infty}}$  
$($resp.\ $\mrm{Mod}^{r,\hat{G}}_{/\mfS_{\infty}})$.
Since $p^{n-1}\mfM'=0$ and $p\mfM''=0$,
we can obtain the desired extension inductively.
\end{proof}  
  
Before starting a maximal (minimal) theory,
we give one result on the ``cokernel'' of a morphism of $(\vphi,\hat{G})$-modules,
which will be used in the proof of Theorem \ref{MThm3}. 

\begin{proposition}
\label{cokernel}
Let $f\colon \hat{\mfM}\to \hat{\mfN}$ be a morphism in ${}_{\mrm{w}}\mrm{Mod}^{r,\hat{G}}_{/\mfS_{\infty}}$  
$($resp.\ $\mrm{Mod}^{r,\hat{G}}_{/\mfS_{\infty}})$. 
Denote by $\mrm{coker}(f)$ the cokernel of $f$ as a morphism of $\vphi$-modules.
Then $\frac{\mrm{coker}(f)}{u{\rm \mathchar`-tor}}$ is an object of $\mrm{Mod}^r_{/\mfS_{\infty}}$.
Furthermore, 
$\frac{\mrm{coker}(f)}{u{\rm \mathchar`-tor}}$ 
has a canonical structure of a weak $(\vphi,\hat{G})$-module 
$($resp.\ a $(\vphi,\hat{G})$-module$)$ induced from that of $\hat{\mfN}$.
\end{proposition}

\begin{proof}
It is enough to check the case where $f$ is a morphism in
$\mrm{Mod}^{r,\hat{G}}_{/\mfS_{\infty}}$.
Put $C=\mrm{coker}(f)$ and denote by $C_{u{\rm \mathchar`-tor}}$ the $u$-torsion part of $C$.
By Proposition \ref{torKi},
we see the fact that 
$\frac{C}{C_{u{\rm \mathchar`-tor}}}$ is an object of $\mrm{Mod}^r_{/\mfS_{\infty}}$.
Since $C$ is finitely generated as a $\mfS$-module,
there exists an integer $n>0$ such that $u^nC_{u{\rm \mathchar`-tor}}=0$.
Then $C'=u^nC$ is $u$-torsion free and thus 
$C'$ is a torsion Kisin module of finite height.
By Corollary \ref{modext},
we have that 
a natural map $\whR\otimes_{\vphi,\mfS} C'\to \whR\otimes_{\vphi,\mfS} C$
is injective.
Since the composition map
$\whR\otimes_{\vphi,\mfS} C\overset{1\otimes u^n}\to \whR\otimes_{\vphi,\mfS} C'\hookrightarrow \whR\otimes_{\vphi,\mfS} C$
is the multiplication-by-$u^{np}$ map,
if we  regard $\whR\otimes_{\vphi,\mfS} C'$
as a submodule of $\whR\otimes_{\vphi,\mfS} C$,
we obtain 
$u^{np}(\whR\otimes_{\vphi,\mfS} C)\subset \whR\otimes_{\vphi,\mfS} C'$.
Since $C'\in \mrm{Mod}^{\infty}_{/\mfS_{\infty}}$,
we know that 
$\whR\otimes_{\vphi,\mfS} C'\subset \cO\otimes_{\vphi,\mfS} C'$
and thus $\whR\otimes_{\vphi,\mfS} C'$ is $u$-torsion free.
Therefore, denoting by 
$(\whR\otimes_{\vphi,\mfS} C)_{u{\rm \mathchar`-tor}}$
the $u$-torsion part of $\whR\otimes_{\vphi,\mfS} C$,
we obtain 
\begin{equation}
\label{coktor}
u^{np}(\whR\otimes_{\vphi,\mfS} C)_{u{\rm \mathchar`-tor}}=0.
\end{equation}
The exact sequence
$0\to C_{u{\rm \mathchar`-tor}}\to C\overset{u^n}{\rightarrow} C'\to 0$
of $\mfS$-modules
induces the exact sequence
\begin{equation}
\label{cokex}
0\to \whR\otimes_{\vphi, \mfS} C_{u{\rm \mathchar`-tor}}\to \whR\otimes_{\vphi, \mfS}C
\overset{u^{np}}{\rightarrow} \whR\otimes_{\vphi, \mfS}C'\to 0 
\end{equation}
since $\mrm{Tor}^{\mfS}_1(C',\whR)=0$ (see Corollary \ref{exact}).
By (\ref{coktor}) and (\ref{cokex}),
we obtain the equality
$\whR\otimes_{\vphi,\mfS} C_{u{\rm \mathchar`-tor}}
=(\whR\otimes_{\vphi,\mfS} C)_{u{\rm \mathchar`-tor}}$
in $\whR\otimes_{\vphi, \mfS}C$.
On the other hand, 
we remark that 
$\hat{G}$-action on $\whR\otimes_{\vphi,\mfS} \mfN$ induces 
that on $\whR\otimes_{\vphi,\mfS} C$. 
Since $\hat{G}$-acts on $(\whR\otimes_{\vphi,\mfS} C)_{u{\rm \mathchar`-tor}}$ stable,
we can equip a $\hat{G}$-action on 
$\whR\otimes_{\vphi,\mfS} \frac{C}{C_{u{\rm \mathchar`-tor}}}$
by using the exact sequence  
$0\to \whR\otimes_{\vphi, \mfS} C_{u{\rm \mathchar`-tor}}\to \whR\otimes_{\vphi, \mfS}C
\to \whR\otimes_{\vphi, \mfS}\frac{C}{C_{u{\rm \mathchar`-tor}}}\to 0$.
Then it is not difficult to check that
$\frac{\mrm{coker}(f)}{u{\rm \mathchar`-tor}}=\frac{C}{C_{u{\rm \mathchar`-tor}}}$ 
is a $(\vphi,\hat{G})$-module. 
 
\end{proof}


\begin{remark}
\label{E1.1}
Let $0\to T' \to T \to T'' \to 0$ and $\hat{\mfM}$ be as in Theorem \ref{MThm2}.
Admitting notions of \'etale $(\vphi,\hat{G})$-modules, 
the proof of Theorem \ref{MThm2} implies that 
the sequence $(\ast): 0\to \hat{\mfM}''\to \hat{\mfM} \overset{g}{\rightarrow} \hat{\mfM}' \to 0$
appeared in the theorem is obtained by a natural way:
let $0\to \hat{M}''\to \hat{M} \to\hat{M}' \to 0$ be a 
sequence of \'etale $(\vphi,\hat{G})$-modules corresponding to $(\ast)$.
Then $\hat{\mfM}$ is a sub $(\vphi,\hat{G})$-module of $\hat{M}$
and $\mfM'=g(\mfM)$ 
(resp.\ $\mfM''=\mfM\cap M$) has a structure of a sub $(\vphi,\hat{G})$-module of $\hat{M}'$ (resp.\ $\hat{M}''$). 
\end{remark}

\subsection{Definitions of maximality and minimality}   

In this subsection, 
we construct maximal objects (resp.\ minimal objects) for $(\vphi,\hat{G})$-modules
by using the theory of \'etale $(\vphi,\hat{G})$-modules given in the previous section. 
Let $\hat{M}=(M,\vphi,\hat{G})\in \mbf{\Phi M}^{\hat{G}}_{/\cO_{\infty}}$
be a torsion \'etale $(\vphi, \hat{G})$-module over $\cO$.
We denote by $F^{r,\hat{G}}_{\mfS}(\hat{M})$  the (partially) ordered set (by inclusion)
of $\hat{\mfM}\in \mrm{Mod}^{r,\hat{G}}_{/\mfS_{\infty}}$ 
which is a sub $(\vphi,\hat{G})$-modules of an \'etale $(\vphi,\hat{G})$-module $M$
such that $\mfM[1/u]=M$.
Note that $\hat{\mfM}$ is a sub $(\vphi,\hat{G})$-modules of  $M$
if and only if 
a natural 
injection\footnote{A natural map $\whR\otimes_{\vphi, \mfS}\mfM\to \cOG\otimes_{\vphi, \cO} M$
is injective by Corollary \ref{ringext}.} 
$\whR\otimes_{\vphi, \mfS}\mfM\hookrightarrow \cOG\otimes_{\vphi, \cO} M$ is $\hat{G}$-equivalent.

\begin{lemma}
\label{sup}
Let $\hat{M}$ be a torsion \'etale $(\vphi,\hat{G})$-module.
Let $\hat{\mfM}_1$ and $\hat{\mfM}_2$ be in $ \mrm{Mod}^{r,\hat{G}}_{/\mfS_{\infty}}$
endowed with injections $\hat{\mfM}_1\to \hat{M}$ and 
$\hat{\mfM}_2\to \hat{M}$ of 
$(\vphi,\hat{G})$-modules. 
Then $\mfM_{12}=\mfM_1+\mfM_2$ $($resp.\ $\mfM'_{12}=\mfM_1\cap \mfM_2)$ in $M$ has a structure of 
a $(\vphi,\hat{G})$-module of height $r$.
In particular, 
the ordered set $F^{r,\hat{G}}_{\mfS}(M)$ has finite supremum and finite infimum. 
\end{lemma}

\begin{proof}
First we note that $\mfM_{12}$ (resp.\  $\mfM'_{12}$) is contained in $\mrm{Mod}^r_{/\mfS_{\infty}}$
and $\mfM_{12}[1/u]=M$ (resp.\ $\mfM'_{12}[1/u]=M$),
see the proof of Proposition 3.2.3 in \cite{CL1}.  
Furthermore,
$\mfM'_{12}$ is canonically isomorphic to the underlying Kisin module of 
the kernel of the morphism of $(\vphi,\hat{G})$-modules
\[
\hat{\mfM}_1\oplus \hat{\mfM}_2 \to \hat{\mfM}_1+ \hat{\mfM}_2\subset \hat{M},\quad (x,y)\mapsto x-y.  
\]  
Thus we obtain the desired result for $\mfM'_{12}$  by Proposition \ref{astc}.
Since $\hat{G}$-actions on $\whR\otimes_{\vphi, \mfS}\mfM_1$ and 
$\whR\otimes_{\vphi, \mfS}\mfM_2$ is restrictions of 
the $\hat{G}$-action on $\cOG\otimes_{\vphi, \cO} M$,
$\hat{G}$ acts on 
$\whR\otimes_{\vphi, \mfS}\mfM_{12}
=\whR\otimes_{\vphi, \mfS}\mfM_1+\whR\otimes_{\vphi, \mfS}\mfM_2
\subset \cOG\otimes_{\vphi, \cO} M$ stable.
For any $\sigma\in \hat{G}$ and $x\in \whR\otimes_{\vphi, \mfS}\mfM_{12}$,
taking $x_1\in \whR\otimes_{\vphi, \mfS}\mfM_1$ and $x_2\in \whR\otimes_{\vphi, \mfS}\mfM_2$ such that 
$x=x_1+x_2$,
we have 
$\sigma(x)-x=(\sigma(x_1)-x_1)+(\sigma(x_2)-x_2)
\in I_+(\whR\otimes_{\vphi,\mfS} \mfM_1)+I_+(\whR\otimes_{\vphi,\mfS} \mfM_2)
=I_+(\whR\otimes_{\vphi,\mfS} \mfM_{12})$ and thus 
$\hat{G}$ acts on $(\whR\otimes_{\vphi,\mfS} \mfM_{12})/I_+(\whR\otimes_{\vphi,\mfS} \mfM_{12})$
trivial.
Hence $\hat{\mfM}_{12}=(\mfM_{12},\vphi,\hat{G})$ is a $(\vphi,\hat{G})$-module
and we obtain the desired result.
\end{proof}

\begin{proposition}
\label{greatest}
$F^{r,\hat{G}}_{\mfS}(\hat{M})$ has a maximum element.
If $r<\infty$, then it also has a minimum element.
\end{proposition}
\begin{proof}
Suppose that $F^{r,\hat{G}}_{\mfS}(\hat{M})$ does not have a maximum element.
Take any $\hat{\mfM}=\hat{\mfM}_0\in F^{r,\hat{G}}_{\mfS}(\hat{M})$.
Since $\hat{\mfM}_0$ is not maximum, 
there exists a $\hat{\mfM}'_1\in F^{r,\hat{G}}_{\mfS}(M)$ such that 
$\mfM_0\not\subset \mfM'_1$.
Put $\mfM_1=\mfM_0+\mfM'_1$ in $M$.
By Lemma \ref{sup}, $\mfM_1$ has a structure of $(\vphi,\hat{G})$-module,
we denote it by $\hat{\mfM}_1$.
We see that $\hat{\mfM}_1\in F^{r,\hat{G}}_{\mfS}(\hat{M})$ 
and $\mfM_0\subsetneq \mfM_1$.
Inductively,
we find  $\hat{\mfM}_i\in F^{r,\hat{G}}_{\mfS}(M)$ with
infinite length 
increasing sequence
\[
\mfM_0\subsetneq \mfM_1\subsetneq \mfM_2 \subsetneq \cdots
\]
in $F^{r}_{\mfS}(M)$.
However, this is a contradiction by Lemma 3.2.4 of \cite{CL1}.
The proof of  the assertion for a minimum element is the same except only 
that we use Lemma 3.2.5 of \cite{CL1}.
\end{proof}

\begin{remark}
If $F^{\infty,\hat{G}}_{\mfS}(\hat{M})$ is not empty,
then $F^{\infty,\hat{G}}_{\mfS}(\hat{M})$ does not have a minimum element.
In fact, if $\hat{\mfM}$ is an object of $F^{\infty,\hat{G}}_{\mfS}(\hat{M})$,
then we obtain the infinite decreasing sequence 
\[
\hat{\mfM}> \wh{u\mfM}> \wh{u^2\mfM}> \cdots 
\]
in $F^{\infty,\hat{G}}_{\mfS}(\hat{M})$.

\end{remark}

\begin{definition}
\label{maxdef}
Let $\hat{\mfM}\in \mrm{Mod}^{r,\hat{G}}_{/\mfS_{\infty}}$. 
We denote by $\Max^r(\hat{\mfM})$ (resp.\ $\Min^r(\hat{\mfM})$)
the maximum element (resp.\ minimum element)
of $F^{r,\hat{G}}_{\mfS}(\wh{\mfM[1/u]})$.
It is endowed with a morphism of $(\vphi,\hat{G})$-modules 
$\iotamax^{\hat{\mfM}}\colon  \hat{\mfM}\to \Max^r(\hat{\mfM})$ 
(resp.\ $\iotamin^{\hat{\mfM}}\colon  \Min(\hat{\mfM})\to \hat{\mfM}$).
We often denote by $\mMax^r(\hat{\mfM})$ (resp.\ $\mMin^r(\hat{\mfM})$) 
the underlying sub $\vphi$-module over $\mfS$ of 
$\Max^r(\hat{\mfM})$ (resp.\ $\Min^r(\hat{\mfM})$). 
We say  that $\hat{\mfM}$ is {\it maximal} (resp.\ {\it minimal}) 
if $\iotamax^{\hat{\mfM}}$ (resp.\ $\iotamin^{\hat{\mfM}}$) is an isomorphism.
\end{definition}

\subsection{Maximal objects for $(\vphi,\hat{G})$-modules}   

In this section, we prove various properties of maximal objects.

\begin{proposition}
\label{func}
Definition \ref{maxdef} gives rise to a functor 
$\Max^r\colon \mrm{Mod}^{r,\hat{G}}_{/\mfS_{\infty}}\to \mrm{Mod}^{r,\hat{G}}_{/\mfS_{\infty}}$.
\end{proposition}
\begin{proof}
We have to prove that 
any map $f\colon \hat{\mfM}\to \hat{\mfM}'$ induces a map
$\Max^r(\hat{\mfM})\to \Max^r(\hat{\mfM}')$.
The map $g=f[1/u]\colon \wh{\mfM[1/u]}\to \wh{\mfM'[1/u]}$ is a morphism in 
$\mbf{\Phi M}^{\hat{G}}_{/\cO_{\infty}}$.
By Corollary \ref{astc},
$g(\Max^r({\hat{\mfM}}))$ is a sub $(\vphi,\hat{G})$-module  over $\mfS$ of $\wh{\mfM'[1/u]}$.
Since $\hat{\mfM}'$ is maximal and 
$g(\Max^r({\hat{\mfM}}))+\hat{\mfM}'$ is an object of $F^{r,\hat{G}}_{\mfS}(\wh{\mfM'[1/u]})$, we see the underlying 
$\vphi$-module of $g(\Max^r({\hat{\mfM}}))$ is contained in  
$\mfM'$
and we have done.
\end{proof}

Denote by $\Max^{r,\hat{G}}_{/\mfS_{\infty}}$ the essential image of the functor
$\Max^r\colon \mrm{Mod}^{r,\hat{G}}_{/\mfS_{\infty}}\to \mrm{Mod}^{r,\hat{G}}_{/\mfS_{\infty}}$.
It is a full subcategory of $\mrm{Mod}^{r,\hat{G}}_{/\mfS_{\infty}}$.
The following two propositions can be proved by essentially the same method of \cite{CL1} 
(cf.\ Proposition 3.3.2,\ 3.3.3,\ 3.3.4 and 3.3.5) and we omit the proof. 

\begin{proposition}
\label{property}
$(1)$
The functor
$\Max^r\colon \mrm{Mod}^{r,\hat{G}}_{/\mfS_{\infty}}\to \mrm{Mod}^{r,\hat{G}}_{/\mfS_{\infty}}$ is a projection, that is, 
$\Max^r\circ \Max^r=\Max^r$.

\noindent
$(2)$ The functor 
$\Max^r\colon \mrm{Mod}^{r,\hat{G}}_{/\mfS_{\infty}}\to \mrm{Mod}^{r,\hat{G}}_{/\mfS_{\infty}}$
is left exact.

\noindent
$(3)$ The functor 
$\Max^r\colon \mrm{Mod}^{r,\hat{G}}_{/\mfS_{\infty}}\to \Max^{r,\hat{G}}_{/\mfS_{\infty}}$
is a left adjoint to the inclusion functor
$\Max^{r,\hat{G}}_{/\mfS_{\infty}}\to \mrm{Mod}^{r,\hat{G}}_{/\mfS_{\infty}}$.
\end{proposition}

\begin{proposition}
\label{univ}
Let $\hat{\mfM}\in \mrm{Mod}^{r,\hat{G}}_{/\mfS_{\infty}}$.
Then the couple $(\Max^r(\hat{\mfM}), \iota^{\hat{\mfM}}_{\mrm{max}})$ is characterized by the 
following universal property:
\begin{itemize}
\item the morphism $\hat{T}(\iota^{\hat{\mfM}}_{\mrm{max}})$ is an isomorphism;
\item for any $\hat{\mfM}'\in  \mrm{Mod}^{r,\hat{G}}_{/\mfS_{\infty}}$ endowed with a morphism 
$f\colon \hat{\mfM}\to \hat{\mfM}'$ such that $\hat{T}(f)$ is an isomorphism,
there exists a unique $g\colon \hat{\mfM}'\to \Max^r(\hat{\mfM})$
such that $g\circ f=\iota^{\hat{\mfM}}_{\mrm{max}}$.
\end{itemize} 
\end{proposition}

Here we are ready to prove the essential part of Theorem \ref{MThm3}.

\begin{theorem}
\label{abelian}
The category $\Max^{r,\hat{G}}_{/\mfS_{\infty}}$ is abelian.
More precisely,
if $f\colon \hat{\mfM}\to \hat{\mfM}'$ is a morphism in 
$\Max^{r,\hat{G}}_{/\mfS_{\infty}}$, then

\vspace{-2mm}
\begin{itemize}
\item[$(1)$] if we denote the kernel of $f$ as a morphism of $\vphi$-modules by $\mrm{ker}(f)$, 
then $\mrm{ker}(f)$ is an object of $\mrm{Mod}^r_{/\mfS_{\infty}}$
and has a canonical structure of a $(\vphi,\hat{G})$-module of height $r$.
If we denote it by $\wh{\mrm{ker}(f)}$, 
then it is maximal
and is the kernel of $f$ in the abelian category $\Max^{r,\hat{G}}_{/\mfS_{\infty}}$;

\vspace{-2mm}
\item[$(2)$] if we denote the cokernel of $f$ as a morphism of $\vphi$-modules by $\mrm{coker}(f)$,
then $\frac{\mrm{coker}(f)}{u{\rm \mathchar`-tor}}$ is an object of $\mrm{Mod}^r_{/\mfS_{\infty}}$
and has a canonical structure of a $(\vphi,\hat{G})$-module of height $r$.
If we denote it by $\wh{\frac{\mrm{coker}(f)}{u{\rm \mathchar`-tor}}}$, then
$\Max^r(\wh{\frac{\mrm{coker}(f)}{u{\rm \mathchar`-tor}}})$ is the cokernel of $f$ in  the abelian category
$\Max^{r,\hat{G}}_{/\mfS_{\infty}}$;
moreover, if $f$ is injective as a morphism of $\vphi$-modules, 
then $\mrm{coker}(f)$ has no $u$-torsion;

\vspace{-2mm}
\item[$(3)$] if we denote the image $($resp.\ coimage$)$ of $f$ as a morphism of $\vphi$-modules 
by $\mrm{im}(f)$ $($resp.\ $\mrm{coim}(f))$, 
then $\mrm{im}(f)$ $($resp.\ $\mrm{coim}(f))$ 
is an object of $\mrm{Mod}^r_{/\mfS_{\infty}}$
and has a canonical structure of a $(\vphi,\hat{G})$-module of height $r$.
If we denote it by $\wh{\mrm{im}(f)}$ $($resp.\ $\wh{\mrm{coim}(f)})$, 
then $\Max^r(\wh{\mrm{im}(f)})$  $($resp.\ $\Max^r(\wh{\mrm{coim}(f)}))$ 
is the image $($resp.\ coimage$)$ of $f$ in the abelian category $\Max^{r,\hat{G}}_{/\mfS_{\infty}}$.
\end{itemize}
\end{theorem}
\begin{proof}
(1) By Corollary \ref{imker},
we know that 
$\mrm{ker}(f)$ has a structure of 
$(\vphi,\hat{G})$-module of height $r$.
We have to show that $\wh{\mrm{ker}(f)}$ is maximal.
Consider the diagram below:
\begin{center}
$\displaystyle \xymatrix{
0\ar[r] & \mrm{ker}(f)\ar[r]\ar@{^{(}->}[d] & 
\mfM\ar[r] \ar@{^{(}->}[d] & 
\mfM' \ar@{^{(}->}[d]\\
0\ar[r] & \mMax^r(\wh{\mrm{ker}(f)})\ar[r]\ar@{^{(}->}[d] & 
\mMax^r(\wh{\mfM})\ar[r] \ar@{^{(}->}[d] & 
\mMax^r(\wh{\mfM}') \ar@{^{(}->}[d]\\
0\ar[r] & \mrm{ker}(f)[1/u]\ar[r] & 
\mfM[1/u]\ar[r] & 
\mfM'[1/u]. 
}$
\end{center}
Top and bottom horizontal sequences
are exact as $\vphi$-modules over $\mfS$.
Put $\mfM_{\mrm{max}}=\mMax^r(\wh{\mrm{ker}(f)})
+\mfM$ in $\mfM[1/u]$ and 
observe that $\mfM_{\mrm{max}}\in  
\mrm{Mod}^r_{/\mfS_{\infty}}$
and $\mfM_{\mrm{max}}$ has a structure of 
a $(\vphi,\hat{G})$-module
with injection of $\hat{G}$-modules
$\wh{\mfM_{\mrm{max}}}\hookrightarrow \wh{\mfM[1/u]}$.
Since $\mfM\subset \mfM_{\mrm{max}}\subset \mfM[1/u]$,
we have $\mfM_{\mrm{max}}[1/u]=\mfM[1/u]$ and thus
$\wh{\mfM_{\mrm{max}}}\in F^{r,\hat{G}}_{\mfS}(\wh{\mfM[1/u]})$.
Since $\hat{\mfM}$ is maximal,
we obtain $\mfM_{\mrm{max}}\subset \mfM$.
Therefore,
we have 
$\mfM_{\mrm{max}}\subset \mfM\cap \mrm{ker}(f)[1/u]=
\mrm{ker}(f)$ (where the equality $\mfM\cap \mrm{ker}(f)[1/u]=
\mrm{ker}(f)$ follows from the above diagram)
and hence $\wh{\mrm{ker}(f)}$ is maximal.

\noindent
(2) By Proposition \ref{cokernel}, we know that 
$\frac{\mrm{coker}(f)}{u\mrm{-tor}}$ has  a structure of a $(\vphi,\hat{G})$-module
of height $r$ induced from that of $\hat{\mfM}'$.
By Proposition \ref{property} (3), 
to check  $\Max^r(\wh{\frac{\mrm{coker}(f)}{u\mrm{-tor}}})$ is the cokernel
of $f$ in the category $\Max^{r,\hat{G}}_{/\mfS_{\infty}}$ is not difficult.

Next we prove the latter assertion;
suppose $f\colon \hat{\mfM}\to \hat{\mfM}'$ is 
injective as a morphism of $\vphi$-modules.
Put $C=\mrm{coker}(f)$ (as a $\mfS$-module).
The following diagram of exact sequences of $\vphi$-modules
are commutative;
\begin{center}
$\displaystyle \xymatrix{
0\ar[r] & \mfM\ar[r] \ar@{^{(}->}[d]&
\mfM' \ar[r] \ar@{^{(}->}[d]& 
C \ar[r] \ar^{g}[d] &  0\\ 
0\ar[r] & \mfM[1/u]\ar[r] &
\mfM'[1/u] \ar[r] & 
C[1/u] \ar[r] & 0.
}$
\end{center}
Put $\mfN=\mfM[1/u]\cap \mfM'$.
We claim that $\mfM=\mfN$. If we admit this claim,
we see that $g$ is injective and thus $C$ is $u$-torsion free,
which is the desired result.
Hence it suffices to prove the claim.
The inclusion $\mfM\subset \mfN$ is clear.
To prove $\mfN\subset \mfM$,
it is enough to prove that 
$\mfN$ has a structure of a $(\vphi,\hat{G})$-module 
and $\hat{\mfN}\in F^{r,\hat{G}}_{\mfS}(\wh{\mfM[1/u]})$.
By the proof of Proposition 3.3.4 of \cite{CL1},
we know that $\mfN\in 
\mrm{Mod}^r_{/\mfS_{\infty}}$.
Furthermore, 
we see that $\mfN[1/u]=\mfM[1/u]$ since
$\mfM\subset \mfN\subset \mfM[1/u]$.
If we denote by $C'$ the cokernel of the inclusion map 
$\mfN\hookrightarrow \mfM'$,
then we know that $C'[1/u]=C[1/u]$ and $\mfM'\hookrightarrow \mfM'[1/u]$
induces an injection $C'\hookrightarrow C'[1/u]$, in particular,
$C'$ is $u$-torsion free and $C'\in \mrm{Mod}^r_{/\mfS_{\infty}}$.
By Corollary \ref{ringext} and \ref{exact},
two horizontal sequences of the diagram
\begin{center}
$\displaystyle \xymatrix{
0\ar[r] & \whR\otimes_{\vphi,\mfS}\mfN\ar[r] \ar[d] & 
\whR\otimes_{\vphi,\mfS} \mfM' \ar[r] \ar[d] & 
\whR\otimes_{\vphi,\mfS} C' \ar[r] \ar[d] &
0
\\
0\ar[r] & \whR\otimes_{\vphi,\mfS}(\mfN[1/u]) \ar[r] & 
\whR\otimes_{\vphi,\mfS} (\mfM'[1/u]) \ar[r] & 
\whR\otimes_{\vphi,\mfS} (C'[1/u]) \ar[r] &
0.
}$
\end{center}
are exact as $\whR$-modules and all vertical arrows are injective.
Since $\mfN[1/u]=\mfM[1/u]$, we obtain the equality
\[
\whR\otimes_{\vphi,\mfS}\mfN = 
(\whR\otimes_{\vphi,\mfS}(\mfN[1/u]))\cap  (\whR\otimes_{\vphi,\mfS} \mfM') 
\]
in $\whR\otimes_{\vphi,\mfS} (\mfM'[1/u])$.
It is not difficult to check that 
the $\hat{G}$-action on $\whR\otimes_{\vphi,\mfS}\mfM$ extends to 
$\whR\otimes_{\vphi,\mfS}(\mfM[1/u])$, 
which coincides to the restriction of 
the $\hat{G}$-action on $\cOG\otimes_{\vphi,\mfS} (\mfM'[1/u])$.
Hence
the $\hat{G}$-action for $\cOG\otimes_{\vphi,\mfS} (\mfM'[1/u])$
is stable under $\whR\otimes_{\vphi,\mfS}\mfN$
and $\mfN$ has a structure of a weak sub $(\vphi,\hat{G})$-module of $\hat{\mfM}'$.
Since $C'\in \mrm{Mod}^r_{/\mfS_{\infty}}$,
the exact sequence 
$0\to \whR\otimes_{\vphi,\mfS}\mfN \to
\whR\otimes_{\vphi,\mfS} \mfM' \to 
\whR\otimes_{\vphi,\mfS} C' \to 0$
allows  $C'$ to have a structure of a weak $(\vphi,\hat{G})$-module.
By Corollary \ref{subweak},
we know that
$\hat{\mfN}$ is in fact a $(\vphi,\hat{G})$-module.
Therefore, maximality of $\hat{\mfM}$ implies that 
$\mfN\subset \mfM$. This proves the claim and we finish the proof of the latter assertion of (2).

\noindent
(3) 
Let $f\colon \hat{\mfM}\to \hat{\mfM}'$ be a morphism in 
$\Max^{r,\hat{G}}_{/\mfS_{\infty}}$.
Corollary \ref{imker} says that 
$\mrm{im}(f)$ has a structure of a sub $(\vphi,\hat{G})$-module of $\hat{\mfM}'$.
The map $f$ induces a map $g\colon \wh{\mrm{im}(f)}\to \hat{\mfM}'$.
It is clear that $\mrm{coker}(f)=\mrm{coker}(g)$ as $\mfS$-modules.
By (2) and Proposition \ref{cokernel}, for the map
$\Max^r(g)\colon \Max(\wh{\mrm{im}(f)})\to \hat{\mfM}'$,
$\mrm{coker}(\Max^r(g))$ (as a $\mfS$-module) is $u$-torsion free
and it has a structure of $(\vphi,\hat{G})$-modules induced from that of $\hat{\mfM}'$.
Note that there exists a canonical isomorphism 
$\wh{\mrm{coker}(\Max^r(g))}\simeq \wh{\frac{\mrm{coker}(f)}{u\mrm{-tor}}}$
as $(\vphi,\hat{G})$-modules.
See the exact sequence of $(\vphi,\hat{G})$-modules
\[
0\to \Max^r(\wh{\mrm{im}(f)})\to \hat{\mfM}'\to \wh{\mrm{coker}(\Max^r(g))}\to 0.
\]
Since the functor $\Max^r\colon \mrm{Mod}^{r,\hat{G}}_{/\mfS_{\infty}}\to 
\mrm{Mod}^{r,\hat{G}}_{/\mfS_{\infty}}$
is left exact (cf.\ Proposition \ref{property}),
we obtain the exact sequence of $(\vphi,\hat{G})$-modules
\[
0\to \Max^r(\wh{\mrm{im}(f)})\to \hat{\mfM}'\to \Max^r(\wh{\mrm{coker}(\Max^r(g))}).
\]
Combining this with the description of kernels and cokernels in the category 
$\Max^{r,\hat{G}}_{/\mfS_{\infty}}$,
we obtain the fact that $\Max^r(\wh{\mrm{im}(f)})$ is the image of $f$
 in the category 
$\Max^{r,\hat{G}}_{/\mfS_{\infty}}$.
The assertion for the coimage can be checked by a similar way.
\end{proof}

\begin{lemma}
\label{exff}
If $\alpha\colon \hat{\mfM}'\to \hat{\mfM}$ and $\beta\colon \hat{\mfM}\to \hat{\mfM}''$
two morphisms in $\Max^{r,\hat{G}}_{/\mfS_{\infty}}$ such that 
$\beta\circ \alpha=0$.
the sequence $0\to \hat{\mfM}'\to \hat{\mfM}\to \hat{\mfM}''\to 0$ is exact 
in $($the abelian category$)$
$\Max^{r,\hat{G}}_{/\mfS_{\infty}}$ if and only if 
$0\to \wh{\mfM'[1/u]}\to \wh{\mfM[1/u]}\to \wh{\mfM''[1/u]}\to 0$ 
is exact in $\mbf{\Phi M}^{\hat{G}}_{/\cO_{\infty}}$.  
Furthermore, the functor 
\[
\Max^{r,\hat{G}}_{/\mfS_{\infty}}\to 
\mbf{\Phi M}^{\hat{G}}_{/\cO_{\infty}},\quad \hat{\mfM}\mapsto \wh{\mfM[1/u]}
\]
is fully faithful.
\end{lemma}

\begin{proof}
Since $\alpha$ and $\beta$ is assumed to be $\hat{G}$-equivalent,
$0\to \wh{\mfM'[1/u]}\to \wh{\mfM[1/u]}\to \wh{\mfM''[1/u]}\to 0$ 
is exact in $\mbf{\Phi M}^{\hat{G}}_{/\cO_{\infty}}$
if and only if 
$0\to \mfM'[1/u]\to \mfM[1/u]\to \mfM''[1/u]\to 0$ 
is exact in $\mbf{\Phi M}_{/\cO_{\infty}}$.
Thus the proof is the same as that of Lemma 3.3.9 in \cite{CL1}.
\end{proof}

\begin{corollary}
\label{exf}
The functor $\hat{T}$ defined on $\Max^{r,\hat{G}}_{/\mfS_{\infty}}$ is exact and fully faithful, 
and its essential image is stable under taking a subquotient.
\end{corollary}

\begin{proof}
The former assertion follows from 
the commutative triangle
\begin{center}
$\displaystyle \xymatrix{
 \Max^{r,\hat{G}}_{/\mfS_{\infty}} \ar^{\hat{T}}[rr] \ar[rd] 
 & 
 & 
 \mrm{Rep}_{\mrm{tor}}(G_{\infty}) \\ 
 & \mbf{\Phi M}^{\hat{G}}_{/\cO_{\infty}} \ar^{\simeq}_{\hat{\mcal{T}}}[ru]
 &  
}$
\end{center}
where $\Max^{r,\hat{G}}_{/\mfS_{\infty}}\to 
\mbf{\Phi M}^{\hat{G}}_{/\cO_{\infty}}$ is  the functor defined by the assignment 
$\hat{\mfM}\mapsto \wh{\mfM[1/u]}$ which is exact and fully faithful (by Lemma \ref{exff}).
The latter assertion follows from Theorem \ref{MThm2}.
\end{proof}

\begin{corollary}
\label{maxex}
The functor $\Max^r\colon \mrm{Mod}^{r,\hat{G}}_{/\mfS_{\infty}}\to \Max^{r,\hat{G}}_{/\mfS_{\infty}}$
is exact.
\end{corollary}

\begin{proof}
This follows from Lemma \ref{exff}.
\end{proof}

\begin{proposition}
\label{maximalcri1}
The category $\Max^{r,\hat{G}}_{/\mfS_{\infty}}$ is stable under the extension in 
$\mrm{Mod}^{r,\hat{G}}_{/\mfS_{\infty}}$, that is,
if 
\[
0\to \hat{\mfM}'\to \hat{\mfM}\to \hat{\mfM}''\to 0
\]
is an exact sequence in $\mrm{Mod}^{r,\hat{G}}_{/\mfS_{\infty}}$ with 
$\hat{\mfM}', \hat{\mfM}''\in \Max^{r,\hat{G}}_{/\mfS_{\infty}}$,
then $\hat{\mfM}\in \Max^{r,\hat{G}}_{/\mfS_{\infty}}$.
\end{proposition}
\begin{proof}
The proof is essentially the same as that of 
Proposition 3.3.13 in \cite{CL1}.
\end{proof}

\begin{proposition}
\label{maximalcri2}
Let $\hat{\mfM}\in \mrm{Mod}^{r,\hat{G}}_{/\mfS_{\infty}}$
and $\mrm{id}\otimes \vphi\colon \mfS\otimes_{\vphi,\mfS} \mfM\to \mfM$ the $\mfS$-linearization
of $\vphi$.
If $\mrm{coker}(\mrm{id}\otimes \vphi)$ is killed by $u^{p-2}$ then $\hat{\mfM}$
is maximal.
\end{proposition}

\begin{proof}
By Corollary \ref{sucLi} and Proposition \ref{maximalcri1},
we can reduce the proof to the case 
where $p\mfM=0$, and then
the proof is essentially the same as that of 
Lemma 3.3.14 in \cite{CL1}.
\end{proof}

\begin{remark}
\label{gen}
All results in this subsection hold even if we replace ``$(\vphi,\hat{G})$-modules'' with 
``weak $(\vphi,\hat{G})$-modules'' 
(e.g. the existence of maximal objects for weak $(\vphi,\hat{G})$-modules).
Proofs are easier than that for ``$(\vphi,\hat{G})$-modules'' since we may omit ``modulo $I_+$'' arguments.  
\end{remark}

\subsection{Minimal objects for $(\vphi,\hat{G})$-modules} 

{\it Throughout this subsection, 
we always assume that $r<\infty$}.
Here we study minimal objects of $(\vphi,\hat{G})$-modules.
Many arguments in this subsection are very similar to those of the maximal case and of \cite{CL1}.
\begin{proposition}
\label{func2}
Definition \ref{maxdef} gives rise to a functor 
$\Min^r\colon \mrm{Mod}^{r,\hat{G}}_{/\mfS_{\infty}}\to \mrm{Mod}^{r,\hat{G}}_{/\mfS_{\infty}}$.
\end{proposition}

\begin{proof}
We have to show that 
any morphism $f\colon \hat{\mfM}\to \hat{\mfN}$ in $\mrm{Mod}^{r,\hat{G}}_{/\mfS_{\infty}}$
embeds $\min^r(\hat{\mfM})$
into $\min^r(\hat{\mfN})$. 
Put $\hat{M}=\wh{\mfM[1/u]}$ and 
$\hat{N}=\wh{\mfN[1/u]}$.
Denote by $g=f[1/u]\colon \hat{M}\to \hat{N}$ the morphism induced from $f$.
Then $g$ induces $\Max^r(f)\colon \Max^r(\mfM)\to \Max^r(\mfN)$, we also denote it by $g$.
We know that 
the kernel $\mfK$ of the map
\[
h\colon \Max^r(\hat{\mfM})\oplus \Min^r(\hat{\mfN})\to \Max^r(\hat{\mfN}),\quad (x,y)\mapsto g(x)-y 
\] 
has a structure as a $(\vphi,\hat{G})$-module $\hat{\mfK}$ of height $r$.
Note that the composition map $\hat{\mfK}\to \Max^r(\hat{\mfM})\oplus \Min^r(\hat{\mfN})\to \Max^r(\hat{\mfM})$
is an isomorphism, where the first arrow is the natural embedding and the second arrow is a first projection.
In particular, we obtain an isomorphism $\eta\colon \wh{\mfK[1/u]}\overset{\sim}{\longrightarrow} \hat{M}$.
If we identify $\wh{\mfK[1/u]}$ and $\hat{M}$ via $\eta$,
then $\hat{\mfK}$ is contained in $F^{r,\hat{G}}_{\mfS}(\hat{M})$
and thus $\min^r(\mfM)\subset \mfK$.
Taking any element $x=(x,y)$ of $\min^r(\mfM)\subset \mfK$,
we have $h(x,y)=0$ and thus $g(x)=y\in \min^r(\hat{\mfN})$. This finishes the proof.
\end{proof}

Denote by $\Min^{r,\hat{G}}_{/\mfS_{\infty}}$
the essential image of the functor 
$\Min^r\colon \mrm{Mod}^{r,\hat{G}}_{/\mfS_{\infty}}\to \mrm{Mod}^{r,\hat{G}}_{/\mfS_{\infty}}$.
The following can be checked by the same way as that of Proposition 3.4.6 of \cite{CL1}.

\begin{proposition}
\label{univ2}
Let $\hat{\mfM}\in \mrm{Mod}^{r,\hat{G}}_{/\mfS_{\infty}}$.
Then the couple $(\Min^r(\hat{\mfM}), \iota^{\hat{\mfM}}_{\mrm{min}})$ is characterized by the 
following universal property:
\begin{itemize}
\item the morphism $\hat{T}(\iota^{\hat{\mfM}}_{\mrm{min}})$ is an isomorphism;
\item for any $\hat{\mfM}'\in  \mrm{Mod}^{r,\hat{G}}_{/\mfS_{\infty}}$ endowed with a morphism 
$f\colon \hat{\mfM}'\to \hat{\mfM}$ such that $\hat{T}(f)$ is an isomorphism,
there exists a unique $g\colon \Min^r(\hat{\mfM})\to \hat{\mfM}'$
such that $f\circ g=\iota^{\hat{\mfM}}_{\mrm{min}}$.
\end{itemize} 
\end{proposition}

\noindent
Since the couple $(\Max^r(\hat{\mfM}^{\vee})^{\vee}, (\iota^{\hat{\mfM}^{\vee}}_{\mrm{max}})^{\vee})$
satisfies the universality appeared in Proposition \ref{univ2}, we obtain

\begin{corollary}
\label{maxmin}
For $\hat{\mfM}\in \mrm{Mod}^{r,\hat{G}}_{/\mfS_{\infty}}$,
we have natural isomorphisms
\[
\Min^r(\hat{\mfM}^{\vee})\simeq \Max^r(\hat{\mfM})^{\vee}\ \ \mrm{and}\ \   
\Max^r(\hat{\mfM}^{\vee})\simeq \Min^r(\hat{\mfM})^{\vee}.
\]
In particular,
duality permutes subcategories 
$\Max^{r,\hat{G}}_{/\mfS_{\infty}}$ and $\Min^{r,\hat{G}}_{/\mfS_{\infty}}$.
\end{corollary}

The following proposition can be proved by essentially the same method of \cite{CL1} 
(cf.\ Proposition 3.4.3,\ 3.4.8, Lemma 3.4.4 and Corollary 3.4.5) and we omit the proof.

\begin{proposition}
\label{property2}
$(1)$
The functor
$\Min^r\colon \mrm{Mod}^{r,\hat{G}}_{/\mfS_{\infty}}\to \mrm{Mod}^{r,\hat{G}}_{/\mfS_{\infty}}$ is a projection, that is, 
$\Min^r\circ \Min^r=\Min^r$.

\noindent
$(2)$ Let $f\colon \hat{\mfM}\to \hat{\mfN}$ be a morphism in 
$\mrm{Mod}^{r,\hat{G}}_{/\mfS_{\infty}}$.
Then $f(\Min^r(\hat{\mfM}))=\Min^r(\wh{f(\mfM)})$.
$($For some notations, see Proposition \ref{astc}.$)$

\noindent
$(3)$ Let $f\colon \hat{\mfM}\to \hat{\mfN}$ be a morphism in 
$\mrm{Mod}^{r,\hat{G}}_{/\mfS_{\infty}}$.
If $f$ is surjective $($resp.\ injective$)$ as a $\mfS$-module morphism,
then $\Min^r(f)$ is also.

\noindent
$(4)$ The functor 
$\Min^r\colon \mrm{Mod}^{r,\hat{G}}_{/\mfS_{\infty}}\to \Min^{r,\hat{G}}_{/\mfS_{\infty}}$
is a right adjoint to the inclusion functor
$\Min^{r,\hat{G}}_{/\mfS_{\infty}}\to \mrm{Mod}^{r,\hat{G}}_{/\mfS_{\infty}}$.
\end{proposition}

\begin{theorem}
\label{abelian2}
The category $\Min^{r,\hat{G}}_{/\mfS_{\infty}}$ is abelian.
More precisely,
if $f\colon \hat{\mfM}\to \hat{\mfM}'$ is a morphism in 
$\Min^{r,\hat{G}}_{/\mfS_{\infty}}$, then

\vspace{-2mm}
\begin{itemize}
\item[$(1)$] if we denote the kernel of $f$ as a morphism of $\vphi$-modules by $\mrm{ker}(f)$, 
then $\mrm{ker}(f)$ is an object of $\mrm{Mod}^r_{/\mfS_{\infty}}$
and has a canonical structure of a $(\vphi,\hat{G})$-module of height $r$.
If we denote it by $\wh{\mrm{ker}(f)}$, 
then $\Min^r(\wh{\mrm{ker}(f)})$ is the kernel of $f$ in the abelian category $\Min^{r,\hat{G}}_{/\mfS_{\infty}}$;

\vspace{-2mm}
\item[$(2)$] if we denote the cokernel of $f$ as a morphism of $\vphi$-modules by $\mrm{coker}(f)$,
then $\frac{\mrm{coker}(f)}{u{\rm \mathchar`-tor}}$ is an object of $\mrm{Mod}^r_{/\mfS_{\infty}}$
and has a canonical structure of a $(\vphi,\hat{G})$-module of height $r$.
If we denote it by $\wh{\frac{\mrm{coker}(f)}{u{\rm \mathchar`-tor}}}$, then
it is minimal
and is the cokernel of $f$ in  the abelian category
$\Min^{r,\hat{G}}_{/\mfS_{\infty}}$;

\vspace{-2mm}
\item[$(3)$] if we denote the image $($resp.\ coimage$)$ of $f$ as a morphism of $\vphi$-modules 
by $\mrm{im}(f)$ $($resp.\ $\mrm{coim}(f))$, 
then $\mrm{im}(f)$ $($resp.\ $\mrm{coim}(f))$ 
is an object of $\mrm{Mod}^r_{/\mfS_{\infty}}$
and has a canonical structure of a $(\vphi,\hat{G})$-module of height $r$.
If we denote it by $\wh{\mrm{im}(f)}$ $($resp.\ $\wh{\mrm{coim}(f)})$, 
then it is minimal
and is the image $($resp.\ coimage$)$ of $f$ in the abelian category $\Min^{r,\hat{G}}_{/\mfS_{\infty}}$.
\end{itemize}
\end{theorem}

\begin{proof}
(1) Since the functor $\Min^r$ is right adjoint (Proposition \ref{property2} (4)),
we see the desired result.

\noindent
(2) Put $C=\frac{\mrm{coker}(f)}{u{\rm \mathchar`-tor}}$.
Recall that $C$ is an object of $\mrm{Mod}^r_{/\mfS_{\infty}}$
and has a canonical structure of a $(\vphi,\hat{G})$-module of height $r$
(Proposition \ref{cokernel}).
If we denote by $g$ a natural projection
$\hat{\mfM}'\to \hat{C}$,
by Proposition \ref{property} (3), we have
\[
\hat{C}=g(\hat{\mfM}')=g(\Min^r(\hat{\mfM}'))
=\Min^r(g(\hat{\mfM}'))=\Min^r(\hat{C})
\]
and thus $\hat{C}$ is minimal.

\noindent
(3) Let $g\colon \hat{C}\to \hat{\mfM}'$ be as in the proof of (2).
By (1) and (2),
we see that 
the image of $f$ in the category $\Min^{r,\hat{G}}_{/\mfS_{\infty}}$
is $\Min^r(\wh{\mrm{ker}(g)})$.
Let $\mfM_g$ be the underlying Kisin module of  $\Min^r(\wh{\mrm{ker}(g)})$.
Then $\mfM_g$ is an inverse image of 
the $u$-torsion part of $\mrm{coler}(f)$ with respect to
a natural projection $\mfM'\to \mrm{coler}(f)$.  
Since $\mfM_g$ is finitely generated as a $\mfS$-module,
there exists a positive integer $N$ such that 
$u^N\mfM_g\subset \mrm{im}(f)$.
Hence we obtain the equation $\mfM_g[1/u]=\mrm{im}(f)[1/u]$.
Consequently, by Proposition \ref{property2} (3), we have
\[
\Min^r(\wh{\mrm{ker}(g)})=\Min^r(\hat{\mfM}_g)=\Min^r(f(\hat{\mfM}))
=f(\Min^r(\hat{\mfM}))=f(\hat{\mfM})=\wh{\mrm{im}(f)}
\] 
and thus $\wh{\mrm{im}(f)}$ is minimal.
The proof for coimage is similar and hence we omit it.
\end{proof}

Proofs for the following three results are similar to those of the maximal case.

\begin{lemma}
\label{exff2}
If $\alpha\colon \hat{\mfM}'\to \hat{\mfM}$ and $\beta\colon \hat{\mfM}\to \hat{\mfM}''$
two morphisms in $\Min^{r,\hat{G}}_{/\mfS_{\infty}}$ such that 
$\beta\circ \alpha=0$.
the sequence $0\to \hat{\mfM}'\to \hat{\mfM}\to \hat{\mfM}''\to 0$ is exact 
in $($the abelian category$)$
$\Min^{r,\hat{G}}_{/\mfS_{\infty}}$ if and only if 
$0\to \wh{\mfM'[1/u]}\to \wh{\mfM[1/u]}\to \wh{\mfM''[1/u]}\to 0$ 
is exact in $\mbf{\Phi M}^{\hat{G}}_{/\cO_{\infty}}$.  
Furthermore, the functor 
\[
\Min^{r,\hat{G}}_{/\mfS_{\infty}}\to 
\mbf{\Phi M}^{\hat{G}}_{/\cO_{\infty}},\quad \hat{\mfM}\mapsto \wh{\mfM[1/u]}
\]
is fully faithful.
\end{lemma}

\begin{corollary}
\label{exf2}
The functor $\hat{T}$ defined on $\Min^{r,\hat{G}}_{/\mfS_{\infty}}$ is exact and fully faithful,
and its essential image is stable under taking a subquotient.
\end{corollary}

\begin{corollary}
\label{maxex2}
The functor $\Min^r\colon \mrm{Mod}^{r,\hat{G}}_{/\mfS_{\infty}}\to \Min^{r,\hat{G}}_{/\mfS_{\infty}}$
is exact.
\end{corollary}

Put $e=[K:K_0]$, the absolute ramification index of $K$.
If $er<p-1$, then $F^r_{\mfS}(\hat{M})$ contains at most one element (cf.\ \cite{CL1}, Remark just after Corollary 3.2.6)
and hence all torsion $(\vphi,\hat{G})$-modules of height $r$ are automatically maximal and minimal. Therefore, we obtain
\begin{corollary}
\label{er}
Suppose $er<p-1$. Then 
$\Max^{r,\hat{G}}_{/\mfS_{\infty}}=\mrm{Mod}^{r,\hat{G}}_{/\mfS_{\infty}}
=\Min^{r,\hat{G}}_{/\mfS_{\infty}}$.
In particular,
the category $\mrm{Mod}^{r,\hat{G}}_{/\mfS_{\infty}}$ is abelian 
and the functor
$\hat{T}\colon \mrm{Mod}^{r,\hat{G}}_{/\mfS_{\infty}}\to \mrm{Rep}_{\mrm{tor}}(G)$
is exact and fully faithful, 
and its essential image is stable under taking a subquotient. 
\end{corollary}

\begin{remark}
Similar to Remark \ref{gen},
all results in this subsection hold even if we replace ``$(\vphi,\hat{G})$-modules'' with 
``weak $(\vphi,\hat{G})$-modules''. 
\end{remark}



\subsection{Some remarks}

First the reader should be careful that
{\it there are no new results in this subsection}.

\subsubsection{Connection with a lifting problem}

Let $r\in \{0,1,2,\dots, \infty\}$.
Let  $\mrm{Rep}^{\mrm{st},r}_{\mrm{fr}}(G)$ be the category of lattices 
inside semi-stable $p$-adic representations with Hodge-Tate weights in $[0,r]$.
Let  
$\mrm{Rep}^{\mrm{st},r}_{\mrm{tor}}(G)$ 
be the category of torsion $\mbb{Z}_p$-representations $T$ 
such that there exists lattices 
$\Lambda_1, \Lambda_2\in \mrm{Rep}^{\mrm{st},r}_{\mrm{fr}}(G)$  satisfying 
$\Lambda_1\subset \Lambda_2$ and $T\simeq \Lambda_2/\Lambda_1$.
The pair $\Lambda_1\subset \Lambda_2$ is called a {\it lift of $T$}.
We are interested in the following question:

\begin{question}
\label{lift}
For any $T\in \mrm{Rep}_{\mrm{tor}}(G)$, does there exists an integer $r\ge 0$
such that $T\in \mrm{Rep}^{\mrm{st},r}_{\mrm{tor}}(G)$?
\end{question}

\noindent
If $T$ is a tamely ramified $\mbb{F}_p$-representation, 
then Caruso and Liu proved that 
the question has an affirmative answer 
(cf.\ \cite{CL2}, Theorem 5.7).
If we fix the choice of $r<\infty$, they also proved that 
Question \ref{lift} has a non-affirmative answer,
which follows from a result on ramification bounds of torsion representations
(cf.\ \cite{CL2}, Theorem 5.4).

We connect  Question \ref{lift} to our results in this paper. 
Recall that $\mrm{Rep}^{\hat{G}}_{\mrm{tor}}(G)$
is the essential image of 
$\hat{T}\colon \mrm{Mod}^{\infty,\hat{G}}_{/\mfS_{\infty}}
\to \mrm{Rep}_{\mrm{tor}}(G)$,
which is  an abelian full subcategory
of $\mrm{Rep}_{\mrm{tor}}(G)$.
For simplicity, put
$\mrm{Rep}^{\mrm{st}}_{\mrm{tor}}(G)=\mrm{Rep}^{\mrm{st},\infty}_{\mrm{tor}}(G)$. 
Then the inclusions
\[
\mrm{Rep}^{\mrm{st}}_{\mrm{tor}}(G)\subset \mrm{Rep}^{\hat{G}}_{\mrm{tor}}(G) 
\subset \mrm{Rep}_{\mrm{tor}}(G)
\]
are known (cf.\ \cite{CL2}, Theorem 3.1.3). 
Thus 
Question \ref{lift} has an affirmative answer
if and only if  
$\mrm{Rep}^{\mrm{st}}_{\mrm{tor}}(G)=\mrm{Rep}^{\hat{G}}_{\mrm{tor}}(G)$ and 
$\mrm{Rep}^{\hat{G}}_{\mrm{tor}}(G)=\mrm{Rep}_{\mrm{tor}}(G)$.
On the other hand, we have seen the following commutative diagram  between categories:
\begin{center}
$\displaystyle \xymatrix{
\mrm{Mod}^{\infty,\hat{G}}_{/\mfS_{\infty}} \ar^{\Max^{\infty}}[r] \ar_{\mrm{forgetful}}[d]&
\Max^{\infty,\hat{G}}_{/\mfS_{\infty}} \ar@{^{(}->}[r] \ar_{\hat{T} }[r] & 
\mrm{Rep}_{\mrm{tor}}(G) \ar^{\mrm{restriction}}[d]\\
\mrm{Mod}^{\infty}_{/\mfS_{\infty}} \ar^{\Max^{\infty}}[r]&
\Max^{\infty}_{/\mfS_{\infty}}  \ar^{\sim\quad }_{T_{\mfS}\quad }[r] & 
\mrm{Rep}_{\mrm{tor}}(G_{\infty}). 
}$
\end{center}
Here, the equivalence between categories $\Max^{\infty}_{/\mfS_{\infty}}$ and $\mrm{Rep}_{\mrm{tor}}(G_{\infty})$
in the above diagram is proved in Proposition 5.6 of \cite{CL2}. 
Since the essential image of 
$\hat{T}\colon \Max^{\infty,\hat{G}}_{/\mfS_{\infty}}
\hookrightarrow \mrm{Rep}_{\mrm{tor}}(G)$
is $\mrm{Rep}^{\hat{G}}_{\mrm{tor}}(G)$,
it seems natural to suggest 
\begin{question}
Is the functor 
$\hat{T}\colon \Max^{\infty,\hat{G}}_{/\mfS_{\infty}} \hookrightarrow \mrm{Rep}_{\mrm{tor}}(G)$
essentially surjective, that is, an equivalence of categories?
This is equivalent to say that, 
for any $\hat{M}\in \mbf{\Phi M}^{\hat{G}}_{/\cO_{\infty}}$,
does there exist a sub $(\vphi,\hat{G})$-module $\hat{\mfM}$,
of finite height, of $\hat{M}$
such that $\mfM[1/u]=M$? 
\end{question}
\noindent
If this has an affirmative answer, then we obtain
$\mrm{Rep}^{\hat{G}}_{\mrm{tor}}(G) 
=\mrm{Rep}_{\mrm{tor}}(G)$.
In particular,
we obtain an equivalence of abelian categories
$\Max^{\infty,\hat{G}}_{/\mfS_{\infty}}
\simeq \mrm{Rep}_{\mrm{tor}}(G)$, 
which implies that 
maximal objects of torsion $(\vphi,\hat{G})$-modules completely 
classify torsion $p$-adic representations of $G$.
On the other hands, we ask following questions:
\begin{question}
\label{reso}
Does any torsion $(\vphi,\hat{G})$-module have a resolution of free $(\vphi,\hat{G})$-modules?
\end{question}
\begin{question}
\label{clex}
Is the category $\mrm{Rep}^{\hat{G}}_{\mrm{tor}}(G)$ closed under extensions in $\mrm{Rep}_{\mrm{tor}}(G)$?
\end{question}
\noindent
Theorem \ref{MThm2} might be related with Question \ref{clex}.
If one of these questions has an affirmative answer, then we obtain
$\mrm{Rep}^{\mrm{st}}_{\mrm{tor}}(G) 
=\mrm{Rep}^{\hat{G}}_{\mrm{tor}}(G)$.

\subsubsection{Connection with torsion Breuil modules}

If we can prove the explicit relation between the categories of torsion Breuil modules 
and the category of torsion $(\vphi,\hat{G})$-modules, then our main result in this paper will give a partial answer 
of Question 2 of \cite{CL1}.

\end{document}